\documentclass[12pt,final]{amsart}
\usepackage{fancybox}
\usepackage[style=alphabetic,url=false, isbn=false, doi=false,maxbibnames=99]{biblatex}
\usepackage[margin=1in]{geometry}
\usepackage{amsmath}
\usepackage{amsthm}
\usepackage{amssymb}
\usepackage{ytableau,nicefrac,xspace}
\usepackage[capitalize]{cleveref}
\crefformat{equation}{(#2#1#3)}
\crefrangeformat{equation}{(#3#1#4--#5#2#6)}
\usepackage{stmaryrd}
\usepackage{xcolor}
\usepackage{quiver}
\usepackage[mathscr]{euscript}
\usepackage{makecell}
\usepackage{tikz,tikz-cd}
\usetikzlibrary{decorations.pathreplacing,calligraphy, calc}

\newtheorem{ithm}{Theorem}

\numberwithin{equation}{section}

\newtheorem{thm}{Theorem}[section]

\newtheorem{defn}[thm]{Definition}
\newtheorem{lem}[thm]{Lemma}

\newtheorem{rmk}[thm]{Remark}
\newtheorem{cor}[thm]{Corollary}

\theoremstyle{remark}

\Crefname{thm}{Theorem}{Theorems}

\newtheorem{example}[thm]{Example}

\newcommand{\Sym}{\mathfrak{S}}

\newcommand{\End}{\operatorname{End}}
\newcommand{\Ext}{\operatorname{Ext}}
\newcommand{\Ind}[2]{\operatorname{Ind}_{#1}^{#2}}

\newcommand{\Res}[2]{\operatorname{Res}_{#2}^{#1}}

\newcommand{\Hom}{\operatorname{Hom}}
\newcommand{\mmod}{\operatorname{-mod}}
\newcommand{\K}{\Bbbk}
\newcommand{\Z}{\mathbb{Z}}

\newcommand{\C}{\mathbb{C}}
\newcommand{\Q}{\mathbb{Q}}

\newcommand{\rola}{X}
\newcommand{\al}{\alpha}
\newcommand{\Bi}{\mathbf{i}}
\usetikzlibrary{decorations.pathreplacing,backgrounds,decorations.markings,shapes.geometric,decorations.pathmorphing}
\tikzset{wei/.style={draw=red,double=red!40!white,double distance=1.5pt,thin}}

\newcommand{\Bj}{\mathbf{j}}

\newcommand{\Br}{\mathbf{r}}
\newcommand{\bell}{\boldsymbol{\ell}}
\newcommand{\Ba}{\mathbf{a}}
\newcommand{\oned}{1_{\Delta}}

\newcommand{\eE}{\EuScript{E}}
\newcommand{\eF}{\EuScript{F}}
\newcommand{\Ccat}{\mathcal{C}}
\newcommand{\CR}[2]{\Theta'_{#1,#2}}
\newcommand{\tw}[2]{\Xi_{#1,#2}}
\newcommand{\Ldot}[1]{\bullet_{\Lambda}}
\newcommand{\bR}[4]{\mathcal{\bar R}\left(\genfrac{}{}{0pt}{}{#1}{#2}, #3,#4\right)}
\newcommand{\Lotimes}{\overset{L}{\otimes}}
\newcommand{\strand}{\draw[thick]}
\newcommand{\redstr}{\draw[wei]}
\newcommand{\point}[1]{\filldraw[black] {#1} circle (3pt);}
\newcommand{\lab}{\node[anchor=north] at }
\newcommand{\casebrace}{\draw[thick] [decorate,
    decoration = {calligraphic brace, mirror, amplitude=5pt}]}

\newcommand{\Rg}[1]{\mathsf{R}(\boldsymbol{#1})}
\newcommand{\Cg}[1]{\mathsf{C}(\boldsymbol{#1})}
\newcommand{\PRz}{\chi_0}

\newcommand{\idm}[3]{\draw[thick] (#1,#2) node [below] {\tiny{$#3$}} -- (#1, #2+2);
}
\newcommand{\didm}[4]{\draw[thick] (#1,#2) node [below] {\tiny{$#4$}} -- (#1, #2+2) node [above] {\tiny{$#3$}};}

\newcommand{\rect}[4]{\draw[thick] (#1,#2) -- (#1+#3, #2); \draw[thick] (#1,#2)--(#1, #2+#4);
\draw[thick](#1+#3, #2) -- (#1+#3, #2+#4);
\draw[thick](#1, #2+#4)--(#1+#3,#2+#4);
}

\newcommand{\permd}[5]{\draw[thick] (#1, #2) node [below] {\tiny{$#4$}} -- (#1+#3, #2+2) node [above] {\tiny{$#5$}};}

\newcommand{\bigon}[4]{\draw[thick] (#1,#2) .. controls (#1+1,#2+1) .. (#1,#2+2);
\draw[thick] (#1+1,#2) .. controls (#1,#2+1) .. (#1+1,#2+2);
\node[anchor=north] at (#1,#2) {\tiny{$#3$}};
\node[anchor=north] at (#1+1,#2) {\tiny{$#4$}};}
\newcommand{\X}[4]{\draw[thick] (#1, #2) -- (#1+2, #2+2);
\draw[thick] (#1+2, #2) -- (#1, #2+2);
\node[anchor=north] at (#1, #2) {\tiny{$#3$}};
\node[anchor=north] at (#1+2, #2) {\tiny{$#4$}};
}
\newcommand{\Xnw}[4]{\draw[thick] (#1, #2) -- (#1+2, #2+2);
\draw[thick] (#1+2, #2) -- (#1, #2+2);
\node[anchor=north] at (#1, #2) {\tiny{$#3$}};
\node[anchor=north] at (#1+2, #2) {\tiny{$#4$}};
\filldraw[black] {(#1+0.5,#2+1.5)} circle (4pt);}

\newcommand{\Xne}[4]{\draw[thick] (#1, #2) -- (#1+2, #2+2);
\draw[thick] (#1+2, #2) -- (#1, #2+2);
\node[anchor=north] at (#1, #2) {\tiny{$#3$}};
\node[anchor=north] at (#1+2, #2) {\tiny{$#4$}};
\filldraw[black] {(#1+1.5,#2+1.5)} circle (4pt);}

\newcommand{\Xsw}[4]{\draw[thick] (#1, #2) -- (#1+2, #2+2);
\draw[thick] (#1+2, #2) -- (#1, #2+2);
\node[anchor=north] at (#1, #2) {\tiny{$#3$}};
\node[anchor=north] at (#1+2, #2) {\tiny{$#4$}};
\filldraw[black] {(#1+0.5,#2+0.5)} circle (4pt);}

\newcommand{\Xse}[4]{\draw[thick] (#1, #2) -- (#1+2, #2+2);
\draw[thick] (#1+2, #2) -- (#1, #2+2);
\node[anchor=north] at (#1, #2) {\tiny{$#3$}};
\node[anchor=north] at (#1+2, #2) {\tiny{$#4$}};
\filldraw[black] {(#1+1.5,#2+0.5)} circle (4pt);}

\newcommand{\nXse}[4]{\draw[thick] (#1, #2) -- (#1+1, #2+2);
\draw[thick] (#1+1, #2) -- (#1, #2+2);
\node[anchor=north] at (#1, #2) {\tiny{$#3$}};
\node[anchor=north] at (#1+1, #2) {\tiny{$#4$}};
\filldraw[black] {(#1+.75,#2+0.5)} circle (4pt);}

\newcommand{\lbr}[2]{\draw[thick] (#1,#2) .. controls (#1-0.25, #2+0.5) and (#1-0.25, #2+1.5).. (#1, #2+2) ;}
\newcommand{\rbr}[2]{\draw[thick] (#1,#2) .. controls (#1+0.25, #2+0.5) and (#1+0.25, #2+1.5).. (#1, #2+2) ;}

\newcommand{\de}[2]{\node at (#1,#2+1) {\small{$=$}};}
\newcommand{\dm}[2]{\node at (#1,#2+1.08) {\small{$-$}};}

\newcommand{\dcf}[3]{\node at (#1, #2+1) {\small{$#3$}};}

\newcommand{\dbcf}[3]{\node [anchor=north] at (#1, #2) {\tiny{$#3$}};}
\newcommand{\dtcf}[3]{\node [anchor=south] at (#1, #2+2) {\tiny{$#3$}};}

\newcommand{\lpull}[4]{\draw[thick] (#1,#2) node [anchor=north] {\tiny{$#4$}} .. controls (#1-#3-1.25,#2+1) .. (#1,#2+2);}
\newcommand{\rpull}[4]{\draw[thick] (#1,#2) node [anchor=north] {\tiny{$#4$}} .. controls (#1+#3+1.25,#2+1) .. (#1,#2+2);}

\newcommand{\red}[2]{\draw[wei](#1,#2) -- (#1, #2+2);}

\newcommand{\R}[3]{\mathcal{R}\left(#1,#2,#3\right)}
\newcommand{\Rk}[4]{\mathcal{R}^{#4}\left(#1,#2,#3\right)}
\newcommand{\Rz}[2]{\mathcal{R}_0\left(#1,#2\right)}
\newcommand{\ideal}[1]{\mathcal{I}(#1)}
\newcommand{\slehat}{\widehat{\mathfrak{sl}}_e}
\newcommand{\PR}{\chi}
 \newcommand{\bcomments}[1]{
}
        \newcommand{\dcomments}[1]{
}
    \newcommand{\sym}{\mathfrak{S}}

\newcommand{\resfn}{\bf{b}}
\newcommand{\affweyl}{\widehat{W}}
\usepackage[nomargin,inline,draft]{fixme}

\newcommand{\Zdaff}{\mathsf{Z}_d^{\operatorname{aff}}}
\makeatletter
\renewcommand*\FXLayoutInline[3]{{\@fxuseface{inline}
  \ignorespaces\noindent \ovalbox{\hspace{.01\textwidth} \begin{minipage}{.95\textwidth}
  	#3 \fxnotename{#1}: #2
  \end{minipage}\hspace{.01\textwidth}}}
  \newline}
\makeatother

\newcommand{\aA}{\mathsf{a}}

\tikzset{wei/.style={draw=red,double=red!40!white,double distance=1.5pt,thin}}

\title{Steadied Khovanov--Lauda--Rouquier algebras and local models for blocks}
\date{}
\author{Dinushi Munasinghe} 
\address[D.M.]{Department of Mathematics, National and Kapodistrian University of Athens, Panepistimioupolis, 157 84 Athens, Greece. }
\email{dinushi.munasinghe@mail.utoronto.ca}

\author{Ben Webster}
\address[B.W.]{Department of Pure Mathematics, University of Waterloo \& Perimeter Institute for Theoretical Physics,
Waterloo, ON, Canada}
\email{ben.webster@uwaterloo.ca}
\begin{document}
\begin{abstract}
	It's known that many different blocks of $\mathbb{F}_p\Sym_n$ for different values of $n$ are equivalent as categories, though the corresponding block algebras are almost never isomorphic.   Thus, it is a challenging problem to give one particularly nice representative of this Morita equivalence class of algebras. This has been accomplished for the case of RoCK blocks through work of Chuang--Kessar, Turner, and Evseev--Kleshchev.  In this paper, we give a new perspective on this problem, applying not just to RoCK blocks of $\Sym_n$, but also to {\it all} blocks of Ariki--Koike algebras. We do this by considering steadied quotients of KLRW algebras: these algebras are a natural generalization of cyclotomic quotients, already related to $\Sym_n$ and Ariki--Koike algebras in work of Brundan--Kleshchev. These algebras are defined by ``tilting'' the cyclotomic relations so that we kill the two-sided ideal defined by certain configurations on the left and right sides of our diagrams.  We show a Morita equivalence between these algebras and blocks of Ariki-Koike algebras generalizing the work discussed above.  
\end{abstract}

\maketitle

\section{Introduction}

A central principle in modular representation theory is that the structure of a block is controlled in a surprisingly strong manner by its defect group.  In the representation theory of symmetric groups, blocks correspond to the weight spaces for a categorical representation of the affine Lie algebra $\mathfrak{\widehat{sl}}_p$, so this principle becomes the understanding that many features of a block only depend on which orbit of the Weyl group this weight lies in.  
In particular, a landmark result of Chuang and Rouquier \cite{CR} shows that if blocks of $\mathbb{F}_p\sym_n$ for all $n\geq 0$ have the same defect group, they are derived equivalent, with the derived equivalence lifting the action of an element of the affine Weyl group.  Their proof only uses the structure of a categorical action, and thus applies in many similar situations, such as the representation theory of Ariki--Koike algebras, cyclotomic $q$-Schur algebras, categories $\mathcal{O}$ for Cherednik algebras, etc.  

However, the similarities between these blocks are hard to observe directly.  In particular, Chuang and Rouquier's derived equivalences are often (in fact, almost always in some sense) $t$-exact and thus arise from Morita equivalences; these are often called {\it Scopes equivalences}, since they were first discovered in the case of $\mathbb{F}_p\sym_n$ by Scopes \cite{scopesCartanMatrices1991}.  However, the corresponding algebras have wildly different dimensions--since the value of $n$ will grow quickly as we move away from the dominant Weyl chamber, many of these blocks are most naturally realized by very high-dimensional algebras.  This problem is most manifest in the cases of RoCK blocks, which all lie in a single Morita equivalence class containing ``most'' blocks in some sense.  

In the case of RoCK blocks of $\mathbb{F}_p\sym_n$, a conjecture of Turner \cite{Turner}, building on the work of Chuang and Kessar \cite{CK} for blocks of $p$-weight $<p$ and proven by Evseev--Kleshchev \cite{EK,EK2} in general, gives a single representative of the Morita equivalence class containing RoCK blocks, sometimes called a ``local model'' for these blocks.  It's not clear, however, how to generalize these results to arbitrary blocks of $\mathbb{F}_p\Sym_n$ and to Ariki--Koike algebras.

 In this paper, we'll give a systematic way of constructing ``local models'' for all blocks of Ariki--Koike algebras (including degenerate cases such as $\mathbb{F}_p\sym_n$).  That is, we'll give a description of algebras Morita equivalent to blocks of symmetric groups and Ariki--Koike algebras, which for blocks far from the dominant Weyl chamber have dramatically lower dimension.  
 
 This Morita equivalence is a generalization of Brundan and Kleshchev's isomorphism of blocks of symmetric groups and Ariki--Koike algebras $\operatorname{AK}_{m,n}(q,\mathbf{Q})$ to cyclotomic quotients of KLR algebras.  Let us quickly remind the reader of this equivalence; the reader can refer to \cite{BK} for additional details.  The Ariki--Koike algebra $\operatorname{AK}_{m,n}(q,\mathbf{Q})$ is the quotient of the affine Hecke algebra $\mathcal{H}^{\operatorname{aff}}_n(q)$ of rank $n$ by the relation $\prod_{i=1}^m(X_1-Q_i)=0$.  For simplicity of notation, we only discuss the case where $q\in \mu_e$ is a primitive $e$th root of unity and the parameters $\mathbf{Q}$ are $q$-connected.  In this case, the parameters $\mathbf{Q}$ of the Ariki--Koike algebra correspond to a dominant weight $\Lambda$ of $\mathfrak{\widehat{sl}}_e$ of level $m$ as in \cite[(1.2)]{BK}.  A block of an Ariki--Koike algebra corresponds to a residue function $\resfn \colon \mu_e\to \Z_{\geq 0}$, which records the number of boxes in a corresponding Young diagram with a given residue by \cite[Th. 2.11]{LM}.  We'll encode the residue function $\resfn$ in terms of a weight $\alpha=\sum_u \resfn (u)\alpha_u$ for $\mathfrak{\widehat{sl}}_e$; we denote the corresponding block algebra by $\operatorname{AK}^{\Lambda}_{\alpha}$. Given such a residue function, we can also define a corresponding cyclotomic quotient of a KLR algebra $\mathcal{R}_\alpha^\Lambda$ for $\mu_e$ considered as a Dynkin diagram with $u \to qu$ giving adjacencies. The dominant weight $\Lambda$ encodes the relations on the dots on the left-most strand, and $\resfn (u)$ gives the number of strands labelled by each element in $u \in \mu_e.$
 \begin{thm}[\cite{BK}]
 	We have an isomorphism $\operatorname{AK}^{\Lambda}_{\alpha}\cong R^{\Lambda}_{\alpha}$ from the block algebra of the Ariki--Koike algebra to the cyclotomic quotient of the KLR algebra.  
 \end{thm}
 In this paper, we'll show that there are a number of other algebras Morita equivalent to these same blocks, which are obtained by varying the relations of the cyclotomic quotient.  These are the {\bf steadied quotients} $\R{\PR}{\Lambda}{\alpha}$ where $\PR$ is a real-valued linear function on the root lattice.  These algebras are defined in \cite[Def. 2.22]{WebwKLR}, but we develop them much further here.  We recover the cyclotomic quotient if all simple roots have the same sign under $\PR$, but if some have positive and some have negative values, then we obtain a ``tilted'' version of the cyclotomic relations.  The result is always a quotient of a KLRW algebra $\Rz{\Lambda}{\alpha}$, with the kernel depending on which roots have negative values under $\PR$ and which have positive values.
 
 In order to state our main result, we use the standard action of the Weyl group $\affweyl$ of type ${A}^{(1)}_{e-1}$, the Coxeter group generated by reflections $s_0,\dots, s_{e-1}$ subject to the relations for the $A_{e-1}^{(1)}$ Cartan matrix. This has the usual action on the weight and root lattices, which induces an action on the pressure functions $\PR$ (this is discussed in more detail in \cref{sec:Weyl}).  Furthermore, we have an action on the residue functions, inducing the action on the root lattice, defined by $w\Ldot{\Lambda} \alpha=w(\alpha-\Lambda)+\Lambda$.  This is induced by matching $\alpha$ with the weight $\Lambda-\alpha$.

The main theorem of this paper is:
 \begin{ithm}\label{conj}
    The algebras $\R{w^{-1}\cdot\PR}{\Lambda}{\alpha}$ and $\R{ \PR}{\Lambda}{w\Ldot{\Lambda} \alpha}$ are Morita equivalent.
\end{ithm}
The most natural setting to apply this is when $\Lambda-\alpha$ is dominant and $\PR$ is negative on all positive roots.  In this case $\R{ \PR}{\Lambda}{w\Ldot{\Lambda} \alpha}\cong R^{\Lambda}_{w\Ldot{\Lambda} \alpha}$  ranges over the Weyl group orbit of $\Lambda-\alpha$, and gives all the blocks of Arike--Koike algebras related by Chuang--Rouquier equivalences.  In particular, in the case where $\Lambda$ is a fundamental weight, if we take $\alpha=d\delta$ then $R^{\Lambda}_{w\Ldot{\Lambda} \alpha}$ will range over the blocks of $\mathbb{F}_p\sym_m$ for all $m$ which have $p$-weight $d$.  

On the other hand, if $\PR',\PR''$ are in the same Scopes chamber (see \cref{sec:Scopes} or \cite[\S 2.2]{rock} for the definition of Scopes chambers), then we have an isomorphism $ \R{ \PR'}{\Lambda}{\alpha}\cong\R{ \PR''}{\Lambda}{\alpha}$ induced by the quotient maps from $\Rz{\Lambda}{\alpha}$.  In particular, since there are finitely many Scopes chambers, only finitely many isomorphism types of algebras appear as $\R{w^{-1}\cdot\PR}{\Lambda}{\alpha}$ as $w$ varies over the affine Weyl group.  
 
We hope that the study of these algebras can shed light on the different blocks that appear in this way, though we leave the development of this thread to future work.  In particular, we know that these algebras must be cellular and have quasi-hereditary covers, since these properties are Morita invariant, but we do not construct these structures directly in this paper.  

In order to show \cref{conj}, we require three key structures:
\begin{itemize}
	\item The algebra $\Rz{\Lambda}{\alpha}$ carries a standard stratification depending on $\PR$ whose bottom stratum is the quotient $\R{\PR}{\Lambda}{\alpha}$.  The homological properties of standard modules and modules filtered by them are key to the calculations we need to do.
	\item The module categories $\oplus_{\alpha}\R{\PR}{\Lambda}{\alpha}\mmod $ carry categorical actions of $\mathfrak{sl}_2$, which are ``tilted'' versions of the usual categorical actions on $\oplus_{\alpha} R^{\Lambda}_{\alpha}$.  Rather than corresponding to simple roots, these functors correspond to the roots defining facets of the alcove containing $\PR$.  
	\item For any pair of pressures, we have a natural complex of bimodules:  \[\bR{\PR}{\PR'}{\Lambda}{\alpha}=\R{\PR}{\Lambda}{\alpha}\Lotimes_{\mathcal{R}_0(\Lambda, \alpha)}  \R{\PR'}{\Lambda}{\alpha}.\]  We will prove that the tensor product with this bimodule induces a derived equivalence (\cref{th:Xi-equivalence}).  
\end{itemize}
Finally, we prove \cref{conj} by comparing the Chuang--Rouquier equivalences induced by our tilted $\mathfrak{sl}_2$-actions with the equivalences induced by tensor product with $\bR{\PR}{s_{\beta}\PR}{\Lambda}{\alpha}$.  

We conclude by developing one connection that is of particular interest in modular representation theory: if $\chi$ lies in the dominant RoCK Scopes chamber (that is, $\chi(\al_i)\gg 0 \gg \chi(\al_0)$ for all simple roots $\al_i\neq \al_0$), then we show directly that $\R{\PR}{\Lambda}{d\delta}$ is Morita equivalent to the algebra $C_{\rho,d}$ defined in \cite[(5.17)]{EK2}, which is shown to be Morita equivalent in turn to the Turner double in \cite[Lem. 8.27]{EK2}.  Thus,  our results can be viewed as a generalization of the results of \cite{EK2} that covers all blocks and generalizes to arbitrary Ariki--Koike algebras.

It's worth noting that the original definition of the algebras $\R{\PR}{\Lambda}{\al}$ was inspired by the geometry of quiver varieties, with $\chi$ playing the role of a stability condition.  By \cite[Cor. 4.7 \& Th. 5.1]{WebcatO}, when $\K=\C$, the algebra $\R{\PR}{\Lambda}{\al}$ is the self-Ext algebra of a DQ-module on a quiver variety with dimension vectors determined by $\Lambda=\sum w_i\Lambda_i$ and $\al=\sum v_i\al_i$ and stability condition $\chi$.  In particular, this means that the Scopes walls for the $W$-orbit of $\Lambda-\al$ should coincide with the arrangement controlling variation of GIT for the quiver variety.  This was noted in the case corresponding to blocks of $\mathbb{F}_p\Sym_n$ in \cite[Lem. 4.3]{gordonQuiver} and is equivalent to \cite[Prop. 5.32]{WebcatO}.

By a result of Maffei \cite{Maffei}, the quiver variety for the data $w\cdot \PR,{\Lambda}$ and $w\Ldot{\Lambda} \alpha$ is isomorphic to that for $\PR,{\Lambda}$ and $\alpha$.   This isomorphism should induce the Morita equivalence of \cref{conj}.  This argument can't be used to prove any results in the case of positive characteristic (which is the most interesting case for our purposes), but is a key motivation for why we guessed that \cref{conj} would hold.

\section*{Acknowledgements}
The authors would like to thank Chris Bowman, 
Nicolas Jacon, Andrew Mathas, Robert Muth, Liron Speyer and Daniel Tubbenhauer for interesting conversations about this paper.

D.\ M. was supported by the Hellenic Foundation for Research and Innovation
(H.F.R.I.) under the Basic Research Financing (Horizontal support for all
Sciences), National Recovery and Resilience Plan (Greece 2.0), Project
Number: 15659, Project Acronym: SYMATRAL.  This research was supported by Mitacs through the Globalink Research Award program.

B.\ W. was supported by Discovery Grant RGPIN-2024-03760 from the
  Natural Sciences and Engineering Research Council of Canada.
This research was also supported by Perimeter Institute for Theoretical Physics. Research at Perimeter Institute is supported by the Government of Canada through the Department of Innovation, Science and Economic Development and by the Province of Ontario through the Ministry of Research and Innovation.

\section{Definition of steadied quotients}
\label{sec:Definition of steadied quotients}
\subsection{KLRW algebras}

We will be working with a class of quotients of weighted KLR (KLRW) algebras first introduced in \cite{WebwKLR}, but we will restrict ourselves to the case of a single red strand and the trivial weighting. We leave consideration of the more general cases to future work.

As usual, we fix a set $I$ and a Cartan datum on this set.  
That is, the free abelian group
generated by the simple roots $\al_i$ for $i\in I$ carries a symmetric bilinear form $\langle
-,-\rangle$ such that $\langle \al_i,\al_i\rangle \in 2\Z_{>0}$, and $C=\left(c_{ij}=2\frac{\langle\al_i,\al_j
  \rangle}{\langle\al_i,\al_i \rangle}\right)$ is a symmetrizable generalized Cartan matrix.  We'll assume for simplicity that $I$ is finite; see \cite[\S 2]{Webunfurl} for more detail on how to generalize this to the infinite case.  
  
A realization of this Cartan matrix over $\mathbb{C}$ is 
\begin{enumerate}
  \item An
$\mathbb{C}$-vector space $\mathfrak{h}$.
\item Elements $\al_i^{\vee}\in
\mathfrak{h}$ 
\item Elements $\al_i\in \mathfrak{h}^*$ for $i\in I$ such that $\al_i^{\vee}(\al_j)=c_{ij}$.  \end{enumerate}
As usual, we can define a Kac--Moody Lie algebra $\mathfrak{g}$ with Cartan $\mathfrak{h}$ generated by formal symbols $E_i$ and $F_i$ satisfying $[E_i,F_i]=\al_i^{\vee}$ and the Serre relations for the Cartan matrix.

For the remainder of the paper, we assume that this Cartan datum is of finite or affine type.
\begin{enumerate}
    \item If the datum is finite type, we will use the unique realization by the Cartan subalgebra of the corresponding simple Lie algebra.
    \item If the datum is affine type, we will use the unique realization where the simple roots and coroots are both linearly independent sets and both span subspaces of codimension one in $\mathfrak{h}^*$ and $\mathfrak{h}$.
\end{enumerate} 
We'll focus primarily on the case where the Cartan matrix is of type $A^{(1)}_{e-1}$, which corresponds to the affine Lie algebra $\slehat$.  In this case, we can identify $I=\Z/e\Z$ with the quiver structure that adds an arrow $i\to i+1$ for all $i\in \Z/e\Z$.  That is, $\al_i^{\vee}(\al_j)=2\delta_{i,j}-\delta_{i,j+1}-\delta_{i,j-1}$.  In the context of Ariki--Koike algebras for a parameter $q\in \K$ of multiplicative order $e$, we identify $I$ with the roots of unity $\mu_e$ via the map $i\mapsto q^i$.  See \cite[\S 2.1]{rock} for a more detailed discussion of the combinatorics of this case.  The results of \crefrange{sec:Definition of steadied quotients}{sec:The categorical action} hold in all finite and affine cases, but there are some details that we will only discuss in the case of $\slehat$.

The diagrammatic algebras that we will study in this paper are made up of linear combinations of KLRW diagrams \cite[Def. 4.1]{knots}, with relations depending on the Cartan datum and certain polynomials $Q_{ij}(u,v)$ for $i\neq j\in I$. As usual, for $A^{(1)}_{e-1}$ we choose the polynomials 
\[Q_{ij}(u,v)= \begin{cases}
    u-v & i=j+1\neq j-1\\
    v-u& i=j-1\neq j+1\\   
    -(v-u)^2& i=j-1= j+1\\
    1 & \text{otherwise}
\end{cases}\]
\begin{defn}
	A KLRW diagram is a collection of strands that connect the lines $y=0$ and $y=1$ in $\mathbb{R}^2$. One strand is a straight red line connecting $(0,0)$ and $(0,1)$ and the other strands are black; strands are allowed to cross but cannot have triple points or tangencies.
Each black strand carries a label by some $i \in I$ and can be adorned with dots. We identify diagrams that differ by isotopies that preserve these genericity conditions.
\end{defn}

Let $b_i$ denote the number of strands with the label $i$ and let $\alpha= \sum_i b_i \alpha_i$.  Let $\K$ be a commutative ring.
\begin{defn}
    We define the algebra $\mathcal{R}_0(\Lambda, \alpha)$ as the algebra made up of $\K$-linear combinations of the KLRW diagrams with $b_i$ strands labelled $i$ modulo the KLRW local relations \cite[(2.6,4.1-2)]{knots}.  When $I=\Z/e\Z$ with the $A^{(1)}_{e-1}$ Cartan datum, these have the form:
\begin{equation}\label{dotcross}
\begin{tikzpicture}[scale=0.2,baseline=10]
\strand (0,0) -- (4,4);
\node at (0,0) [anchor=north] {$i$};
\node at (4,0) [anchor=north] {$j$};
\node at (7,0) [anchor=north] {$i$};
\node at (11,0) [anchor=north] {$j$};
\node at (14,0) [anchor=north] {$i$};
\node at (18,0) [anchor=north] {$j$};
\node at (21,0) [anchor=north] {$i$};
\node at (25,0) [anchor=north] {$j$};
\strand (0,4) -- (4,0);
\filldraw[black] {(1,3)} circle (8pt);
\lab (5.5,3.25) {$-$};
\strand (7,0) -- (11,4);
\strand (7,4) -- (11,0);
\filldraw[black] {(10,1)} circle (8pt);
\lab (12, 3) {$=$};
\strand (14,0) -- (18,4);
\strand (14,4) -- (18,0);
\filldraw[black] {(15,1)} circle (8pt);
\lab (19.5,3.25) {$-$};
\strand (21,0) -- (25,4);
\strand (21,4) -- (25,0);
\filldraw[black] {(24,3)} circle (8pt);
\lab (26.5,3.25) {$=$};
\casebrace (29,6) -- (29,-2);
\lab (30,5) {$0$};
\lab (35,5) {$i \neq j$};
\strand (30,-1) -- (30,1);
\strand (31,-1) -- (31,1);
\lab (35,1.5) {$i = j$};
\end{tikzpicture}
\end{equation}

\begin{equation}\label{bigon}
\begin{tikzpicture}[scale=0.5,baseline=-30pt]
\bigon{3}{-3}{i}{j}
\lab (5,-1.5) {$=$};
\casebrace (6.5,5) -- (6.5,-9);
\idm{8}{1}{i}
\idm{9}{1}{j}
\idm{8}{-2}{i}
\idm{9}{-2}{j}
\lab (8,5) {$0$};
\lab (19.5,5) {$i = j$};
\lab (20.5,2.5) {$i\neq j, j\pm 1$};
\lab (10.5,-0.5) {$-$};
\idm{12}{-2}{i}
\idm{13}{-2}{j}
\point{(8,-1)};
\point{(13,-1)};
\lab (20,-0.5) {$i=j+1$};
\idm{8}{-5}{i}
\idm{9}{-5}{j}
\lab (10.5,-3.25) {$-$};
\idm{12}{-5}{i}
\idm{13}{-5}{j}
\point{(9,-4)};
\point{(12,-4)};
\lab (20,-3.5) {$i=j-1$};
\lab (7,-6.5) {$-$};
\idm{8}{-8}{i}
\idm{9}{-8}{j}
\lab (10.5,-6.5) {\small{$+2$}};
\idm{12}{-8}{i}
\idm{13}{-8}{j}
\lab (14,-6.5) {$-$};
\idm{15}{-8}{i}
\idm{16}{-8}{j}
\point{(8,-6.75)};
\point{(8,-7.25)};
\point{(12,-7)};
\point{(13,-7)};
\point{(16,-6.75)};
\point{(16,-7.25)};
\lab (21.5, -6.5) {$i = j+1 = j-1$};
\end{tikzpicture}
\end{equation}
\begin{equation}
    \label{triple}
    \begin{tikzpicture}[scale=0.5,baseline=10pt]
      \permd{0}{0}{2}{i}{}
      \permd{2}{0}{-2}{i}{}
      \lpull{1}{0}{0}{j} \dcf{2.5}{0}{=}  \rpull{5.5}{0}{0}{j}  \permd{4.5}{0}{2}{i}{}
      \permd{6.5}{0}{-2}{i}{} \dcf{7.5}{0}{+ h} \lbr{9}{0} \idm{10}{0} {i} \point{(10, 1)} \idm{11}{0}{j} \idm{12}{0}{i} \dbcf{13}{0}{,} \idm{14}{0}{i} \idm{15}{0}{j}
      \point{(15, 1)}
      \idm{16}{0}{i} \dbcf{17}{0}{,} \idm{18}{0}{i} \idm{19}{0}{j} \idm{20}{0}{i}
      \point{(20, 1)}
      \rbr{21}{0} 
    \end{tikzpicture}
\end{equation}
where:
\begin{equation}
    h(u,v,w)=\begin{cases}
        1 & i=j+1\neq j-1\\
        -1 & i=j-1\neq j+1\\
        -u +2v -w & i=j+ 1=j-1\\
                0 & \text{ otherwise}
    \end{cases}
\end{equation}
and for any other choice of labels, the right hand side is zero.   
There is one red strand which is fixed in place. A black strand labelled $i$ (an ``$i$-strand") requires $\langle \Lambda, \alpha_i\rangle$ dots to cross this red strand. That is, we have the local relations:

\begin{equation}\label{cyclotomic}
\begin{tikzpicture}[scale=0.5]
\redstr (0,0) -- (0,4);
\strand (1,0) -- (1,4);
\point{(1,2)};
\node[anchor=south west] at (1,2) {$\langle \Lambda, \alpha_i\rangle$};
\lab (1,0) {$i$};
\node at (4,2) {$=$};
\redstr (5,0) -- (5,4);
\strand (6,0) .. controls (4,2) .. (6,4);
\lab (6,0) {$i$};   
\end{tikzpicture}\qquad 
\begin{tikzpicture}[scale=0.5]
\redstr (1,0) -- (1,4);
\strand (0,0) -- (0,4);
\point{(0,2)};
\node[anchor=south east] at (0,2) {$\langle \Lambda, \alpha_i\rangle$};
\lab (0,0) {$i$};
\node at (2,2) {$=$};
\redstr (4,0) -- (4,4);
\strand (3,0) .. controls (5,2) .. (3,4);
\lab (3,0) {$i$};   
\end{tikzpicture}
\end{equation}

\begin{equation}\label{redcrossing}
\begin{tikzpicture}[scale=0.7]
\red{1}{0} \dm{2}{0} \red{3}{0} \draw[thick] (-0.5,0) -- (1.5,2);
\draw[thick] (1.5,0) -- (-0.5,2);
\draw[thick] (2.5, 0) -- (4.5,2);
\draw[thick] (2.5,2) -- (4.5,0);
\node at (-0.5,0) [anchor=north] {\tiny{$i$}};
\node at (1.5,0) [anchor=north] {\tiny{$i$}};
\node at (2.5,0) [anchor=north] {\tiny{$i$}};
\node at (4.5,0) [anchor=north] {\tiny{$i$}};
\node at (1,0) [anchor=north] {\tiny{$\Lambda$}};
\node at (3,0) [anchor=north] {\tiny{$\Lambda$}};
\de{5}{0}
\node at (7,1) {$\sum \limits_{\substack{a+b\\=\langle \Lambda , \alpha_i^\vee \rangle - 1}}$};
\idm{8.5}{0}{i} \red{9.5}{0} \idm{10.5}{0}{i}
\point{(8.5,1)}
\point{(10.5,1)}
\node at (8.5,1) [anchor = east] {$a$};
\node at (10.5,1) [anchor = west] {$b$};
\end{tikzpicture}
\end{equation}

\end{defn}

For a given horizontal slice of the diagram, there is a corresponding idempotent in the algebra.  We'll usually let $\bell=(\ell_1,\dots, \ell_{k_\ell})$ be the labels on the black strands to the left of the red strand, read left-to-right, and $\Br=(r_1,\dotsc, r_{k_r})$ the labels on the black strands to the right of the red, read {\it right-to-left}; that is, in both cases, we read from outside in toward the red strand.  
We denote this idempotent $e(\bell,\Br)$.
\begin{align}\label{idemp}
e(\bell,\Br)=\begin{tikzpicture}[scale=0.5,baseline=8pt]
    \strand (-1,0) -- (-1,2);
    \strand (-3,0) -- (-3,2);
    \strand (1,0) -- (1,2);
    \strand (3,0) -- (3,2);
    \redstr(0,0) -- (0,2);
    \lab (-3,0) {$\ell_1$};
    \lab (-1,0) {$\ell_{k_\ell}$};
    \lab (3,0) {$r_1$};
    \lab (1,0) {$r_{k_r}$};
    \lab (-2,1) {\dots};
    \lab (2,1) {\dots};
\end{tikzpicture}   
\end{align}
For each such idempotent, we can consider the sums \[\alpha_{\bf{\ell}}=\sum_{m=1}^{k_\ell}\alpha_{\ell_m} \qquad \alpha_{\bf{r}}=\sum_{m=1}^{k_r}\alpha_{r_m} \]

\begin{defn}\label{def:KLR}
A {\bf KLR diagram} is a diagram with only black strands that satisfies the same local constraints as a KLRW diagram.  The {\bf KLR algebra} $R_{\alpha}$ is the quotient of the formal $\K$-linear combinations of KLR diagrams where labels on strands sum to $\alpha$ by the relations \crefrange{dotcross}{triple}.
    
    In this case, we will write $e(\Bi)$ for the idempotent where the labels on strands, read left-to-right, are $(i_1,\dots, i_n)$.  Note that $R_{\alpha}$ is Morita equivalent to $\Rz{0}{\alpha}$ via the inclusion adding a red strand to the right of the diagram, sending $e(\Bi)\mapsto e(\Bi,\emptyset)$. 
\end{defn}
\subsection{Pressures}

To each simple root $\alpha_i$, we assign a \textbf{pressure}
$\PR(\alpha_i) \in \mathbb{R}$, which we extend to a function on the
root lattice of $\widehat{\mathfrak{sl}}_e$ by linearity.  For the
sake of normalization, we require $\chi(\delta) =-1.$ Since we will often want to consider the evaluation of this function on the labels of some group of strands, we let $\alpha_{(i_1,\dots, i_r)}=\al_{i_1}+\cdots +\al_{i_r}$.

\begin{defn}
    An element $\alpha$ of the root lattice is said to have \textbf{positive pressure} if $\PR(\alpha)>0$ and \textbf{negative pressure} if $\PR(\alpha)<0$.
\end{defn}
This name comes from a physical metaphor for how we will sometimes move strands.  
The pressure of a root can be visualized as the force 
on a strand with that label: $\PR(\alpha_i) >0$ means that an $i$-strand will have a tendency to move toward the right (the positive direction), while $\PR(\alpha_i)<0$ indicates that a strand will tend to move to the left. It can be useful to think of strands as colliding and getting stuck together, with the labels $i_1, \dotsc , i_k$ appearing being captured by the sum $\beta=\alpha_{i_1} + \dots + \alpha_{i_k}$ of the corresponding simple roots, so $\PR(\beta)$ encodes the total force on this group of strands (see \cref{rem:pressure}).
Consider the height function $h\colon \rola\to \Z$ defined on simple roots as $h(\al_i)=1$ for all $i$ and extended linearly.
If we imagine each strand as a particle with unit mass, then  $h(\beta)$ is the total mass of this group. 
Thus, in this case, the quotient $\PR(\beta)/h(\beta)$ captures the acceleration of this group of strands by its usual proportionality with force.

\begin{defn}
    The {\bf steadied quotient} $\R{\chi}{\Lambda}{\alpha}$ is the quotient of $\mathcal{R}_0(\Lambda, \alpha)$ by the two-sided ideal $\ideal{\PR}$ generated by the idempotents of the following form:
\begin{align}\label{steadiedidemp}
\begin{tikzpicture}[scale=0.5,baseline=6pt]
    \strand (-2,0) -- (-2,2);
    \strand (-4,0) -- (-4,2);
    \strand (2,0) -- (2,2);
    \strand (4,0) -- (4,2);
    \redstr(0,0) -- (0,2);
    \lab (-4,0) {$\ell_1$};
    \lab (-2,0) {$\ell_i$};
    \lab (4,0) {$r_1$};
    \lab (2,0) {$r_j$};
    \lab (-1,1) {\dots};
    \lab (1,1) {\dots};
    \lab (-3,1) {\dots};
    \lab (3,1) {\dots};
    \node at (5,1) {$=0$};
\end{tikzpicture}   
\end{align}
whenever $\PR(\alpha_{(\ell_1, \dots, \ell_i)})<0$ for some $i$ or $\PR(\alpha_{(r_1, \dots, r_j)})>0$ for some $j$. That is, idempotents become zero when strands are able to ``escape'' to infinity in the direction of their momentum. Note that we only require a subset of the strands, counting from the outside-in, to escape.  
\bcomments{I didn't see how to clearly phrase the commented out part.}
\end{defn}
In particular, if $e(\bell,\Br)$ is non-zero in this quotient,
$\alpha_{\Br}$ must have negative pressure, and
$\alpha_{\bell}$ must have positive pressure.

In the physical metaphor discussed above, if we have $\PR(\alpha_{\ell_1} + \dots + \alpha_{\ell_i})<0$, this means that the center of mass of this group of strands is accelerating leftward and some group of them will escape off to $-\infty$.  Similarly, if $\PR(\alpha_{r_1} + \dots +\alpha_{r_j})>0$, then the center of mass of this group of strands is accelerating rightward, and some group of them will escape off to $\infty$.  Thus, if $e(\bell,\Br)$ is non-zero in this quotient, we must have that in our physical model, no group of strands is pushed to $\pm \infty$.  

In the case when $\PR(\alpha_i) <0$ for all $i \in I$, we recover the algebra denoted $T^{\lambda}$ in \cite{knots}.  This is the same as
the usual cyclotomic KLR algebra by \cite[Thm 4.18]{knots}, which shows that the usual cyclotomic ideal is induced by the relations
\begin{center}
 \begin{tikzpicture}[scale=0.5]
    \strand (0,0) -- (0,2);
    \redstr(1,0) -- (1,2);
    \lab (0,0) {$i$};
    \node at (2,0.5) {$\dots$};
    \node at (3.5,1) {$=0$};
\end{tikzpicture}      
\end{center}
That is, $e((i), \Br)=0$ for any $i \in I$.  As discussed in the introduction, the main theorem of \cite{BK} can identify this with a block of an Ariki-Koike algebra (if $\K$ is a field of characteristic coprime to $e$), or a cyclotomic degenerate affine Hecke algebra (if $e$ is the characteristic of $\K$).

Given a fixed $e>0$, we will often work with the pressure $\PRz(\alpha_i) = -\frac{1}{e}$ for all $i \in I$. 

\begin{example}\label{ex:2a02a1}
    The algebra $\R{\PR}{\Lambda_0}{2\alpha_0+2\alpha_1}$, where $\PR(\alpha_0), \PR(\alpha_1) <0$ and $e=2$. For normalization of $\chi(\delta) = -1$ we take $\chi(\alpha_i) = -\frac12.$ 

We have a total of two black $0$-strands, two black $1$-strands, and one red strand, giving us a total of $30$ possible straight line idempotents. The possible strand labels read left-to-right are
\[ 0011 , \quad 0101, \quad 0110, \quad 1001, \quad 1010, \quad 1100\]
each giving five possible idempotents based on the placement of the red strand.

Idempotents with a black strand on the left of the red strand are zero, so the nonzero idempotents only have black strands on the right. Since a $1$-strand directly to the right of the red strand can be pulled freely across to the left, and since any $0$-strand directly to the right of the red strand with at least one dot can also be pulled across to the left, any such idempotents are also zero. Furthermore, by squaring the idempotent:
\begin{center}
    \begin{tikzpicture}[scale=0.5]
    \lbr{-1}{0}
        \red{0}{0}
        \idm{1}{0}{0}
        \idm{2}{0}{0}
        \rbr{3}{0}{0}
        \node at (3.7,2) {\tiny{$2$}};
        \node at (4,1) {\tiny{$=$}}; 
        \lbr{5}{0} \red{6}{0} 
        \Xsw{7}{0}{0}{0} \node at (10,1) {\tiny{$-$}};
        \red{11}{0} \Xne{12}{0}{0}{0}{0}
        \rbr{15}{0}
         \lbr{16}{0} \red{17}{0} 
        \Xnw{18}{0}{0}{0} \node at (21,1) {\tiny{$-$}};
        \red{22}{0} \Xse{23}{0}{0}{0}{0}
        \rbr{26}{0}
        \node at (4,-3) {\tiny{$=$}};
        
        \red{5}{-4}
        \bigon{6}{-4}{0}{0} 
        \point{(6.3,1.2-4)}
        \point{(6.3, 0.8-4)}
         \node at (8, -3) {\tiny{$-$}};
         \red{9}{-4}
         \bigon{10}{-4}{0}{0}
         \point{(10.25,-3)}
         \point{(10.75,-2.25)}
         \node at (12, -3) {\tiny{$-$}};
         \red{13}{-4}
         \bigon{14}{-4}{0}{0}
         \point{(14.25,-3)}
         \point{(14.75, -3.75)}
         \node at (16, -3) {\tiny{$+$}};
         \red{17}{-4}
         \bigon{18}{-4}{0}{0}
         \point{(18.75, -2.25)}
         \point{(18.75,-3.75)}
         \node at (4,-7) {\tiny{$=$}};
         \node at (5, -7) {$0$};
    \end{tikzpicture}
    \end{center}
    The computation of squaring the idempotent above will prove generally useful. Note that without taking the quotient, we still have:

\begin{lem}\label{lem:id-sq}
For any $e$, if $i$ is such that $\langle \Lambda, \alpha_i\rangle = 1:$
\begin{center}
\begin{tikzpicture}[scale=0.5]
     \lbr{-1}{0}
        \red{0}{0}
        \idm{1}{0}{i}
        \idm{2}{0}{i}
        \rbr{3}{0}{i}
        \node at (3.7,2) {\tiny{$2$}};
        \node at (4,1) {\tiny{$=$}};
    \red{5}{0}
        \bigon{6}{0}{i}{i} 
        \point{(6.3,1.2)}
        \point{(6.3, 0.8)}
         \node at (8, 1) {\tiny{$-$}};
         \red{9}{0}
         \bigon{10}{0}{i}{i}
         \point{(10.25,1)}
         \point{(10.75,0.25)}
         \node at (12, 1) {\tiny{$-$}};
         \red{13}{0}
         \bigon{14}{0}{i}{i}
         \point{(14.25,1)}
         \point{(14.75, 1.75)}
\node at (4,-3) {\tiny{$=$}};
\red{5}{-4} \rpull{6}{-4}{0.2}{i} \lpull{7}{-4}{2}{i}
\point{(5.5,-3.5)}
\node at (8.5, -3) {\tiny{$-$} };
\red{10}{-4} \rpull{11}{-4}{0.2}{i} \lpull{12}{-4}{2}{i}
  \point{(11.125,-3.9)}
\node at (13, -3) {\tiny{$-$} };
\red{15}{-4} \rpull{16}{-4}{0.2}{i} \lpull{17}{-4}{2}{i}
\point{(16.75, -2.075)}
\end{tikzpicture}
\end{center}   
\end{lem}
Thus, any idempotents with labels $0011, 1001, 1010, 1100$ are zero, and the nonzero straight-line idempotents are:
    \begin{center}
    \begin{tikzpicture}[scale=0.5, baseline=10]
\red{0}{0} \idm{1}{0}{0} \idm{2}{0}{1} \idm{3}{0}{0} \idm{4}{0}{1}
    \end{tikzpicture}
   \quad and \qquad 
        \begin{tikzpicture}[scale=0.5, baseline=10]
\red{0}{0} \idm{1}{0}{0} \idm{2}{0}{1} \idm{3}{0}{1} \idm{4}{0}{0} 
    \end{tikzpicture}        
    \end{center}
corresponding to residue sequences of the following tableaux:
\[\begin{ytableau}
    1 & 2 & 3 & 4
\end{ytableau} \quad \text{and} \;\qquad \begin{ytableau}
    1 & 2 &4 \\ 3
\end{ytableau}\]
This recovers the cyclotomic KLR algebra $R^{\Lambda}_{\alpha}.$
\end{example}

The varying pressures are thus a skewed version of the cyclotomic quotient relation, which is just one extreme case; RoCK blocks correspond to cases at the opposite extreme, where the values of $\PR(\beta)$ for all real roots are large compared to $\PR(\delta)=-1$.  More precisely:
\begin{defn}
    We call a pressure {\bf RoCK} for a given choice of $\alpha$ if whenever $\beta+m\delta\leq \alpha$ for some $m\in \Z_{\geq 0}$, the signs of $\PR(\beta)$ and $\PR(\beta+m\delta)$ are the same.  
\end{defn}

\Cref{conj} shows that $\R{\PR}{\Lambda}{\alpha}$ with $\Lambda-\alpha$ dominant and $\PR$ RoCK is Morita equivalent to a RoCK block of the corresponding Ariki--Koike algebra.  We'll expand on this in more detail in \cref{sec:RoCK}.

\subsection{Dependence on alcoves}

Given an element of the root lattice $\beta$, we've already discussed that the quotient $\PR(\beta)/h(\beta)$ will be the acceleration of a group of strands whose labels sum to $\beta$.  These slopes also define a convex order on the set of positive roots by the rule \[ \beta_1 \preceq \beta_2 \text{ if and only if } \frac{\chi(\beta_1)}{h(\beta_1)}\leq \frac{\chi(\beta_2)}{h(\beta_2)}.\]
There is a beautiful theory of cuspidal modules over the KLR algebra $R_{\beta}$ with respect to this order, which many of our constructions will generalize.  See \cite[\S 2]{tingleyMirkovicVilonenPolytopes2016} for the construction of cuspidal modules with respect to this type of order.   In particular, we will often want to consider the semi-cuspidal quotient of $R_{\beta}$.

\begin{defn}
The semi-cuspidal quotient $C_{\beta}$ is the quotient of $R_{\beta}$ by the ideal generated by concatenations $e(\Bi\Bj)$ where $\alpha_{\Bi}\preceq \alpha_{\Bj} $.  
\end{defn}
If $\beta$ is not a multiple of a root, this quotient is trivial.
By \cite[Th. 3.3]{brundanHomologicalProperties2014}, if $\beta$ is a real root, then $C_{\beta}$ is Morita equivalent to $\K[x]$ by its action on the standard module $\Delta(\beta)$.  On the other hand, if $\beta=d\delta$, this quotient is much more complicated; its structure will be relevant in \cref{sec:RoCK}.

\begin{lem}\label{lem:only-roots}
   The ideal $\ideal{\chi}$ is generated by elements of the form \eqref{steadiedidemp} in the case where $\beta=\alpha_{(\ell_1,\dots, \ell_i)}$ is a root that satisfies $\PR(\beta)<0$ or $\beta=\alpha_{(r_1,\dots, r_j)}$ is a root that satisfies $\PR(\beta)>0$. 
   That is, the ideal $\ideal{\PR}$ only depends on the Weyl chamber in which $\PR$ lives. 
\end{lem}
\begin{proof}
Assume for now that we have $\beta=\alpha_{(\ell_1,\dots, \ell_i)}$ and $\PR(\beta)<0$.  Consider the projective module over $R_{\beta}$ defined by the idempotent for $\Bi=(\ell_1,\dots, \ell_i)$.  In order to complete the proof, we need to show that for any indecomposable summand $P$ of this projective, there is a sequence $\Bi'=(\ell_1', \cdots, \ell_i')$ such that
\begin{enumerate}
    \item The module $P$ is also a summand of $R\cdot e(\Bi')$.
     \item For some $p$, the sum $\beta'_p=\al_{(\ell_1', \cdots, \ell_p')}$ is a root that satisfies $\PR(\beta'_p)>0$.  
\end{enumerate}

Of course, any indecomposable summand $P$ has a simple quotient $L$, and we only need to prove that $e(\Bi')L\neq 0$ for some $\Bi'$ as above. This follows from the cuspidal decomposition of $L$ \cite[Thm. 2.4]{tingleyMirkovicVilonenPolytopes2016}:  
the simple $L$ is the unique quotient of an induction $L_1\circ \cdots \circ L_n$, where $L_k$ is a cuspidal simple $R_{\beta_k}$ module, and $\beta_1\preceq \beta_2\preceq \cdots \preceq\beta_n$;  by convexity, we have $\beta_1\preceq \beta$, and $\PR(\beta_1)<0$.  Furthermore, \cite[Cor. 2.12]{tingleyMirkovicVilonenPolytopes2016} shows that for a generic choice of $\PR$, we must have that $\beta_k$ is a root (possibly imaginary).  In particular, any concatenation of idempotents with non-zero image on $L_1,\dots, L_n$ gives a $\Bi'$ with the properties (1-3) above.

The proof for $\PR(\alpha_{(r_1,\dots, r_{j})})>0$ is completely symmetric.  
\end{proof}

\subsection{Actions of the Weyl group}
\label{sec:Weyl}

As discussed in the introduction, we can structure our discussion using the action of the Weyl group of our chosen Cartan datum.  In the case of  $A_{e-1}^{(1)}$, this Coxeter group is generated by reflections $s_0,\dots, s_{e-1}$ subject to the usual Coxeter relations---the reflections $s_{i},s_{i+1}$ generate a copy of $\Sym_3$. 

This group will act on the defining parameters of the steadied quotient algebras by two related, but subtly different, actions.

Given any real root $\beta$, the $\pm\beta$ root spaces generate a copy of $\mathfrak{sl}_2$.  This subalgebra contains a unique line in the Cartan, and there is a unique element $\beta^{\vee}$ of this line such that $\langle
\beta,\beta^{\vee}\rangle=2.$  
\begin{defn}\label{alpha}
    Given $w \in \affweyl$, let $w \cdot \alpha$ be the usual action on weights. For any root $\beta$, we have the reflection 
    \begin{align*}
        s_\beta \cdot \alpha= \alpha -\langle \alpha,\beta^{\vee}\rangle \beta. 
    \end{align*}   
    We let $s_i=s_{\alpha_i}$ for simple roots.
\end{defn}

Using the usual dual action, this induces an action on pressures as well: for $w \in \affweyl$, let $w \cdot \PR$ be the pressure function obtained by precomposing $\PR$ with $w^{-1}$.  In the case of  $A_{e-1}^{(1)}$, the simple reflections act by
    \begin{equation*}
        s_i \cdot \PR (\alpha_j) = \begin{cases}
            -\PR (\alpha_j) & i =j\\
            \PR(\alpha_j) + \PR (\alpha_i) & j = i \pm 1 \\
            \PR (\alpha_j) & \text{otherwise}
        \end{cases}
    \end{equation*}  
    
We will typically want to associate the weight $\Lambda-\alpha$ to a given algebra. In particular, the natural Weyl group action to consider on our data corresponds to the action on these weights.  Thus, we let
\begin{equation} \label{eq:w-dot-action}
w\Ldot{\Lambda}\alpha=w(\alpha-\Lambda)+\Lambda\qquad \qquad s_\beta\Ldot{\Lambda}\alpha=\alpha -\langle \Lambda-\alpha,\beta^{\vee}\rangle \beta	
\end{equation}
We could equivalently write this definition as $\Lambda-w\Ldot{\Lambda}\alpha=w(\Lambda-\alpha)$. 

In the context of Ariki-Koike algebras, we often think of the coefficients $b_i$ as corresponding to the number of boxes of residue $i$ in Young diagrams corresponding to a given box.  In these terms, the integer $\al_i^{\vee}(\Lambda-\alpha)$ can be interpreted as $\#\{ \text{addable }i\text{-boxes}\}-\#\{ \text{removable }i\text{-boxes}\}$.  Thus, the action of $s_i$ can be understood as adding or removing the appropriate number of boxes of residue $i$.

 Let us examine these actions on \cref{ex:2a02a1}. Recall that the generating idempotents there were 
 
 \[ \begin{tikzpicture}[scale=0.5, baseline=10]
     \red{0}{0} \idm{1}{0}{0} \idm{2}{0}{1} \idm{3}{0}{1} \idm{4}{0}{0}
 \end{tikzpicture} \quad \text{and} \quad \begin{tikzpicture}[scale=0.5, baseline=10]
     \red{0}{0} \idm{1}{0}{0} \idm{2}{0}{1} \idm{3}{0}{0} \idm{4}{0}{1}
 \end{tikzpicture} \]

\begin{example}[$\R{s_0\cdot \PRz}{\Lambda_0}{2\alpha_0+2\alpha_1}$]\label{ex:2a02a1-act0:chi} Let $\chi=s_0\cdot \PRz,$ so 
\[ \chi (\alpha_0) = \frac12 \qquad \qquad  \chi (\alpha_1) = -\frac32.\]

We again consider the full range of possible strand labels, each of which corresponds to $5$ different possible idempotents based on the placement of the red strand:
\[ 0011 , \quad 0101, \quad 0110, \quad 1001, \quad 1010, \quad 1100\]

Since $\chi(\alpha_1)<0$ and $\langle \Lambda, \alpha_1\rangle=0$, any such idempotent with left-most black strand labelled $1$, regardless of whether it lies to the left or right of the red strand, must be zero. Thus any idempotent labelled $1001, 1010,$ or $1100$ is zero. Similarly, any idempotent with a $0$-strand to the far right (right of the red strand) is also zero. Since $\chi(\delta) = -1,$ and $\chi(2\alpha_0+\alpha_1)<0$, all idempotents with three or more black strands on the left-hand side are zero. That $1$-strands can move freely across the red strand eliminates a few more options.

For example:
\begin{center}
 \begin{tikzpicture}[scale=0.5]
     \red{2}{0} \idm{0}{0}{0} \idm{1}{0}{0} \idm{3}{0}{1} \idm{4}{0}{1} \dcf{5}{0}{=} \red{8}{0} \idm{6}{0}{0} \idm{7}{0}{0} \idm{10}{0}{1} \lpull{9}{0}{1}{1} \dcf{11}{0}{= \; 0}
\end{tikzpicture}
   
\end{center}

The non-zero idempotents in the resulting algebra are 
\begin{center}
 \begin{tikzpicture}[scale=0.5]\node at (-1.2,1){$e_1=$};
    \red{0}{0} \idm{1}{0}{0} \idm{2}{0}{0} \idm{3}{0}{1} \idm{4}{0}{1} 
    \end{tikzpicture} 
    \qquad  \qquad  
\begin{tikzpicture}[scale=0.5]\node at (6.8,1){$e_2=$};
    \red{8}{0} \idm{9}{0}{0} \idm{10}{0}{1} \idm{11}{0}{0} \idm{12}{0}{1}
    \end{tikzpicture}      
    \qquad   
 \begin{tikzpicture}[scale=0.5]\node at (14.8,1){$e_3=$};
    \red{17}{0} \idm{16}{0}{0} \idm{18}{0}{0} \idm{19}{0}{1} \idm{20}{0}{1} 
\end{tikzpicture}   
\end{center}
This algebra is quite large, but principally because of the dimensions of its simple representations; if $\operatorname{char}(\K)\neq 2$, then it has a 6-dimensional irreducible representation.  
The algebra $e_2\mathcal{R}e_2$ is isomorphic to the wreath product generated by $s,y_1,y_2$ with relations 
\[s^2-1=y_1^2=y_2^2=y_1y_2-y_2y_1=y_1s-sy_2=y_2s-sy_1=0\] 
under the map
\begin{center}
 \begin{tikzpicture}[scale=0.5]\node at (-2.5,1){$s-1\mapsto$};
    \red{0}{0} \X{1}{0}{0}{0} \X{2}{0}{1}{1} 
 \end{tikzpicture} \qquad  \begin{tikzpicture}[scale=0.5]\node at (5.5,1){$y_1\mapsto$};
    \red{8}{0} \idm{9}{0}{0} \point{(10,1)} \idm{10}{0}{1} \idm{11}{0}{0} \idm{12}{0}{1}  \end{tikzpicture} \\
 \begin{tikzpicture}[scale=0.5]\node at (5.5,1){$y_2\mapsto$};
    \red{8}{0} \idm{9}{0}{0} \point{(12,1)} \idm{10}{0}{1} \idm{11}{0}{0} \idm{12}{0}{1}\dm{13}{0}\red{14}{0} \idm{15}{0}{0} \point{(17,1)} \idm{16}{0}{1} \idm{17}{0}{0} \idm{18}{0}{1}
  \end{tikzpicture}     
\end{center}
induced by the homomorphism of \cite[Cor. 6.19]{evseevRoCKBlocks}; this is proven by hand on pages 1416--17 of \cite{evseevRoCKBlocks}. If $\operatorname{char}(\K)\neq 2$, then $\mathcal{R}e_2$ is a projective generator, but if $\operatorname{char}(\K)= 2$, the algebra above is local and $\mathcal{R}e_2$ is indecomposable.  There is another indecomposable projective in this case, given by $\mathcal{R}e_1'$ for the idempotent 
\begin{center}
 \begin{tikzpicture}[scale=0.5]\node at (-2.5,1){$e_1'=$};
    \red{0}{0} \nXse{1}{0}{0}{0} \nXse{3}{0}{1}{1} 
 \end{tikzpicture} 
 \end{center}
\end{example}

\begin{example}[$\R{s_1\cdot \PRz}{\Lambda_0}{2\alpha_0+2\alpha_1}$]\label{ex:2a02a1-act1:chi} 
Let $\chi=s_1\cdot \PRz,$ so 
\[\chi (\alpha_0) = -\frac32 \qquad \qquad  \chi (\alpha_1) = \frac12.\]

Consider the idempotent
\begin{equation*}
     \begin{tikzpicture}[scale=0.5]
     \node at (-1,1){$e'=$};
    \red{1}{0} \idm{0}{0}{1} \idm{2}{0}{1} \idm{3}{0}{0} \idm{4}{0}{0} \dcf{5}{0}{=} \idm{6}{0}{1} \red{7}{0} \lpull{8}{0}{1}{1} \idm{9}{0}{0} \idm{10}{0}{0}
\end{tikzpicture}
\end{equation*}
Applying \cref{lem:id-sq}, we obtain
\begin{center}
    \begin{tikzpicture}[scale=0.5]
        \dcf{0}{0}{e(1,100) = } \idm{2}{0}{1} \red{3}{0} \lpull{4}{0}{1}{1} \rpull{5}{0}{0.25}{0} \lpull{6}{0}{3.125}{0} 
        \point{(4.5,0.35)}
        \dcf{7}{0}{-} \idm{8}{0}{1} \red{9}{0} \lpull{10}{0}{1}{1} \rpull{10.8}{0}{0.25}{0} \lpull{12}{0}{3.125}{0} 
        \point{(10.95,0.08)}
        \dcf{13}{0}{-} \idm{14}{0}{1} \red{15}{0} \lpull{16}{0}{1}{1} \rpull{17}{0}{0.25}{0} \lpull{18}{0}{3.125}{0}
        \point{(17.75,1.925)}
    \end{tikzpicture}
\end{center}
which is zero since $\chi(\alpha_0+2\alpha_1)<0.$

Continuing with arguments similar to \cref{ex:2a02a1-act0:chi}, we obtain non-zero idempotents
\begin{center}
 \begin{tikzpicture}[scale=0.5]
    \node at (-1.2,1){$e_1=$}; \red{0}{0} \idm{1}{0}{1} \idm{2}{0}{0} \idm{3}{0}{1} \idm{4}{0}{0} 
 \end{tikzpicture} \qquad  \begin{tikzpicture}[scale=0.5]
    \node at (6.8,1){$e_2=$}; \red{8}{0} \idm{9}{0}{0} \idm{10}{0}{1} \idm{11}{0}{1} \idm{12}{0}{0}
 \end{tikzpicture} \qquad  \begin{tikzpicture}[scale=0.5]\node at (14.8,1){$e_3=$};
    \red{17}{0} \idm{16}{0}{1} \idm{18}{0}{0} \idm{19}{0}{1} \idm{20}{0}{0} 
\end{tikzpicture}   
\end{center}

By \cref{conj}, we must have a Morita equivalence \[\R{s_1\cdot \PRz}{\Lambda_0}{2\alpha_0+2\alpha_1}\sim \R{\PRz}{\Lambda_0}{2\alpha_0+2\alpha_1}.\]  In fact, this is induced by an isomorphism of algebras \[(e_2+e_3)\R{s_1\cdot \PRz}{\Lambda_0}{2\alpha_0+2\alpha_1}(e_2+e_3) \cong  \R{\PRz}{\Lambda_0}{2\alpha_0+2\alpha_1}\] which sends
\[ \begin{tikzpicture}[scale=0.5,baseline=10]
    \red{0}{0} \idm{-1}{0}{1} \idm{1}{0}{0} \idm{2}{0}{1} \idm{3}{0}{0} 
 \end{tikzpicture}\,\mapsto \, \begin{tikzpicture}[scale=0.5,baseline=10]
\red{0}{0} \idm{1}{0}{0} \idm{2}{0}{1} \idm{3}{0}{0} \idm{4}{0}{1} 
 \end{tikzpicture}\qquad \qquad\begin{tikzpicture}[scale=0.5,baseline=10]
    \red{0}{0} \idm{1}{0}{0} \idm{2}{0}{1} \idm{3}{0}{1} \idm{4}{0}{0} 
 \end{tikzpicture}\,\mapsto \, \begin{tikzpicture}[scale=0.5,baseline=10]
\red{0}{0} \idm{1}{0}{0} \idm{2}{0}{1} \idm{3}{0}{1} \idm{4}{0}{0}  
 \end{tikzpicture}\]

\end{example}

\begin{example}[$\R{s_1s_0\cdot \PRz}{\Lambda_0}{2\alpha_0+2\alpha_1}$]\label{ex:2a02a1-act10:chi} 
Let $\chi=s_1s_0\cdot \PRz$ so, 
\[\chi (\alpha_0) = -\frac52 \qquad \qquad  \chi (\alpha_1) = \frac32.\]

 Note that while $\chi(\alpha_0)$ and $\chi(\alpha_1)$ have the same sign as in \cref{ex:2a02a1-act1:chi}, we now have $\chi(\alpha_0+2\alpha_1)>0$ so that (for example)
\begin{center}
    \begin{tikzpicture}[scale=0.5]
      \red{0}{0} \idm{1}{0}{0} \idm{2}{0}{1} \idm{3}{0}{1} \idm{4}{0}{0}     
    \end{tikzpicture}
\end{center}
is zero. On the other hand, while we can still apply \cref{lem:id-sq} to $e(1,100)$ to pull a $1$-strand and $0$-strand left, this idempotent, which was zero in \cref{ex:2a02a1-act1:chi}, is nonzero due to the change in the pressure function.

Overall, we have the following non-zero idempotents:

\begin{center}
 \begin{tikzpicture}[scale=0.5]
     \node at (-1.2,1){$e_1=$};  \red{0}{0} \idm{1}{0}{1} \idm{2}{0}{1} \idm{3}{0}{0} \idm{4}{0}{0} 
     \end{tikzpicture} \qquad  \begin{tikzpicture}[scale=0.5]
    \node at (5.8,1){$e_2=$}; 
    \red{7}{0} \idm{8}{0}{1} \idm{9}{0}{0} \idm{10}{0}{1} \idm{11}{0}{0}
     \end{tikzpicture} \qquad  \begin{tikzpicture}[scale=0.5]\node at (12.8,1){$e_3=$};
    \red{15}{0} \idm{14}{0}{1} \idm{16}{0}{1} \idm{17}{0}{0} \idm{18}{0}{0} 
\node at (20.8,1){$e_4=$};
    \idm{22}{0}{1} \red{23}{0} \idm{24}{0}{0} \idm{25}{0}{1} \idm{26}{0}{0} 
\node at (28.8,1){$e_5=$};
    \idm{30}{0}{1} \idm{31}{0}{1} \idm{32}{0}{0} \red{33}{0} \idm{34}{0}{0}
\end{tikzpicture}   
\end{center}
While not isomorphic to $\R{s_0\cdot \PRz}{\Lambda_0}{2\alpha_0+2\alpha_1}$, the same calculations as in that case, we have an isomorphism to $e_1\mathcal{R}e_1$ of the same wreath product algebra given by  
\begin{center}
 \begin{tikzpicture}[scale=0.5]\node at (-2.5,1){$s-1\mapsto$};
    \red{0}{0} \X{1}{0}{1}{1} \X{2}{0}{0}{0} 
 \end{tikzpicture} \qquad  \begin{tikzpicture}[scale=0.5]\node at (5.5,1){$y_1\mapsto$};
    \red{8}{0} \idm{9}{0}{1} \point{(9,1)} \idm{10}{0}{0} \idm{11}{0}{1} \idm{12}{0}{0}  \end{tikzpicture} \\
 \begin{tikzpicture}[scale=0.5]\node at (5.5,1){$y_2\mapsto$};
    \red{8}{0} \idm{9}{0}{1} \point{(11,1)} \idm{10}{0}{0} \idm{11}{0}{1} \idm{12}{0}{0}\dm{13}{0}\red{14}{0} \idm{15}{0}{1} \point{(18,1)} \idm{16}{0}{0} \idm{17}{0}{1} \idm{18}{0}{0}
  \end{tikzpicture}     
\end{center}
In fact, \[(e_1+e_2)\R{s_1s_0\cdot \PRz}{\Lambda_0}{2\alpha_0+2\alpha_1}(e_1+e_2)\cong (e_1+e_2)\R{s_0\cdot \PRz}{\Lambda_0}{2\alpha_0+2\alpha_1}(e_1+e_2)\] by reflecting the portion of the diagram left of the red line in a vertical axis---this isomorphism induces a Morita equivalence.  This is expected from \cref{conj}, since $s_0s_1\Ldot{\Lambda_0}(2\alpha_0+2\alpha_1)=s_0\Ldot{\Lambda_0}(2\alpha_0+2\alpha_1)$.  
\end{example}

To observe behaviour for a fixed $p$-weight on a larger scale, let us look at a chart of generating idempotents and idempotents in the ideal as we act on the pressure function (note that some idempotents which are consistently zero have been excluded):

\begin{center}
\begin{tabular}{ c| c c c c c c c c c c c c }
$\alpha=2\alpha_0+\alpha_1$ & 

\resizebox{10mm}{!}{
  \begin{tikzpicture}[scale=0.15]
    \red{0}{0} \idm{1}{0}{0} \idm{2}{0}{0}\idm{3}{0}{1}
  \end{tikzpicture}
  }

 & \resizebox{10mm}{!}{
  \begin{tikzpicture}[scale=0.15]
    \red{0}{0} \idm{1}{0}{0} \idm{2}{0}{1}\idm{3}{0}{0}
  \end{tikzpicture}
  } & \resizebox{10mm}{!}{
  \begin{tikzpicture}[scale=0.15]
    \red{0}{0} \idm{1}{0}{1} \idm{2}{0}{0}\idm{3}{0}{0}
  \end{tikzpicture}
  } & \resizebox{10mm}{!}{
  \begin{tikzpicture}[scale=0.15]
    \red{1}{0} \idm{0}{0}{0} \idm{2}{0}{0}\idm{3}{0}{1}
  \end{tikzpicture}
  } & \resizebox{10mm}{!}{
  \begin{tikzpicture}[scale=0.15]
    \red{1}{0} \idm{0}{0}{0} \idm{2}{0}{1}\idm{3}{0}{0}
  \end{tikzpicture}
  } &  \resizebox{10mm}{!}{
  \begin{tikzpicture}[scale=0.15]
    \red{1}{0} \idm{0}{0}{1} \idm{2}{0}{0}\idm{3}{0}{0}
  \end{tikzpicture}
  } & \resizebox{10mm}{!}{
  \begin{tikzpicture}[scale=0.15]
    \red{2}{0} \idm{0}{0}{0} \idm{1}{0}{0}\idm{3}{0}{1}
  \end{tikzpicture}
  }  \\ \hline
\makecell{$\chi(\alpha_0)=\frac{-1}{2}$\\
$\chi(\alpha_1)=\frac{-1}{2}
$} & 0& \checkmark & 0 &0 &0 &0 &0\\ \hline 
\makecell{$s_0\chi(\alpha_0)=\frac{1}{2}$\\
$s_0\chi(\alpha_1)=\frac{-3}{2}
$} & \checkmark & 0 & 0 & \checkmark & 0 & 0 &0\\  \hline 
\makecell{$s_1\chi(\alpha_0)=\frac{-3}{2}$\\
$s_1\chi(\alpha_1)=\frac{1}{2}$}& 0& \checkmark & 0 &0 &0 &0 &0 \\  \hline 
\makecell{$s_1s_0\chi(\alpha_0)=\frac{-5}{2}$\\
$s_1s_0\chi(\alpha_1)=\frac{3}{2}$} & 0& \checkmark & 0 &0 &0 &0 &0 \\  \hline 
\makecell{$s_0s_1\chi(\alpha_0)=\frac{3}{2}$\\
$s_0s_1\chi(\alpha_1)=\frac{-5}{2}$} & 0& 0& 0& \checkmark & 0 & 0 & \checkmark \\  \hline 
\makecell{$s_0s_1s_0\chi(\alpha_0)=\frac{5}{2}$\\
$s_0s_1s_0\chi(\alpha_1)=\frac{-7}{2}$} & 0& 0& 0& \checkmark & 0 & 0 & \checkmark \\  \hline 
 \makecell{$s_1s_0s_1\chi(\alpha_0)=\frac{-7}{2}$\\
$s_1s_0s_1\chi(\alpha_1)=\frac{5}{2}$} & 0& \checkmark & 0 &0 &0 &0 &0 \\  \hline
 \makecell{$s_0s_1s_0s_1\chi(\alpha_0)=\frac{7}{2}$\\
$s_0s_1s_0s_1\chi(\alpha_1)=\frac{-9}{2}$}  & 0& 0& 0& \checkmark & 0 & 0 & \checkmark \\  \hline
 \makecell{$s_1s_0s_1s_0\chi(\alpha_0)=\frac{-9}{2}$\\
$s_1s_0s_1s_0\chi(\alpha_1)=\frac{7}{2}$} & 0& \checkmark & 0 &0 &0 &0 &0\\  \hline
\end{tabular}
\end{center}
We now do the case $\alpha=2 \delta,$ noting that some idempotents which are consistently zero, and some which can be obtained by freely moving a strand such that $\langle \Lambda, \alpha_i\rangle = 0$ across the red strand to obtain an idempotent already listed have been left out of the chart.
\begin{center}
\begin{tabular}{ c | c c c c c c c c c c c c }
$\alpha=2\alpha_0+2\alpha_1$ & 

\resizebox{10mm}{!}{
  \begin{tikzpicture}[scale=0.15]
    \red{0}{0} \idm{1}{0}{1} \idm{2}{0}{1}\idm{3}{0}{0} \idm{4}{0}{0}
  \end{tikzpicture}
  }

 & \resizebox{10mm}{!}{
  \begin{tikzpicture}[scale=0.15]
   \red{0}{0} \idm{1}{0}{1} \idm{2}{0}{0}\idm{3}{0}{1} \idm{4}{0}{0}
  \end{tikzpicture}
  } & 
  
  \resizebox{10mm}{!}{
  \begin{tikzpicture}[scale=0.15]
    \red{0}{0} \idm{1}{0}{1} \idm{2}{0}{0}\idm{3}{0}{0} \idm{4}{0}{1}
  \end{tikzpicture}
  } & 
  \resizebox{10mm}{!}{
  \begin{tikzpicture}[scale=0.15]
    \red{0}{0} \idm{1}{0}{0} \idm{2}{0}{1} \idm{3}{0}{0} \idm{4}{0}{1}
  \end{tikzpicture}
  } & 
  \resizebox{10mm}{!}{
  \begin{tikzpicture}[scale=0.15]
    \red{0}{0} \idm{1}{0}{0} \idm{2}{0}{0}\idm{3}{0}{1} \idm{4}{0}{1}
  \end{tikzpicture}
  } &  
  \resizebox{10mm}{!}{
  \begin{tikzpicture}[scale=0.15]
    \red{0}{0} \idm{1}{0}{0} \idm{2}{0}{1}\idm{3}{0}{1} \idm{4}{0}{0}
  \end{tikzpicture}
  } & \resizebox{10mm}{!}{
  \begin{tikzpicture}[scale=0.15]
    \red{1}{0} \idm{0}{0}{0} \idm{2}{0}{1}\idm{3}{0}{0} \idm{4}{0}{1}
  \end{tikzpicture}
  } & \resizebox{10mm}{!}{
  \begin{tikzpicture}[scale=0.15]
    \red{1}{0} \idm{0}{0}{0} \idm{2}{0}{0}\idm{3}{0}{1} \idm{4}{0}{1}
  \end{tikzpicture}
  } & \resizebox{10mm}{!}{
  \begin{tikzpicture}[scale=0.15]
    \red{1}{0} \idm{0}{0}{0} \idm{2}{0}{1}\idm{3}{0}{1} \idm{4}{0}{0}
  \end{tikzpicture}
  }

  &  \resizebox{10mm}{!}{
  \begin{tikzpicture}[scale=0.15]
    \red{3}{0} \idm{0}{0}{1} \idm{1}{0}{1}\idm{2}{0}{0} \idm{4}{0}{0}
  \end{tikzpicture}
  }

  \\ \hline
\makecell{$\chi(\alpha_0)=\frac{-1}{2}$\\
$\chi(\alpha_1)=\frac{-1}{2}
$} & 0&0&0& \checkmark & 0 & \checkmark & 0 & 0 & 0 & 0 \\ \hline 
\makecell{$s_0\chi(\alpha_0)=\frac{1}{2}$\\
$s_0\chi(\alpha_1)=\frac{-3}{2}
$} & 0 & 0 & 0 & \checkmark & \checkmark & 0 & 0 & \checkmark & 0 & 0  \\  \hline 
\makecell{$s_1\chi(\alpha_0)=\frac{-3}{2}$\\
$s_1\chi(\alpha_1)=\frac{1}{2}$} & 0 & \checkmark & 0 & 0 & 0 & \checkmark & 0 & 0 & 0 & 0  \\  \hline 
\makecell{$s_1s_0\chi(\alpha_0)=\frac{-5}{2}$\\
$s_1s_0\chi(\alpha_1)=\frac{3}{2}$} & \checkmark & \checkmark & 0 & 0 & 0 & 0 & 0 & 0 & 0 &  \checkmark \\  \hline 
\makecell{$s_0s_1\chi(\alpha_0)=\frac{3}{2}$\\
$s_0s_1\chi(\alpha_1)=\frac{-5}{2}$}  & 0 & 0 & 0 & \checkmark & \checkmark & 0 & 0 & \checkmark & 0 & 0  \\  \hline 
\makecell{$s_0s_1s_0\chi(\alpha_0)=\frac{5}{2}$\\
$s_0s_1s_0\chi(\alpha_1)=\frac{-7}{2}$} & 0 & 0 & 0 & \checkmark & \checkmark & 0 & 0 & \checkmark & 0 & 0 \\  \hline 
 \makecell{$s_1s_0s_1\chi(\alpha_0)=\frac{-7}{2}$\\
$s_1s_0s_1\chi(\alpha_1)=\frac{5}{2}$} & \checkmark & \checkmark & 0 & 0 & 0 & 0 & 0 & 0 & 0 &  \checkmark \\  \hline
 \makecell{$s_0s_1s_0s_1\chi(\alpha_0)=\frac{7}{2}$\\
$s_0s_1s_0s_1\chi(\alpha_1)=\frac{-9}{2}$} & 0 & 0 & 0 & \checkmark & \checkmark & 0 & 0 & \checkmark & 0 & 0  \\  \hline
 \makecell{$s_1s_0s_1s_0\chi(\alpha_0)=\frac{-9}{2}$\\
$s_1s_0s_1s_0\chi(\alpha_1)=\frac{7}{2}$} & \checkmark & \checkmark & 0 & 0 & 0 & 0 &0 & 0 & 0 & \checkmark \\  \hline
\end{tabular}
\end{center}

Looking at the action of the same Weyl group elements on the cyclotomic side, we observe the following:

\begin{example}[$\R{\PRz}{\Lambda_0}{s_0\cdot (2\alpha_0+2\alpha_1)}$]\label{ex:2a02a1-act0:alpha}

Starting with the Young diagrams from \cref{ex:2a02a1} we have one addable $0$-box for the partition $(4)$, and two addable minus one removable for the partition $(3,1)$. In both cases, we add a single $0$-strand. Possible Young tableaux giving idempotent labellings are:
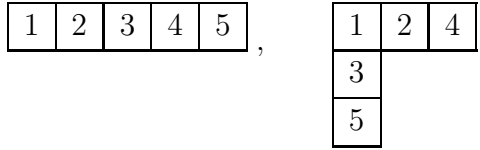
\[\begin{ytableau}
    1 & 2 & 3 & 4 & 5
\end{ytableau} \; , \qquad  \begin{ytableau}
    1 & 2 & 4  \\ 3 \\ 5 
\end{ytableau}\]
giving us the following generating idempotents:

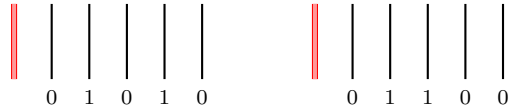
\begin{center}
    \begin{tikzpicture}[scale=0.5]
        \red{0}{0} \idm{1}{0}{0} \idm{2}{0}{1} \idm{3}{0}{0} \idm{4}{0}{1} \idm{5}{0}{0}  
        
        \red{8}{0} \idm{9}{0}{0} \idm{10}{0}{1} \idm{11}{0}{1} \idm{12}{0}{0} \idm{13}{0}{0}

    \end{tikzpicture}
\end{center}

\end{example}

\begin{example}[$\R{\PRz}{\Lambda_0}{s_1\cdot (2\alpha_0+2\alpha_1)}$]\label{ex:2a02a1-act1:alpha}

Looking at the partition $(4)$, for example, we see that although we have one addable box of residue $1$, there is also a removable one, so overall $s_1$ adds no strands to the algebra.

\end{example}

\begin{example}[$\R{\PRz}{\Lambda_0}{s_1s_0\cdot (2\alpha_0+2\alpha_1)}$]\label{ex:2a02a1-act10:alpha}
We add two $1$-strands to \cref{ex:2a02a1-act0:alpha}. Possible Young tableaux giving idempotent labellings are:
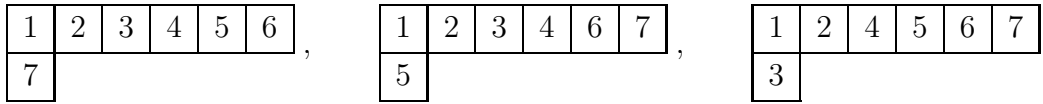
\[\begin{ytableau}
    1 & 2 & 3 & 4 & 5 & 6 \\ 7
\end{ytableau}  \; , \qquad \begin{ytableau}
    1 & 2 & 3 & 4 & 6 & 7 \\ 5
\end{ytableau}  \; , \qquad \begin{ytableau}
    1 & 2 & 4 & 5 & 6 & 7 \\ 3
\end{ytableau}\]
giving us  generating idempotents

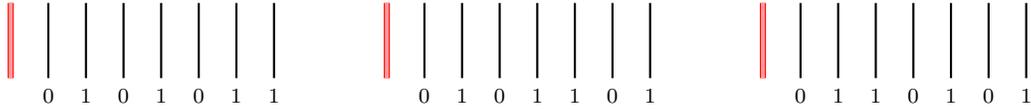
\begin{center}
    \begin{tikzpicture}[scale=0.5]
        \red{0}{0} 
        \idm{1}{0}{0} 
        \idm{2}{0}{1} 
        \idm{3}{0}{0} 
        \idm{4}{0}{1} 
        \idm{5}{0}{0} 
        \idm{6}{0}{1} 
        \idm{7}{0}{1} 
        
        \red{10}{0} 
        \idm{11}{0}{0} \idm{12}{0}{1} \idm{13}{0}{0} \idm{14}{0}{1} \idm{15}{0}{1}
        \idm{16}{0}{0}
        \idm{17}{0}{1}

        \red{20}{0} 
        \idm{21}{0}{0} \idm{22}{0}{1} \idm{23}{0}{1} \idm{24}{0}{0} \idm{25}{0}{1}
        \idm{26}{0}{0}
        \idm{27}{0}{1}

    \end{tikzpicture}
\end{center}

\end{example}

\section{Standardizations}
\label{sec:Standardizations}
In order to prove basic facts about the algebras $\Rz{\Lambda}{\alpha}$ and $\R{\PR}{\Lambda}{\alpha}$, we need to consider an analogue of the standard modules studied in \cite{brundanHomologicalProperties2014} and other works on KLRW algebras.

\subsection{Slope data}
\label{sec:slope-data}
\bcomments{I tried to change the indices here so $k$ is always the number of labels.  I might have missed some spots where I used it as a dummy variable. There was a real mess here, so if something looks wrong, probably I didn't fix the indices properly.}

Given a sequence $\Bi=(i_1,\dots, i_k)$, we define a piecewise linear function
\[f_{\Bi}:[0,k]\to \mathbb{R}\]
whose graph is the piecewise linear path that joins the points
$(m,\PR(\alpha_{i_1}+\cdots +\alpha_{i_m}))$.  In terms of formulas
\begin{align*}
     f_{\Bi}(x)&=\PR(\alpha_{i_1}+\cdots +\alpha_{i_{m}})+\PR(\al_{i_{m+1}})(x-m)& m\leq x\leq m+1.
\end{align*}
Let $f_{\Bi}^+\colon \mathbb{R}_{\geq 0}\to \mathbb{R}$ be the minimal concave, weakly increasing function such that $f_{\Bi}^+\geq f_{\Bi}$, and $f_{\Bi}^-\colon \mathbb{R}_{\geq 0}\to \mathbb{R}$ the maximal concave, weakly decreasing, function such that $f_{\Bi}^-\leq f_{\Bi}$.  Explicitly, 
\begin{align*}
    f_{\Bi}^+(x)&=\max\left(\epsilon f_{\Bi}(x_1)+(1-\epsilon)f_{\Bi}(x_2) \mid x_1,x_2\in [0,k], \epsilon\in [0,1], \epsilon x_1+(1-\epsilon)x_2<x\right)\\
    f_{\Bi}^-(x)&=\min\left(\epsilon f_{\Bi}(x_1)+(1-\epsilon)f_{\Bi}(x_2) \mid x_1,x_2\in [0,k], \epsilon\in [0,1], \epsilon x_1+(1-\epsilon)x_2<x\right)
\end{align*}  

More geometrically, we can find the graph of $ f_{\Bi}^{\pm}$ by starting at the origin $(u_0^{\pm},z_0^{\pm})=(0,0)$ and then inductively constructing $(u_i^{\pm},z_i^{\pm})$ by doing the following process:
\begin{enumerate}
    \item If there is a ray with beginning at $(u_n^{\pm},z_n^{\pm})$ with positive/negative slope intersecting the graph of  $f_{\Bi}$, consider the ray with greatest/least slope, and let $(u_{n+1}^{\pm},z_{n+1}^{\pm})$ be its last point of intersection with the graph.  The graph of $f_{\Bi}^{\pm}$ contains the portion of the ray between $(u_n^{\pm},z_n^{\pm})$ and $(u_{n+1}^{\pm},z_{n+1}^{\pm})$.
    \item If there is no such ray, then $f_{\Bi}^{\pm}$ is constant for $x\geq u_n^{\pm}$.  
\end{enumerate}
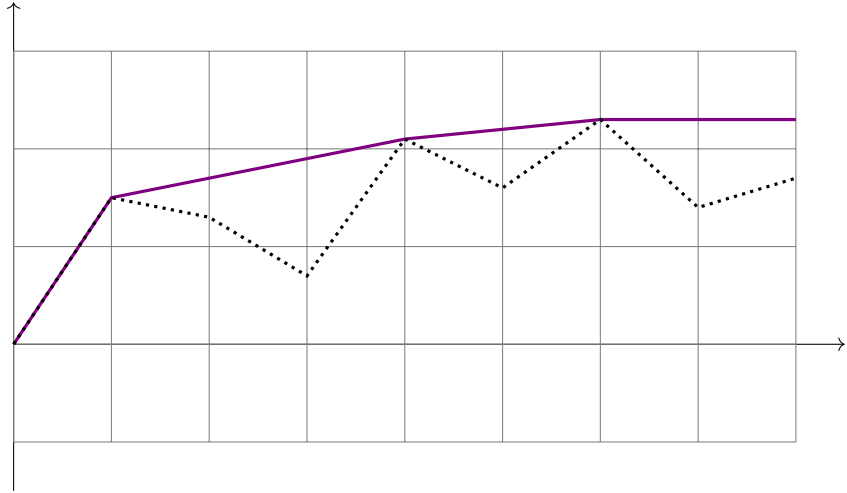
\begin{figure}\centerline{   \begin{tikzpicture}[scale=1.3]
\draw[->] (0, 0) -- (8.5, 0) ;
        \draw[->] (0, -1.5) -- (0, 3.5) ;

\draw[gray, thin] (-0, -1) grid (8, 3);

        \draw[very thick,violet!100!black] (0,0) -- (1,1.5) -- (4,2.1)--(6,2.3)--(8,2.3);
                \draw[very thick,dotted]
            (0, 0) -- (1, 1.5)
            -- (2, 1.3)
            -- (3, .7)
            -- (4, 2.1)
            -- (5, 1.6)-- (6,2.3)--(7,1.4) -- (8,1.7);
    \end{tikzpicture}
    }    \caption{An illustration of the definitions of $f_{\Bi}^+$.  The dotted black line is the graph of $f_{\Bi}$, the graph of $f_{\Bi}^+$ is in purple.}
\end{figure}

For an idempotent $e(\bell,\Br)$, we want to study the functions $f_{\Br}^+$ and $f_{\bell}^-$.   In what follows, we let $(u_n^{\pm},z_n^{\pm})$ be the points defined above for these functions, respectively.

These functions help us to understand the relations \eqref{steadiedidemp}: an idempotent is set to zero by these relations if $f_{\bell}^-\neq 0$ or $f_{\Br}^+\neq 0$.  More generally, we can think of the functions $f_{\bell}^-$  and $f_{\Br}^+$ as a measure of ``how unstable'' an idempotent is.  

The functions $f_{\bell}^-,f_{\Br}^+$ can also be described in terms of the {\bf slope datum} 
\[\boldsymbol{\gamma}(\bell,\Br)=(\gamma_{-s},\gamma_{-s+1},\dots, \gamma_0,\dots, \gamma_t)\] such that $\sum_{i=-s}^t \gamma_i=\alpha$ and
\[\frac{\chi(\gamma_{-s})}{h(\gamma_{-s})}< \frac{\chi(\gamma_{-s+1})}{h(\gamma_{-s+1})}<\cdots < \frac{\chi(\gamma_{-1})}{h(\gamma_{-1})}<0<\frac{\chi(\gamma_{1})}{h(\gamma_{1})}<\cdots <\frac{\chi(\gamma_{t})}{h(\gamma_{t})}\] where, for $m>0$, we have:
\[\gamma_{-m}=\sum_{p=u^-_{s-m}+1}^{u^-_{s-m+1}} \al_{\ell_p}\qquad\qquad  \gamma_{m}=\sum_{q=u^+_{t-m}+1}^{u^+_{t-m+1}}\al_{r_q}\]
That is, we obtain $\gamma_{-m}$ by summing the roots $\al_{\ell_i}$ between points where the function $f_{\bell}^-$ changes slopes and $\gamma_m$ by a similar sum of  $\al_{r_j}$ for $f_{\Br}^+$.   
Equivalently, we have 
\begin{align*}
	(u_{s-p+1}^-,v_{s-p+1}^-)&=\left(\sum_{a=-s}^{-p}\chi(\gamma_{-a}),\sum_{a=-s}^{-p}h(\gamma_{-a})\right)\\
 (u_{t-p+1}^+,v_{t-p+1}^+)&=\left(\sum_{b=q}^{t}\chi(\gamma_{b}),\sum_{b=q}^{t}h(\gamma_{b})\right)
\end{align*}
and this uniquely fixes $\boldsymbol{\gamma}$.  We can define functions $f^\pm_{\boldsymbol{\gamma}}$ as the unique piece-wise linear function whose graph passes through these points.

\begin{rmk}\label{rem:pressure}
    If we return to the physical metaphor of imagining pressure as a force on strands and imagine that once strands collide, they become stuck together and move according to their average pressure, then the slope datum of an idempotent exactly corresponds to the groups we will obtain as time goes to $\infty$.  In particular, $\gamma_0$ corresponds to the strands which end up stuck to the red strand, and all other components correspond to those that escape to $\pm \infty$, with the slope of $f^{\pm}_{\boldsymbol{\gamma}}$ giving the relative speeds at which different groups move when $t$ is large.  
\end{rmk}

\begin{example}\label{ex:4 graphs-1}
    Let $e=2$,  $\alpha= 3\delta$ and 
    $\chi(\alpha_0) = -3/2$, $\chi(\alpha_1)=1/2$. We'll graph some examples of these functions.  In order to visualize both functions together, we'll graph $f_{\bell}(-x)$ on the negative $x$ axis and    $f_{\Br}(x)$ is on the positive $x$-axis in purple, with the function 
        \[f(x)=\begin{cases}
        f_{\Br}^{+}(x)& x\geq 0\\
        f_{\bell}^-(-x)& x\leq 0
    \end{cases}\]  graphed in red. With these conventions, an idempotent is steady if the purple graph is entirely within the NW and SE quadrants, and the red line is the ``tightest'' weakly increasing piecewise linear function passing through the origin, and staying below the purple line for $x<0$ and above it for $x>0$. The slope data $\boldsymbol{\gamma}$ can be read off as well:  $\gamma_{-1}$ is the sum of the roots corresponding to the leftmost segment with positive slope (assuming there is such a segment left of the $y$-axis), then $\gamma_{-2}$ the next such segment, until we reach the $y$-axis.  Symmetrically, $\gamma_1$ corresponds to the rightmost segment with positive slope (if there is one right of the $x$-axis).  We finally find $\gamma_0=\alpha-\sum_{i\neq 0}\gamma_i$ by summing the roots corresponding to horizontal segments of the red graph.

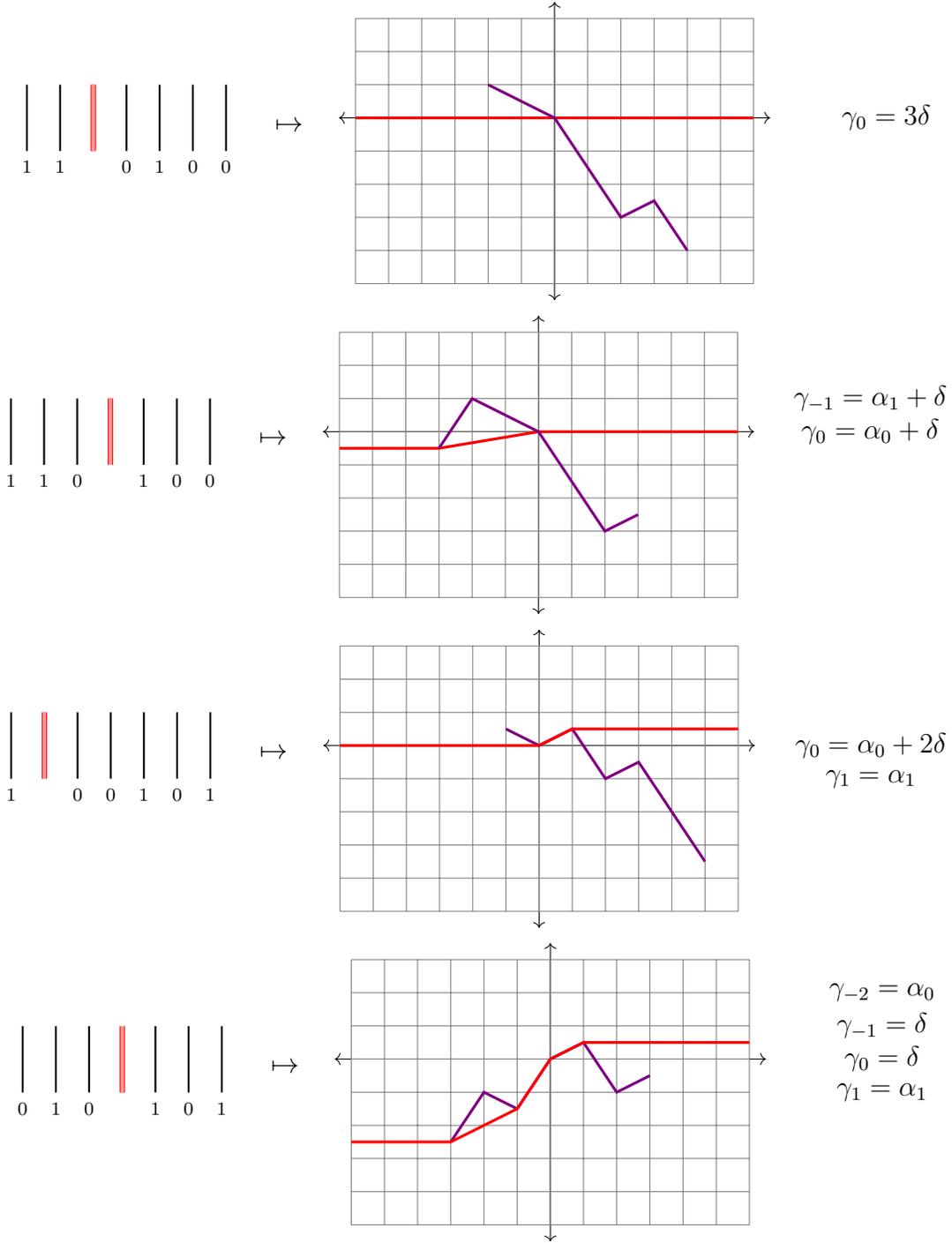
\begin{equation*}
   \begin{tikzpicture}[scale=0.5,baseline=3mm]
        \idm{0}{0}{1} \idm{1}{0}{1} \red{2}{0} \idm{3}{0}{0} \idm{4}{0}{1} \idm{5}{0}{0}
        \idm{6}{0}{0}
    \end{tikzpicture}\quad \mapsto \quad  \begin{tikzpicture}[scale=0.5,baseline=-2mm]
\draw[<->] (-6.5, 0) -- (6.5, 0) ;
        \draw[<->] (0, -5.5) -- (0, 3.5) ;

\draw[gray, thin] (-6, 3) grid (6, -5);

\draw[very thick,red!100!black] (-6,0) -- (6,0);
        \draw[very thick,violet!100!black] (-2,1)-- (0,0) -- (1,-1.5) -- (2, -3) -- (3, -2.5) -- (4, -4);
\node at (10,0){$\gamma_0=3\delta$};
    \end{tikzpicture}
\end{equation*}
\begin{equation*}
   \begin{tikzpicture}[scale=0.5,baseline=3mm]
      \idm{0}{0}{1} \idm{1}{0}{1} \red{3}{0} \idm{2}{0}{0} \idm{4}{0}{1} \idm{5}{0}{0}
        \idm{6}{0}{0}
    \end{tikzpicture}\quad \mapsto \quad  \begin{tikzpicture}[scale=0.5,baseline=-2mm]
\draw[<->] (-6.5, 0) -- (6.5, 0) ;
        \draw[<->] (0, -5.5) -- (0, 3.5) ;

\draw[gray, thin] (-6, 3) grid (6, -5);

\draw[very thick,red!100!black](-6,-0.5)--(-3,-0.5) -- (0,0) -- (6,0);
        \draw[very thick,violet!100!black] (-3,-0.5)--(-2,1)-- (0,0) -- (1,-1.5) -- (2, -3) -- (3, -2.5);
\node at (10,1){$\gamma_{-1}=\al_1+\delta$};
\node at (10,0){$\gamma_0=\al_0+\delta$};
    \end{tikzpicture}
\end{equation*}
\begin{equation*}
   \begin{tikzpicture}[scale=0.5,baseline=3mm]
      \idm{0}{0}{1} \idm{2}{0}{0} \red{1}{0} \idm{3}{0}{0} \idm{4}{0}{1} \idm{5}{0}{0}
        \idm{6}{0}{1}
    \end{tikzpicture}\quad \mapsto \quad  \begin{tikzpicture}[scale=0.5,baseline=-2mm]
\draw[<->] (-6.5, 0) -- (6.5, 0) ;
        \draw[<->] (0, -5.5) -- (0, 3.5) ;

\draw[gray, thin] (-6, 3) grid (6, -5);

\draw[very thick,violet!100!black] (-1,0.5)-- (0,0) -- (1,.5) -- (2, -1) -- (3, -.5)--(4,-2)--(5,-3.5);
        
        \draw[very thick,red!100!black](-6,0)--(0,0) -- (1,.5)--(6,.5);
\node at (10,-1){$\gamma_{1}=\al_1$};
\node at (10,0){$\gamma_0=\al_0+2\delta$};

    \end{tikzpicture}
\end{equation*}
\begin{equation*}
   \begin{tikzpicture}[scale=0.5,baseline=3mm]
      \idm{0}{0}{0} \idm{1}{0}{1}  \idm{2}{0}{0} \red{3}{0} \idm{4}{0}{1} \idm{5}{0}{0}
        \idm{6}{0}{1}
    \end{tikzpicture}\quad \mapsto \quad  \begin{tikzpicture}[scale=0.5,baseline=-2mm]
\draw[<->] (-6.5, 0) -- (6.5, 0) ;
        \draw[<->] (0, -5.5) -- (0, 3.5) ;

\draw[gray, thin] (-6, 3) grid (6, -5);

\draw[very thick,violet!100!black] (-3,-2.5)--(-2,-1)--(-1,-1.5)-- (0,0) -- (1,.5) -- (2, -1) -- (3, -.5);
        
        \draw[very thick,red!100!black](-6,-2.5)--(-3,-2.5)--(-1,-1.5)--(0,0) -- (1,.5)--(6,.5);
\node at (10,-1){$\gamma_{1}=\al_1$};
\node at (10,0){$\gamma_0=\delta$};
\node at (10,1){$\gamma_{-1}=\delta$};
\node at (10,2){$\gamma_{-2}=\al_0$};
    \end{tikzpicture}
\end{equation*}
\end{example}

\begin{comment}
\begin{example}
    Let $e=2$ and $\alpha= 3\delta$. Consider the idempotent in $\mathcal{R}(s_1\PRz, \Lambda_0, \alpha)$
\end{example}    
\end{comment}

We partially order the set of all possible slope data by saying that $\boldsymbol{\gamma}\leq \boldsymbol{\gamma}'$ if $f_{\boldsymbol{\gamma}}^+\leq f_{\boldsymbol{\gamma}'}^+$ and   
$f_{\boldsymbol{\gamma}}^-\geq f_{\boldsymbol{\gamma}'}^-$.  
We can extend this to a pre-order on idempotents via the map $(\bell, \Br)\mapsto\boldsymbol{\gamma} (\bell, \Br)$.
The minimal element of this partial order is $\boldsymbol{\gamma}=(\al)$.

\subsection{Standard modules}
As usual when we have a pre-order on idempotents, we can define standard modules: 
\begin{defn}
    Given an idempotent $e(\bell,\Br)$, the standard module $\Delta(\bell,\Br)$ is the quotient of $\R{\PR}{\Lambda}{\al}e(\bell,\Br)$ by the submodule generated by the image $e(\bell',\Br')\R{\PR}{\Lambda}{\al}e(\bell,\Br)$ of all idempotents $e(\bell',\Br')$ with $\boldsymbol{\gamma}(\bell',\Br')\not \leq \boldsymbol{\gamma}(\bell,\Br)$.
\end{defn}

Let $e(\boldsymbol{\gamma})$ be the sum of all idempotents with a fixed slope datum $\boldsymbol{\gamma}=\boldsymbol{\gamma}(\bell,\Br)$ and let $\Delta(\boldsymbol{\gamma})$ be the sum of the corresponding standard modules.
Note that by \cref{lem:only-roots}, the module $\Delta(\boldsymbol{\gamma})$ is only non-zero if each $\gamma_i$ for $i\neq 0$ is a multiple of a positive root.

\begin{defn}
	Fix a slope datum $\boldsymbol{\gamma}$.  Let 
	\begin{align*}
		\Rg{\gamma}&=R_{\gamma_{-s}}\otimes  \cdots \otimes R_{\gamma_{-1}}\otimes \Rz{\Lambda}{\gamma_0}\otimes R_{\gamma_1}\otimes  \cdots \otimes R_{\gamma_t}\\
		\Cg{\gamma}&= C_{\gamma_{-s}}\otimes  \cdots \otimes C_{\gamma_{-1}}\otimes \R{\PR}{\Lambda}{\gamma_0}\otimes C_{\gamma_1}\otimes  \cdots \otimes C_{\gamma_t}	\end{align*}
\end{defn}

Note that there is a homomorphism given by horizontal composition $\Rg{\gamma}\to \Rz{\Lambda}{\al}$.  Given modules $M_{p}$ over $R_{\gamma_p}$ if $p\neq 0$ and over $\Rz{\Lambda}{\gamma_0}$ if $p=0$, we can consider the base change:
\[M_{-s}\circ \cdots \circ M_{r}=\Rz{\Lambda}{\al}\otimes_{\Rg{\gamma}}M_{-s}\otimes \cdots \otimes  M_{r}.\]
An important lemma for us is: 
\begin{lem}\label{lem:leq-shuffle}
    If the module $M_{-s}\otimes \cdots \otimes  M_{r}$ factors through $\Cg{\gamma}$, then if \[e(\bell',\Br')(M_{-s}\circ \cdots \circ M_{r})\neq 0,\] we must have $\boldsymbol{\gamma}(\bell',\Br')\leq \boldsymbol{\gamma}$.  Furthermore, $e(\boldsymbol{\gamma})(M_{-s}\circ \cdots \circ M_{r})=M_{-s}\otimes \cdots \otimes M_{r}$.  
\end{lem}
\begin{proof}
    This follows immediately from the fact that any non-trivial shuffle of cuspidal words in this order will be strictly shorter in our pre-order.   The argument is identical to that of \cite[Lem. 2.7--9]{tingleyMirkovicVilonenPolytopes2016}. 
\end{proof}
This can be extended to prove that if $M_k$ is simple for all $k$, then $M_{-s}\circ \cdots \circ M_{r}$ has a unique simple quotient, and every simple module can be written uniquely in this form.  

\begin{thm}
We have isomorphisms \[\Delta(\boldsymbol{\gamma})\cong C_{\gamma_{-s}}\circ \cdots \circ  \R{\PR}{\Lambda}{\gamma_0}\circ \cdots \circ C_{\gamma_{r}}\] 
    \[\End(\Delta(\boldsymbol{\gamma}))^{\operatorname{op}}\cong \Cg{\gamma}.\]
\end{thm}
    
\begin{proof}
Sending $e(\boldsymbol{\gamma})$ to the identity element in the ring $\Cg{\gamma}$ induces a surjective map $\R{\PR}{\Lambda}{\al}e(\boldsymbol{\gamma})\to  C_{\gamma_{-s}}\circ \cdots \circ  \R{\PR}{\Lambda}{\gamma_0}\circ \cdots \circ C_{\gamma_{r}}$, 
By \cref{lem:leq-shuffle}, all idempotents with non-zero image in the target are $\leq \boldsymbol{\gamma}$, so this map factors through the quotient $\Delta(\boldsymbol{\gamma}).$   

On the other hand, the image $e(\boldsymbol{\gamma})\Delta(\boldsymbol{\gamma})$ is naturally a module over $\Rg{\gamma}$.  An element of the non-cuspidal ideal factors through an idempotent which is higher in the partial order, by definition.  Similarly, an element of the kernel of the map $\Rz{\Lambda}{\gamma_0}\to \R{\PR}{\Lambda}{\gamma_0}$ factors through an idempotent higher in this order. Thus, such elements are zero in $e(\boldsymbol{\gamma})\Delta(\boldsymbol{\gamma})$, and so the action on this space factors through the quotient $\Cg{\gamma}$.  Thus, acting on the image of $e(\boldsymbol{\gamma})$ induces a surjective map $C_{\gamma_{-s}}\circ \cdots \circ  \R{\PR}{\Lambda}{\gamma_0}\circ \cdots \circ C_{\gamma_{r}}\to \Delta(\boldsymbol{\gamma})$, which must be inverse to the map constructed above. 

This shows that $e(\boldsymbol{\gamma})\Delta(\boldsymbol{\gamma})\cong \End(\Delta(\boldsymbol{\gamma}))^{\operatorname{op}}$ is exactly $\Cg{\gamma}$, using \cref{lem:leq-shuffle} again.  Thus, the right action by the elements of this quotient induces the desired isomorphism.  
\end{proof}

We can extend this result to show that 
\[M_{-s}\circ \cdots \circ M_{r}\cong \Delta(\boldsymbol{\gamma})\otimes_{\Cg{\gamma}} (M_{-s}\otimes \cdots \otimes M_{r}).\]
This gives a fully-faithful functor $\Cg{\gamma}\mmod \to \Rz{\Lambda}{\alpha}$ of {\bf standardization}.
In particular, we can choose a simple $\Cg{\gamma}$-module $L$, and the resulting proper standard \[\bar{\Delta}(L)=\Delta(\boldsymbol{\gamma})\otimes_{\Cg{\gamma}}L.\]  Unlike the module $\Delta(\boldsymbol{\gamma})$, this module is finite-dimensional.
There is a unique indecomposable summand ${\Delta}(L)$ of $\Delta(\boldsymbol{\gamma})$ which maps surjectively to $\bar{\Delta}(L)$; this is the standardization of the projective cover of $L_p$.  
 
 Let $V$ be a graded vector space with each degree finite-dimensional.  We let $V^*$  denote the restricted dual, the vector space of functionals on $V$ which are 0 on all but finitely many degrees;  this is the same as the direct sum of the duals taken degree-wise.  
 We'll also work with the costandard modules $\nabla(\boldsymbol{\gamma})=(\Cg{\boldsymbol{\gamma}}\otimes_{\Rg{\boldsymbol{\gamma}}}\Rz{\Lambda}{\al})^*$;  note that, unlike the $ \Delta(\boldsymbol{\gamma})$, these modules are not finitely generated.    We have $\End(\nabla(\boldsymbol{\gamma}))^{\operatorname{op}}\cong \Cg{\gamma}$, so $\Hom(N,\nabla(\boldsymbol{\gamma}))$ is a right $\Cg{\gamma}$-module.

We note that we obtain analogues of \cite[Th. 3.12-3]{brundanHomologicalProperties2014} with essentially identical proofs.  By analogy, we define a $\Delta$-flag to be a filtration with all subquotients isomorphic to ${\Delta}({L})$ for some ${L}$ and a $\bar{\Delta}$-flag to be a filtration with all subquotients isomorphic to $\bar{\Delta}({L})$ for some ${L}$.
\begin{thm}\label{thm:standard-standard}
If $e(\bell,\Br)M=0$ for all $\boldsymbol{\gamma}(\bell,\Br)\not \leq \boldsymbol{\gamma}$, then $\Ext^d(\bar{\Delta}({L}), M)=0.$ 
Furthermore, 
\[\Ext^d(\bar{\Delta}({L}),\nabla(\boldsymbol{\gamma}))=\begin{cases}
    L^* & d=0\\
    0 & d>0
\end{cases}
    \]
A $\R{\PR}{\Lambda}{\al}$-module $N$ has a $\bar{\Delta}$-flag if and only if $\Ext^d(N,\nabla(\boldsymbol{\gamma}))=0$ for all $d>0$ and all $\boldsymbol{\gamma}$.  In this case, $N$ is filtered by the standardizations of the left $\Cg{\gamma}$-modules $N_{\boldsymbol{\gamma}}=\Hom(N, \nabla(\boldsymbol{\gamma}))^*$.  
\end{thm}

In particular, this means that any projective $\R{\PR}{\Lambda}{\al}$-module has a $\Delta$-flag. 
A natural extension of this result is that there are right standard modules as well, and if $M$ is a right module with a $\Delta$-flag and $N$ a left module with a $\bar \Delta$-flag, then $\operatorname{Tor}^d(M,N)=0$ for $d>0$.  In particular, for all $k>0$:
\begin{lem}\label{lem:Tor-flag}
    If $N$ has a $\bar \Delta$-flag, then \[\operatorname{Tor}^k_{\Rz{\Lambda}{\al}}(\R{\PR}{\Lambda}{\al}, N)=0 \qquad\qquad  \Ext^k_{\Rz{\Lambda}{\al}}(\R{\PR}{\Lambda}{\al},N^*)=0.\]   
If the action on $N$ factors through  $\R{\PR}{\Lambda}{\al}$ then $\R{\PR}{\Lambda}{\al}\Lotimes_{\Rz{\Lambda}{\al}} N=N$.   
\end{lem}

\subsection{Finite-dimensionality}

\begin{lem}
   The algebra $\R{\PR}{\Lambda}{\alpha}$ is finite-dimensional.
\end{lem}
\begin{proof}
{\bf Reduction to nilpotence of dots}: Since $e(\bell,\Br)\Rz{\Lambda}{\alpha}$ is finitely generated as a module over $e(\bell,\Br)\Rz{\Lambda}{\alpha}e(\bell,\Br)$, the same is true of $e(\bell,\Br)\R{\PR}{\Lambda}{\alpha}$ as a module over $e(\bell,\Br)\R{\PR}{\Lambda}{\alpha}e(\bell,\Br)$.  Thus, we need only to prove that the latter algebra is finite-dimensional.  Furthermore, since $e(\bell,\Br)\R{\PR}{\Lambda}{\alpha}e(\bell,\Br)$ is finitely generated as a module over the subalgebra generated by the dots, we need only to show that the subalgebra generated by the dots is finite-dimensional, or equivalently that any dot on one of the strands in $e(\bell,\Br)$ is nilpotent. 

{\bf Defining $g$'s:} In order to perform the needed induction, we need to define variations on the functions $f_{\Bi}^{\pm}$ for $\Bi=(i_1,\dots, i_k)$ which don't require the weak increasing/decreasing property.  Let 
\begin{align*}\mathbf{E}_x&=\{(x_1,x_2,\epsilon)\in [0,r]\times [0,r]\times  [0,1] \mid  \epsilon x_1+(1-\epsilon)x_2=x\}\\
    g_{\Bi}^+(x)&=\begin{cases}
    	\max\left(\epsilon f_{\Bi}(x_1)+(1-\epsilon)f_{\Bi}(x_2) \mid (x_1,x_2, \epsilon)\in \mathbf{E}_x\right) & x\in [0,k]\\
    	\frac{\PR(\al_{\Bi})}{k}x& x\geq k
    \end{cases}\\
    g_{\Bi}^-(x)&=\begin{cases}\min\left(\epsilon f_{\Bi}(x_1)+(1-\epsilon)f_{\Bi}(x_2) \mid (x_1,x_2, \epsilon)\in \mathbf{E}_x\right) &x\in [0,k]\\
    	\frac{\PR(\al_{\Bi})}{k}x & x\geq k
\end{cases}
\end{align*}  
The function $g_{\Bi}^{+}$  is the least (resp.\ greatest) convex function greater than (resp.\ less than) $f_{\Bi}$ on its domain of definition, and then extended by the line connecting this point to the origin from the end of its domain.   Similarly, the function $g_{\Bi}^{-}$ is the greatest concave function less than $f_{\Bi}$ on its domain of definition, and then extended by the same line.  In both cases, this extension will almost always break the convexity or concavity, but is helpful for our induction.
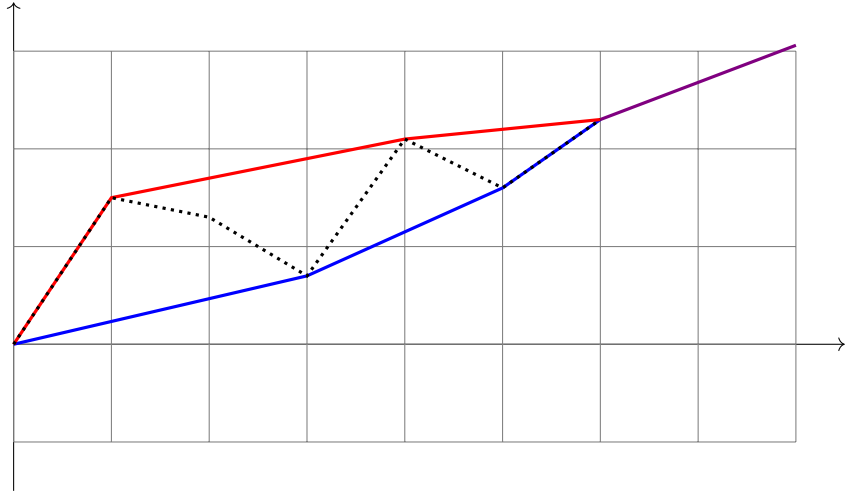
\begin{figure}\centerline{   \begin{tikzpicture}[scale=1.3]
\draw[->] (0, 0) -- (8.5, 0) ;
        \draw[->] (0, -1.5) -- (0, 3.5) ;

\draw[gray, thin] (-0, -1) grid (8, 3);

\draw[very thick, violet!100!black](6,2.3)--(8,3.06);
        \draw[very thick,blue] (0,0) -- (3,.7) -- (5,1.6)--(6,2.3);
        
        \draw[very thick,red!100!black] (0,0) -- (1,1.5) -- (4,2.1)--(6,2.3);
                \draw[very thick,dotted]
            (0, 0) -- (1, 1.5)
            -- (2, 1.3)
            -- (3, .7)
            -- (4, 2.1)
            -- (5, 1.6)-- (6,2.3);
    \end{tikzpicture}
    }    \caption{An illustration of the definitions of $g_{\Bi}^{\pm}$.  The dotted black line is the graph of $f_{\Bi}$, the graph of $g_{\Bi}^+$ is in red, that of $g_{\Bi}^-$ in blue.  Note that these coincide for $x\geq k$, shown in purple.}
\end{figure}

{\bf Setting up induction:} We'll prove this inductively on the choice of $(\bell,\Br)$; assume that $e(\bell',\Br')\R{\PR}{\Lambda}{\alpha}e(\bell',\Br')$ is finite-dimensional whenever \begin{enumerate}
    \item   $g_{\bell'}^-\geq g_{\bell}^-$    
    \item $g_{\Br'}^+\leq g_{\Br}^+$  
    \item there is a point where one of these inequalities is strict.
\end{enumerate}  
For simplicity, we'll call idempotents satisfying this condition ``higher'' than $(\bell,\Br)$

We can complete the proof with the assumption that the dot in question is left of the red strand; the argument in the other direction is symmetric.  

Note that if $(\bell',\Br')$ is higher than $(\bell,\Br)$, then know that $e(\bell,\Br)\R{\PR}{\Lambda}{\alpha}e(\bell',\Br')$ is finitely generated over $e(\bell',\Br')\R{\PR}{\Lambda}{\alpha}e(\bell',\Br')$, and thus is finite-dimensional, so the same is true of the product $e(\bell,\Br)\R{\PR}{\Lambda}{\alpha}e(\bell',\Br')\R{\PR}{\Lambda}{\alpha}e(\bell,\Br)$.  If we let $I$ be the two-sided ideal generated by all such products, we can conclude $I$ is finite-dimensional, so we need only show that $\mathsf{R}=e(\bell,\Br)\R{\PR}{\Lambda}{\alpha}e(\bell,\Br)/I$ is finite-dimensional.

{\bf Sliding over the red strand:} As discussed before, the function $g_{\bell}^-$ must be weakly decreasing, or the idempotent will be 0 by \eqref{steadiedidemp}.  Much like the definition of $u_i^{\pm}$ in \cref{sec:slope-data}, we can define $v_i^{\pm}$ by letting $v_0^{\pm}=0$, and $v_{m+1}^{\pm}$ to be the next $x$-value $\geq v_{m}^{\pm}$ where the slope of $g_{\Bi}^{\pm}$ changes.  

Let $D_t$ be the subalgebra in $\mathsf{R}$ generated by dots on the strands between the $(v_{t-1}^{\pm}+1)$-st to the $v_{t}^{\pm}$-th, counted from right to left.  We'll prove that the subalgebra generated by $D_1,\dots, D_t$ is finite-dimensional by induction on $t$.

Assuming that the $r+1$st through $s$th strands are the $t$th group, acting on this group of strands gives a map of
$R_{\gamma}$, where $\gamma=\alpha_{i_{v_{t-1}^{\pm}+1}}+\cdots+ \alpha_{i_{v_{t}^{\pm}}}$ is given by $\mathsf{R}$.  This map factors through the cuspidal quotient $C_{\gamma}$.  This means that $\mathsf{R}=0$ unless $\gamma$ is a root.  

Furthermore, we can consider the diagram $\varphi_{\spadesuit}$ which is defined by crossing the strands in the $t$th group over the red strand and all black strands between them (the 1st through $(t-1)$st groups), with every crossing of a pair of strands with the same label replaced by $\wp=\tikz[scale=.3,baseline]{\Xnw{0}{-1}{}{}\dm{3}{-1}\Xsw{4}{-1}{}{}}$.  If we compose this diagram with its reflection $\varphi_{\spadesuit}^*$ through a horizontal plane, then we can apply \cite[Lem. 1.5]{kangSymmetricQuiver2018} to show that 
\begin{equation}\label{eq:varphi-spade}
    \varphi_{\spadesuit}^*\varphi_{\spadesuit}=\pm \prod_{\substack{r<p\leq s\\ s<q\\ i_p\neq i_q}} Q_{i_{p} i_{q}}(y_p,y_q).
\end{equation}  The sign appears because the reflection of $\wp$ is $-\wp$.

{\bf Comparing functions:} The top $(\bell',\Br')$ of $\varphi_{\spadesuit}$ is given by the idempotent where we have moved the group of strands $(i_{r+1},\dots, i_s)$ to the end of $\Br$, that is, $\Br'=(j_1,\dots, j_u,i_s,\dots, i_{r+1})$.  
Since the segment between $r$ and $s$ is replaced with segments with positive slope, and the ratio $\PR(\al_{\bell'})/h(\al_{\bell'})>\PR(\al_{\bell})/h(\al_{\bell})$ increases, we have $g^-_{\bell}\leq g_{\bell'}^+$.

The function $g_{\Br'}$ is the extension of $g_{\Br}$ by the function determined by $(i_s,\dots, i_{r+1})$.  By assumption, this graph stays above a line of slope $\PR(\gamma)/(s-r)>0$; this value is positive since $g_{\bell}^-$ is increasing.  
Combined with the fact that $\PR(\al_{\Br'})/h(\al_{\Br'})>\PR(\al_{\Br})/h(\al_{\Br})$, this shows that $g^+_{\Br}\leq g_{\Br'}^+$.

 If $\PR(\alpha_{\Br}+\gamma)>0$, then our diagram is zero by \eqref{steadiedidemp}.  Otherwise, since $g_{\Br}^+$ is decreasing, there must be some $x< u$ such that $g_{\Br}^+(x)=\chi(\alpha_{\Br}+\gamma)$.  We must have $ g_{\Br'}^+(x)>\chi(\alpha_{\Br}+\gamma)$, since this function is strictly decreasing as well, so $g_{\Br}^+$ and $g_{\Br'}^+$ are not equal.  This shows that in $\mathsf{R}$, the left-hand side, and thus also the right-hand side, of \eqref{eq:varphi-spade} vanishes.  

{\bf Deducing nilpotence:} By \cite[Lem. 15.4]{mcnamaraRepresentationsKhovanov2017}, in the cuspidal quotient $C_{\gamma}$, the difference $y_p-y_{p'}$ is nilpotent in $D_t$ for all $p,p'\in [r+1,s]$.  Thus, if we let  \[\mathfrak{Q}(u)=\prod_{\substack{r<p\leq s\\ s<q\\ i_p\neq i_q}} Q_{i_{p} i_{q}}(u,y_q)\]
then for any $p\in [r+1,s]$, we have that $\mathfrak{Q}(y_p)$ is nilpotent in $\mathsf{R}$.   Each of the coefficients of this polynomial in $u$, with the exception of the leading coefficient, is a positive degree element of the subalgebra generated by $D_1,\dots, D_{t-1}$.  By induction, all these elements are nilpotent, so this proves that $y_p$ is also nilpotent in $D_t$.  This completes the proof.  
\end{proof}

\begin{example}\label{ex:4 graphs-2}
Let us return to the case of \cref{ex:4 graphs-1} and show the functions $g_{\Br}^{+},g_{\bell}^-$ much as we did with $f_{\Br}^{+},f_{\bell}^-$.  Recall that
$e=2$,  $\alpha= 3\delta$ and 
    $\chi(\alpha_0) = -3/2$, $\chi(\alpha_1)=1/2$. We draw the graph of 
    \[g(x)=\begin{cases}
        g_{\Br}^{+}(x)& x\geq 0\\
        g_{\bell}^-(-x)& x\leq 0
    \end{cases}\] in yellow in the graphs below:

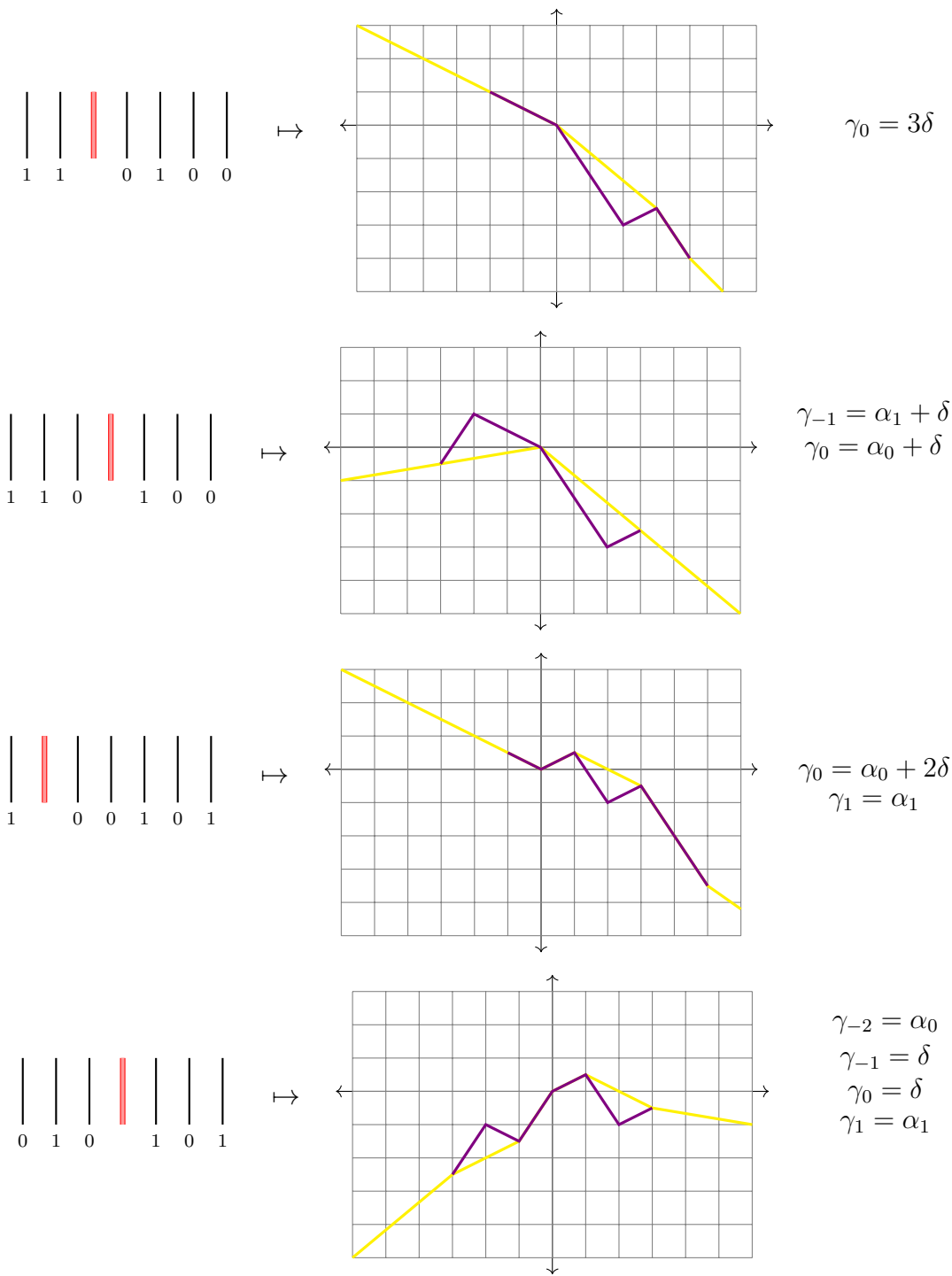
\begin{equation*}
   \begin{tikzpicture}[scale=0.5,baseline=3mm]
        \idm{0}{0}{1} \idm{1}{0}{1} \red{2}{0} \idm{3}{0}{0} \idm{4}{0}{1} \idm{5}{0}{0}
        \idm{6}{0}{0}
    \end{tikzpicture}\quad \mapsto \quad  \begin{tikzpicture}[scale=0.5,baseline=-2mm]
\draw[<->] (-6.5, 0) -- (6.5, 0) ;
        \draw[<->] (0, -5.5) -- (0, 3.5) ;

\draw[gray, thin] (-6, 3) grid (6, -5);

\draw[very thick,yellow!100!black] (-6,3) --(0,0) -- (3, -2.5) -- (4, -4) -- (5,-5);
        \draw[very thick,violet!100!black] (-2,1)-- (0,0) -- (1,-1.5) -- (2, -3) -- (3, -2.5) -- (4, -4);
\node at (10,0){$\gamma_0=3\delta$};
    \end{tikzpicture}
\end{equation*}
\begin{equation*}
   \begin{tikzpicture}[scale=0.5,baseline=3mm]
      \idm{0}{0}{1} \idm{1}{0}{1} \red{3}{0} \idm{2}{0}{0} \idm{4}{0}{1} \idm{5}{0}{0}
        \idm{6}{0}{0}
    \end{tikzpicture}\quad \mapsto \quad  \begin{tikzpicture}[scale=0.5,baseline=-2mm]
\draw[<->] (-6.5, 0) -- (6.5, 0) ;
        \draw[<->] (0, -5.5) -- (0, 3.5) ;

\draw[gray, thin] (-6, 3) grid (6, -5);

\draw[very thick,yellow!100!black](-6,-1)--(-3,-0.5) -- (0,0) -- (6,-5);
        \draw[very thick,violet!100!black] (-3,-0.5)--(-2,1)-- (0,0) -- (1,-1.5) -- (2, -3) -- (3, -2.5);
\node at (10,1){$\gamma_{-1}=\al_1+\delta$};
\node at (10,0){$\gamma_0=\al_0+\delta$};
    \end{tikzpicture}
\end{equation*}
\begin{equation*}
   \begin{tikzpicture}[scale=0.5,baseline=3mm]
      \idm{0}{0}{1} \idm{2}{0}{0} \red{1}{0} \idm{3}{0}{0} \idm{4}{0}{1} \idm{5}{0}{0}
        \idm{6}{0}{1}
    \end{tikzpicture}\quad \mapsto \quad  \begin{tikzpicture}[scale=0.5,baseline=-2mm]
\draw[<->] (-6.5, 0) -- (6.5, 0) ;
        \draw[<->] (0, -5.5) -- (0, 3.5) ;

\draw[gray, thin] (-6, 3) grid (6, -5);

        \draw[very thick,yellow!100!black](-6,3)--(0,0) -- (1,.5)--(3,-.5)--(5,-3.5)--(6, -4.2);
        \draw[very thick,violet!100!black] (-1,0.5)-- (0,0) -- (1,.5) -- (2, -1) -- (3, -.5)--(4,-2)--(5,-3.5);
        
\node at (10,-1){$\gamma_{1}=\al_1$};
\node at (10,0){$\gamma_0=\al_0+2\delta$};

    \end{tikzpicture}
\end{equation*}
\begin{equation*}
   \begin{tikzpicture}[scale=0.5,baseline=3mm]
      \idm{0}{0}{0} \idm{1}{0}{1}  \idm{2}{0}{0} \red{3}{0} \idm{4}{0}{1} \idm{5}{0}{0}
        \idm{6}{0}{1}
    \end{tikzpicture}\quad \mapsto \quad  \begin{tikzpicture}[scale=0.5,baseline=-2mm]
\draw[<->] (-6.5, 0) -- (6.5, 0) ;
        \draw[<->] (0, -5.5) -- (0, 3.5) ;

\draw[gray, thin] (-6, 3) grid (6, -5);

        \draw[very thick,yellow!100!black](-6,-5)--(-3,-2.5)--(-1,-1.5)--(0,0) -- (1,.5)--(3,-.5)--(6,-1);
        \draw[very thick,violet!100!black] (-3,-2.5)--(-2,-1)--(-1,-1.5)-- (0,0) -- (1,.5) -- (2, -1) -- (3, -.5);
        
\node at (10,-1){$\gamma_{1}=\al_1$};
\node at (10,0){$\gamma_0=\delta$};
\node at (10,1){$\gamma_{-1}=\delta$};
\node at (10,2){$\gamma_{-2}=\al_0$};
    \end{tikzpicture}
\end{equation*}
\end{example}

\section{The categorical action}
\label{sec:The categorical action}

\subsection{Induction and restriction functors}

Recall that $\mathcal{I}(\chi)\subset \Rz{\Lambda}{\alpha}$ is the kernel of the map to $\R{\PR}{\Lambda}{\alpha}$; since in our discussion below we will consider different values of $\alpha$, we will include them in the notation and write $\mathcal{I}(\chi,\alpha)$.  

\begin{defn}
	We call $\beta$ {\bf right-addable} if the image of $\mathcal{I}(\chi,\alpha)\otimes 1_{R_{\beta}}$  under the horizontal composition lies in $\mathcal{I}(\chi,\alpha+\beta)$, and {\bf left-addable} if the image of $1_{R_{\beta}}\otimes \mathcal{I}(\chi,\alpha)$ lies in $\mathcal{I}(\chi,\alpha+\beta)$.  
\end{defn}
By \cref{lem:only-roots}, it follows immediately that whether $\beta$ is left- or right-addable only depends on the alcove in which $\chi$ lies.  One natural class of these roots are those which define the facets of the alcove---up to sign, these are the simple roots for the basis which is positive on the alcove.  We will always choose the root which is in the usual positive roots, and call these {\bf $\PR$-simple}. We call such a root {\bf positive/negative $\PR$-simple} depending on the sign of $\PR$.  
\begin{lem}\label{lem:addable-roots}
	If $\beta$ is positive (resp. negative)  $\PR$-simple, then $\beta$ is left (resp. right) addable and the map $R_{\beta}\otimes \Rz{\Lambda}{\alpha}\to  \Rz{\Lambda}{\alpha+\beta}$ induces a map $C_{\beta}\otimes \R{\PR}{\Lambda}{\alpha}\to  \R {\PR}{\Lambda}{\alpha+\beta}. $
\end{lem}
\begin{proof}
Since the argument is symmetric, consider the case where $\beta$ is positive $\PR$-simple.  By \cref{lem:only-roots}, the ideal $\mathcal{I}(\chi,\alpha)$ only depends on the alcove of $\PR$, so we can without loss of generality assume that for any finite set $\{\beta_1,\dots, \beta_g\}$, we have $\PR(\beta)<|\PR(\beta_i)|$ for $i=1,\dots, g$.

	The ideal $\mathcal{I}(\chi,\alpha) $ is generated by the idempotents $e(\bell,\Br)$ where $\alpha_{(\ell_1,\dots, \ell_{i})}$ is a negative pressure root and those where $\alpha_{(r_1,\dots, r_{j})}$ is a positive pressure root.  The latter kind of generators are unchanged by horizontal multiplication on the left.  For the former, we can take the leftmost strands corresponding to $\beta$ and those with labels $\alpha_{(\ell_1,\dots, \ell_{i})}$ to obtain $\beta+ \alpha_{(\ell_1,\dots, \ell_{i})} $.  Since only finitely many roots can appear as $\alpha_{(\ell_1,\dots, \ell_{i})} $ for fixed $\alpha$, we can assume that $\PR(\beta)<|\PR(\alpha_{(\ell_1,\dots, \ell_{i})})|$ whenever $\alpha_{(\ell_1,\dots, \ell_{i})} $ is a negative pressure root.  Thus $\PR(\beta+ \alpha_{(\ell_1,\dots, \ell_{i})})<0, $ and so the image lies in $\mathcal{I}(\chi,\alpha+\beta). $
	
Finally, we need to show that the map of $R_{\beta}$ factors through $C_{\beta}$. If we write an idempotent $\bell'$ with $\alpha_{\bell'}=\beta$ as the concatenation $\bell'=\bell_1'\bell_2'$, we can assume that $\PR(\beta)<|\PR(\alpha_{\bell'_k})|$ for $k=1,2$, which is only possible if these have opposite sign.  Thus, if we have $\PR(\alpha_{\bell_1'})/h(\alpha_{\bell_1'})<\PR(\alpha_{\bell_2'})/h(\alpha_{\bell_2'}) $, we must have $\PR(\alpha_{\bell_1'}) <0$, so the horizontal composition of $e(\bell')$ with any diagram in $\Rz{\Lambda}{\alpha}$ lies in $\mathcal{I}(\chi,\alpha+\beta)$. This shows the desired factorization.
\end{proof}

 We wish to construct a categorical action of $\mathfrak{sl}_2$ on the categories $\R{\PR}{\Lambda}{\alpha}\mmod$ for each $\PR$-simple root $\beta$.  This action is analogous to that on categories of modules over cyclotomic quotients using induction and restriction functors induced by the non-unital inclusions from \cref{lem:addable-roots}.  Let $1_{\beta,\alpha}$ be the image of the identity of $C_{\gamma}\otimes \R{\PR}{\Lambda}{\alpha}$ under this map and define
\[\Ind{\alpha}{\alpha+\beta}(M)=\R{\PR}{\Lambda}{\alpha+\beta}1_{\beta,\alpha}\otimes_{R_{\beta}\otimes \R{\PR}{\Lambda}{\alpha}} (\Delta(\beta)\boxtimes M). \]
\[\Res{\alpha+\beta}{\alpha}(N)=\Hom_{R_{\beta}}(\Delta(\beta),1_{\beta,\alpha}N) \]
For a negative $\PR$-simple root, we can define similar functors by reversing the roles of left and right.  

\subsection{Categorical $\mathfrak{sl}_2$-actions}
Consider the direct sums 
\begin{equation}\label{eq:Ccat-def}
    \Ccat_r(\Lambda)=\bigoplus_{\langle\Lambda-\alpha,\beta^{\vee}\rangle=r}\R{\PR}{\Lambda}{\alpha}\mmod\qquad \qquad \Ccat(\Lambda)=\bigoplus_{r\in \Z}\Ccat_r(\Lambda).
\end{equation}  We can define autofunctors of this direct sum:
\[\eF=\bigoplus_{\alpha}\Ind{\alpha}{\alpha+\beta}\colon \Ccat_r(\Lambda)\to \Ccat_{r+2}(\Lambda)\qquad\qquad  \eE=\bigoplus_{\alpha}\Res{\alpha+\beta}{\alpha}\colon \Ccat_r(\Lambda)\to \Ccat_{r-2}(\Lambda)\]

Recall that two functors $\eE,\eF$ define a categorical $\mathfrak{sl}_2$ action if 	
\begin{enumerate}
\item[(KM1)] there is a prescribed adjunction
$(\eF,\eE )$;
\item[(KM2)]
For $m \geq 0$ there is an action of the
nilHecke algebra  $NH_m$ with $m$ strands on the $m$th power functor $\eE^m $;
\item[(KM3)]
If $r\geq 0$, there is an isomorphism $\eF\eE|_{\Ccat_{r}}\oplus \operatorname{Id}_{\Ccat_{r}}^{\oplus r}\Rightarrow
\eE\eF|_{\Ccat_{r}}
$, induced by the column vector
\begin{equation}
\left[
\begin{tikzpicture}[baseline = -1mm,thick]
	\draw[->] (0.28,-.28) to (-0.28,.28);
\draw[<-] (-0.28,-.28) to (0.28,.28);
\end{tikzpicture}
\quad
\mathord{
\begin{tikzpicture}[baseline = 1mm,thick]
	\draw[-] (0.4,.4) to[out=-90, in=0] (0.1,0);
	\draw[->] (0.1,0) to[out = 180, in = -90] (-0.2,.4);
\end{tikzpicture}
}
\quad
\mathord{
\begin{tikzpicture}[baseline = 1mm,thick]
	\draw[-] (0.4,.4) to[out=-90, in=0] (0.1,0);
	\draw[->] (0.1,0) to[out = 180, in = -90] (-0.2,.4);
\node at (0.37,.2) {$\bullet$};
\node at (0.57,.2) {$\scriptstyle x$};
\end{tikzpicture}
}
\quad
\cdots
\quad
\mathord{
\begin{tikzpicture}[baseline = 1mm,thick]
	\draw[-] (0.4,.4) to[out=-90, in=0] (0.1,0);
	\draw[->] (0.1,0) to[out = 180, in = -90] (-0.2,.4);
\node at (0.37,.2) {$\bullet$};
\node at (.87,.2) {$\scriptstyle x^{r-1}$};
\end{tikzpicture}
}
\right]^T
\end{equation}
 If $r\leq 0$, there is an isomorphism $\eE \eF|_{\Ccat_r}
\oplus \operatorname{Id}_{\Ccat_r}^{\oplus -r}\Rightarrow
\eF\eE|_{\Ccat_r}$ induced by the row vector
\begin{equation}
\left[
\mathord{
\begin{tikzpicture}[baseline = 0,thick]
	\draw[<-] (-0.28,-.3) to (0.28,.3);
\draw[->] (0.28,-.3) to (-0.28,.3);
   \end{tikzpicture}
}\quad 
\mathord{
\begin{tikzpicture}[baseline = 1mm,thick]
	\draw[-] (0.4,0) to[out=90, in=0] (0.1,0.4);
      \node at (-0.15,0.45) {$\phantom\bullet$};
	\draw[->] (0.1,0.4) to[out = 180, in = 90] (-0.2,0);
\end{tikzpicture}
}\quad
\mathord{
\begin{tikzpicture}[baseline = 1mm,thick]
	\draw[-] (0.4,0) to[out=90, in=0] (0.1,0.4);
	\draw[->] (0.1,0.4) to[out = 180, in = 90] (-0.2,0);
      \node at (-0.15,0.45) {$\phantom\bullet$};
      \node at (-0.15,0.2) {$\bullet$};
\node at (-0.37,.2) {$\scriptstyle x$};
\end{tikzpicture}
}\quad\cdots \quad 
\mathord{
\begin{tikzpicture}[baseline = 1mm,thick]
	\draw[-] (0.4,0) to[out=90, in=0] (0.1,0.4);
	\draw[->] (0.1,0.4) to[out = 180, in = 90] (-0.2,0);
\node at (-0.65,.3) {$\scriptstyle x^{-r-1}$};
      \node at (-0.15,0.42) {$\phantom\bullet$};
      \node at (-0.15,0.2) {$\bullet$};
\end{tikzpicture}
}\,
\right]
\end{equation}
\end{enumerate}

\begin{thm}\label{sl2action}
     For each $\PR$-simple root $\beta$, the functors $\eE,\eF$ define a categorical $\mathfrak{sl}_2$ action.  
\end{thm}
We'll require a long series of lemmas to prove this.  First we observe that:
\begin{lem}
    Properties (KM1) and (KM2) hold.
\end{lem}
\begin{proof}
    Property (KM1) is just the property that $\eF$ is the left adjoint of $\eE$; this is the standard tensor-Hom adjunction.

By \cite[Lem. 3.6]{brundanHomologicalProperties2014}, we have a natural transformation:
\[ \tau \colon\Ind{\alpha+\beta}{\alpha+2\beta}\circ \Ind{\alpha}{\alpha+\beta}\to \Ind{\alpha+\beta}{\alpha+2\beta}\circ \Ind{\alpha}{\alpha+\beta} \] which satisfies the relations of the nilHecke algebra:
\[x_1\tau-\tau x_2= \tau x_1-x_2\tau=1\qquad \tau^2=0\qquad \tau_{12}\tau_{23}\tau_{12}=\tau_{23}\tau_{12}\tau_{23}.\]  This defines the desired nilHecke action.  Note that this is only proven for the finite-type case in \cite{brundanHomologicalProperties2014}, but as noted in \cite[Th. 24.1]{mcnamaraRepresentationsKhovanov2017}, the same proof works for real roots of affine Lie algebras.
\end{proof}

From (KM2), we can define divided powers of $\eF$:
\begin{defn}
    Let $\Delta_{\beta}^{(n)}$ be the submodule $\bigcap_{i=1}^{n-1}\ker \tau_i$ of $\Delta_{\beta}^{n}$.  Let $\eF^{(n)}$ be the corresponding subfunctor of $\eF^n$.  

    Since $NH_n$ is a matrix algebra of rank $n!$, we have that $\Delta_{\beta}^{n}\cong (\Delta_{\beta}^{(n)})^{\oplus n!}$ and $\eF^{n}\cong (\eF^{(n)})^{\oplus n!}$.
\end{defn}

\subsection{The property (KM3)}
Thus, only (KM3) remains to check.

Consider a module $M$ over $\R{\PR}{\Lambda}{\alpha}$.
Since $\beta$ is positive $\PR$-simple, the induction $M\circ \Delta_{\beta}^{(k)}$ is a standardization.
This is not the case for $\Delta_{\beta}^{(k)}\circ M $;  however, we claim that this module has a ${\Delta}$-flag.  We can find this by considering the submodule $N_a$ generated by the diagram where we pull the left $a$ groups for $\beta$ right over the strands in $M$, so $N_{k,0}= \Delta_{\beta}^{(k)}\circ M   $ and $N_{k,k+1}=0$.  

\begin{lem}\label{lem:N-filtration}
If $M$ is projective,  $N_{k,a}/N_{k,a+1}\cong \eF^{(k-a)}M\circ \Delta_{\beta}^{(a)}$. 
\end{lem}
\begin{proof}
Consider $\mathbb{R}\Hom(\Delta_{\beta}^{(k)}\circ M,\nabla(\boldsymbol{\gamma}))$.  Note that this is 0 unless $\boldsymbol{\gamma} $ has the form $\gamma_0=\al+(k-b)\beta, \gamma_{1}=b\beta$, since otherwise $\operatorname{Res}_{\boldsymbol{\gamma}}(\Delta_{\beta}^{(k)}\circ M)=0$.  

In the case of $\gamma_0=\al+(k-b)\beta, \gamma_{1}=b\beta$, we have a natural map  $\Delta_{\beta}^{(k-b)}\circ M\otimes \Delta^{(b)}\to \operatorname{Res}_{\boldsymbol{\gamma}}(\Delta_{\beta}^{(k)}\circ M)$ given by attachment to the diagram
\begin{center}
\begin{tikzpicture}[yscale=0.7]
    \rect{0}{0}{2.5}{1}
    \rect{3}{0}{1.5}{1}
    \node at (1.25,0.5) {$\Delta_\beta^{(k)}$};
    \node at (3.75,0.5) {$M$};
    \draw[thick] (0.25,1) -- (0.25,3);
    \draw[thick] (0.5,1) -- (0.5,3);
    \draw[thick] (0.75,1) -- (0.75,3);
    \draw[thick] (1,1) -- (1,3);
    \draw[thick] (1.25,1) -- (1.25,3);
    \draw[thick] (1.5,1) -- (1.5,3);
    \draw[thick] (1.75,1) to [out =60, in =240] (4.25,3);
     \draw[thick] (2,1) to [out =60, in =240] (4.5,3);
      \draw[thick] (2.25,1) to [out =60, in =240] (4.75,3);

     \draw[thick] (3.25,1) to [out =90, in =270] (2,3);
     \draw[thick] (3.5,1) to [out =90, in =270] (2.25,3);
     \draw[thick] (3.75,1) to [out =90, in =270] (2.5,3);
     \draw[thick] (4,1) to [out =90, in =270] (2.75,3);
     \draw[thick] (4.25,1) to [out =90, in =270] (3,3);
\end{tikzpicture}
\end{center}
which is an isomorphism.  
Thus, 
\begin{multline*}
\Ext^k_{\Rz{\Lambda}{\alpha+k\beta}}(\Delta_{\beta}^{(k)}\circ M, \nabla(\boldsymbol{\gamma}))\cong \Ext^k_{\Rg{\boldsymbol{\gamma}}}(\Delta_{\beta}^{(k-b)}\circ M\otimes \Delta^{(b)}, \Cg{\boldsymbol{\gamma}}^*)\\
\cong \Ext^k_{\Rz{\Lambda}{\alpha+(k-b)\beta}}(\Delta_{\beta}^{(k-b)}\circ M, \R{\PR}{\Lambda}{\alpha+(k-b)\beta}^*)\otimes ( \Delta^{(b)})^*
\end{multline*}
By adjunction, we have 
\begin{multline*}
\Ext^k_{\Rz{\Lambda}{\alpha+(k-b)\beta}}(\Delta_{\beta}^{(k-b)}\circ M, \R{\PR}{\Lambda}{\alpha+(k-b)\beta}^*)\\\cong \Ext^k_{\Rz{\Lambda}{\alpha}}(M,\eE^{(k-b)}\R{\PR}{\Lambda}{\alpha+(k-b)\beta}^*)   \\
\cong \Ext^k_{\R{\PR}{\Lambda}{\alpha}}(M,\eE^{(k-b)}\R{\PR}{\Lambda}{\alpha+(k-b)\beta}^*) 
\end{multline*}
where the last isomorphism follows from \cref{lem:Tor-flag}.  Since $M$ is projective, this Ext-group vanishes for $k>0$.

By adjunction, this means that $ \Hom_{\Rz{\Lambda}{\alpha+k\beta}}(\Delta_{\beta}^{(k)}\circ M, \nabla(\boldsymbol{\gamma}))\cong (\eF^{(k-b)}M\otimes \Delta_{\beta}^{(b)})^*$.
Thus, \cref{thm:standard-standard} implies that $\Delta_{\beta}^{(k)}\circ M $  is filtered by the standardizations $ \eF^{(k-a)}M \circ\Delta_{\beta}^{(a)}$ for $a=0,1,\dots, k$.

    If we act on the generating element of $N_a$, this gives us a map $ \tilde{\eF}^{(k-a)}M\circ\Delta_{\beta}^{(a)}\to N_{k,a}$.  Furthermore, any element in the kernel of $\tilde{\eF}^{(k-a)}M\to {\eF}^{(k-a)}M$ must give an element of $N_{a+1}$.  Thus, we have a surjective map $ \eF^{(k-a)}M\circ \Delta_{\beta}^{(a)}\to N_a/N_{a+1}.$  
    
    Since we already know that a standard filtration by these modules exists, the dimension $\Delta_{\beta}^{(k)}\circ M $ in any graded degree is the sum of the graded dimensions of $\eF^{(k-a)}M\circ \Delta_{\beta}^{(a)}$; on the other hand, the same is true of the graded dimensions of $N_a/N_{a+1}$.  This is only possible if all the surjective maps $\eF^{(k-a)}M\circ \Delta_{\beta}^{(a)}\to N_a/N_{a+1}$ are isomorphisms.  
\end{proof}
\begin{lem}\label{lem:standard-resolution}
If $M$ is projective, we have a resolution 
	\begin{equation}\label{eq:divided power sequence}
	0\to M \circ \Delta_{\beta}^{(n)}  \to  \cdots \overset{\partial}\to \Delta_{\beta}^{(n-1)}\circ M \circ \Delta_{\beta} \to \Delta_{\beta}^{(n)}\circ M \to \eF^{(n)}M
	\end{equation}
 where the map $\partial$ acts on elements of $\Delta_{\beta}^{(n-k)}\circ M \circ \Delta_{\beta}^{(k)}$ attaching to the diagram     
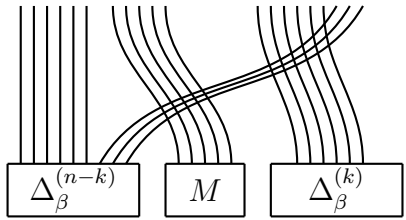
\begin{center}
\begin{tikzpicture}[scale=0.7]
    \rect{0}{0}{2.5}{1}
    \rect{3}{0}{1.5}{1}
    \rect{5}{0}{2.5}{1}
    \node at (1.25,0.5) {$\Delta_\beta^{(n-k)}$};
    \node at (3.75,0.5) {$M$};
        \node at (6.25,0.5) {$\Delta_\beta^{(k)}$};
    \draw[thick] (0.25,1) -- (0.25,4);
    \draw[thick] (0.5,1) -- (0.5,4);
    \draw[thick] (0.75,1) -- (0.75,4);
    \draw[thick] (1,1) -- (1,4);
    \draw[thick] (1.25,1) -- (1.25,4);
    \draw[thick] (1.5,1) -- (1.5,4);
    \draw[thick] (1.75,1) to [out =60, in =240] (6.25,4);
     \draw[thick] (2,1) to [out =60, in =240] (6.5,4);
      \draw[thick] (2.25,1) to [out =60, in =240] (6.75,4);

     \draw[thick] (3.25,1) to [out =90, in =270] (2,4);
     \draw[thick] (3.5,1) to [out =90, in =270] (2.25,4);
     \draw[thick] (3.75,1) to [out =90, in =270] (2.5,4);
     \draw[thick] (4,1) to [out =90, in =270] (2.75,4);
     \draw[thick] (4.25,1) to [out =90, in =270] (3,4);
     
      \draw[thick] (5.5,1) to [out =90, in =270] (4.75,4);
     \draw[thick] (5.75,1) to [out =90, in =270] (5,4);
     \draw[thick] (6,1) to [out =90, in =270] (5.25,4);
     \draw[thick] (6.25,1) to [out =90, in =270] (5.5,4);
     \draw[thick] (6.5,1) to [out =90, in =270] (5.75,4);
     \draw[thick] (6.75,1) to [out =90, in =270] (6,4);
\end{tikzpicture}
\end{center}
\end{lem}
\begin{proof}
 Consider the filtration on $\Delta_{\beta}^{(n-k)}\circ M \circ \Delta_{\beta}^{(k)} $ induced by \cref{lem:N-filtration}.  This filtration has subquotients
  \[(N_{n-k,a}\circ \Delta_{\beta}^{(k)})/( N_{n-k,a+1}\circ\Delta_{\beta}^{(k)})\cong \eF^{(n-k-a)}M\circ\Delta_{\beta}^{(a)}\circ \Delta_{\beta}^{(k)}\cong ( \eF^{(n-k-a)}M\circ\Delta_{\beta}^{(k+a)})^{\oplus \binom{k+a}{a}}.\]
  We can realize the latter isomorphism by shuffling the groups of strands for the different copies of $\Delta_{\beta}$.  
  The map $\partial$ sends $ N_{n-k-1,a-1}\circ\Delta_{\beta}^{(k+1)}$ to $N_{n-k,a}\circ \Delta_{\beta}^{(k)}$, and thus we can consider the induced complex on the associated graded, where each term is a sum of copies of $\eF^{(n-k-a)}M\circ \Delta_{\beta}^{(k+a)}$.  Note that by the usual formula $\binom{k+a}{a}=\binom{k+a-1}{a-1}+\binom{k+a-1}{a}$, we have that this complex is exact if the image of the map \[\partial\colon( N_{n-k-1,a-1}\circ\Delta_{\beta}^{(k+1)})/( N_{n-k-1,a}\circ\Delta_{\beta}^{(k+1)})\to  ( N_{n-k,a}\circ\Delta_{\beta}^{(k)})/(N_{n-k,a+1}\circ \Delta_{\beta}^{(k)})\] 
  is a summand isomorphic to $\binom{k+a-1}{a-1}$ copies of $\eF^{(k-a)}M\circ \Delta_{\beta}^{(k+a)}$.  Indeed this image is exactly the summand corresponding to the shuffles where the leftmost element of the shuffle is from the set of size $a$.  This must be precisely the kernel since we know that a complementary summand must map injectively  to $ ( N_{n-k+1,a+1}\circ\Delta_{\beta}^{(k-1)})/( N_{n-k+1,a+2}\circ\Delta_{\beta}^{(k-1)})$, by our calculation of the size of the image.  
  
  Since a complex which is exact after taking associated graded is exact, this completes the proof.
\end{proof}
An important special case of this result is that for $M$ projective, we have 
    \begin{equation}
       \label{eq:n=1 sequence} M\circ \Delta_{\beta}\overset{X}\to \Delta_{\beta}\circ M\to \eF M\to 0,
    \end{equation}  
    where $X$ is the diagram  
\begin{center}
\begin{tikzpicture}
    \node at (-1,0) {$X=$};
    \permd{0}{0}{2}{}{}
    \permd{.25}{0}{2}{}{}
    \permd{0.5}{0}{2}{}{}
    \permd{1.75}{0}{-2}{}{}
    \permd{2}{0}{-2}{}{}
    \permd{2.25}{0}{-2}{}{}
    \rect{-0.25}{-1}{1}{1}
    \rect{1.5}{-1}{1 
    }{1}
    \node at (0.25,-0.5) {$M$};
    \node at (2,-0.5) {$\Delta_\beta$};
\end{tikzpicture}    
\end{center}
It follows immediately from  \cite[Lem. 4.18]{kangCategorificationHighest2012} that:
\begin{lem}
    The module $\eE \R{\PR}{\Lambda}{\alpha}$ is projective as a $\R{\PR}{\Lambda}{\alpha+\beta}$-module for all $\alpha$.  In particular, the functor $\eF$ is exact.  
\end{lem}
It's easy to conclude from this that \cref{lem:N-filtration}, \cref{lem:standard-resolution} and \eqref{eq:n=1 sequence} hold even if $M$ is not projective, since all functors appearing are exact and all complexes are natural in $M$.

Of course, $\eE$ is already known to be exact, so both functors induce maps on the Grothendieck group $K_0(\Ccat)$.  We wish to prove that these satisfy the relations of $\mathfrak{sl}_2$.

\begin{lem}\label{lem:sl2-rels}
    Assume that $M$ is a simple $\R{\PR}{\lambda}{\alpha}$-module such that $\eE(M)=0$ and let $k=\langle\Lambda-\alpha,\beta^{\vee}\rangle $.  Then $\eE \eF^{(n)}(M)=\eF^{(n-1)}(M)^{\oplus k-n+1}$ if $n\leq k$ and $\eF^n(M)=0$ if $n>k$.  
\end{lem}

In order to describe this isomorphism, let us first discuss the corresponding result for the nilHecke algebra $NH_n^k$ on $n$ strands with a cyclotomic relation of degree $k$.  The submodule
\[NH_{(n)}^k=\{a\in NH_n^k\mid a\psi_i=0 \text{ for all } i\}\] is generated by the half-twist element, which has minimal degree.  We can think of this as a thick strand splitting into $n$ thin strands.  This projective module is generated over $\K[y_1,\dots, y_n]$ by this element, with the kernel generated by the complete symmetric polynomials $h_{k-n+m}(y_1,\dots, y_n)$ for $m=1,\dots,n$.  Thus restriction is just restricting this ring to the subring $\K[y_1,\dots, y_{n-1}]$.  The elements $1,y_n,\dots, y_{n}^{k-n}$ are a free basis of this module as an $NH_{n-1}^k$-module as is easily seen using the formula
\begin{equation}
	\label{eq:complete-symmetric}
h_{m}(y_1,\dots, y_n)=h_{m}(y_1,\dots, y_{n-1})+h_{m-1}(y_1,\dots, y_{n-1})y_n+\cdots +y_{n}^m.
\end{equation}

\begin{proof}
First consider the case $n=1$.  In this case, we apply $\eE$ to the sequence of \cref{lem:standard-resolution}.  The result is a complex of modules of $\K[y]\otimes \R{\PR}{\lambda}{\alpha} $ given by 
\begin{equation}
    \K[y] \otimes M \overset{X^*X}{\longrightarrow} \K[y] \otimes M\longrightarrow \eE\eF M\to 0
\end{equation}
for the diagram
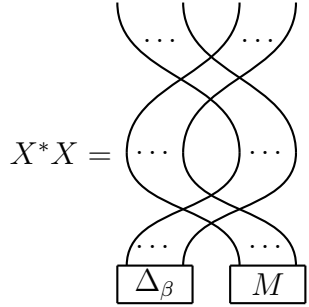
\begin{center}
  \begin{tikzpicture}[scale=0.5]
\node at (-1.5,4) {$X^*X=$};
    \rect{0}{0}{2}{1} \rect{3}{0}{2}{1}
    \node at (1,0.5) {$\Delta_\beta$};
    \node at (4,0.5) {$M$};
       \draw[thick] (0.25,1) to [out =90, in =270] (3.25,4);
     \draw[thick] (1.75,1) to [out =90, in =270] (4.75,4);
     \draw[thick] (3.25,1) to [out =90, in =270] (0.25,4);
     \draw[thick] (4.75,1) to [out =90, in =270] (1.75,4);

  \draw[thick] (3.25,4) to [out =90, in =270] (0,8);
     \draw[thick] (4.75,4) to [out =90, in =270] (1.75,8);
     \draw[thick] (0.25,4) to [out =90, in =270] (3.25,8);
     \draw[thick] (1.75,4) to [out =90, in =270] (4.75,8);
\node at (1 ,1.5 ) {\small{$\dots$}};
\node at (4 ,1.5 ) {\small{$\dots$}};
\node at (1 ,4 ) {\small{$\dots$}};
\node at (4 , 4) {\small{$\dots$}};
\node at (1.2 , 7) {\small{$\dots$}};
\node at ( 3.8, 7) {\small{$\dots$}};
\end{tikzpicture}
\end{center}
Since $M$ is simple, $\End_{\K[y]\otimes \R{\PR}{\lambda}{\alpha}}(M\otimes \K[y] )=\K[y]$, so $XX^*$ must act by a polynomial which is non-zero (by the exactness of \cref{lem:standard-resolution}) and homogeneous of degree $k$  in $y$, since $X$ has degree $k$ in $\Rz{\Lambda}{\alpha+\beta}$, and $y$ is degree 2.  That is, it is of the form $ay^k$, for some $a\in \K$.  
  This implies that $\eE\eF M\cong M\otimes \K[y]/(y^k)\cong M^{\oplus k}$, as desired.  

Now, consider the case where $k\geq n>1$.  We can again apply $\eE$ to the sequence \cref{lem:standard-resolution}.  We obtain that $\eE \eF^{(n)}(M) $ is the quotient of $\eE(\Delta_{\beta}^{(n)}\circ M)\cong \K[y]\otimes \Delta_{\beta}^{(n-1)}M $ by the image of the map from $\eE(\Delta_{\beta}^{(n-1)}\circ M \circ \Delta_{\beta}) $. 

Let $U\subset \K[y]$ be the span of $\{1,\dots, y^{n-k}\}$; this is an $n-k+1$-dimensional subspace.  We have an induced natural map $U\otimes_{\K}\Delta_{\beta}^{(n-1)}\circ M \to \eE(\Delta_{\beta}^{(n)}\circ M) $.  When we compose with the quotient map to $\eE \eF^{(n)}(M) $, we obtain a map $U\otimes_{\K}\eF^{(n-1)}(M) \to \eE \eF^{(n)}(M) $.  This map is surjective, since standard nilHecke algebra calculations show that $h_m(y_1,\dots, y_n)=0$ for $m> k-n$,  and so \eqref{eq:complete-symmetric} shows how to write $y_n^m$ in terms of lower powers of $y_n$.  

Now, assume that we have an element of the kernel.  The image of any lift of this element to  $U\otimes_{\K}\Delta_{\beta}^{(n-1)}\circ M$ must map to the image of the map from $\eE(\Delta_{\beta}^{(n-1)}\circ M \circ \Delta_{\beta}) $.  Any diagram in this space must cap off either the rightmost set of strands, or a subset of the leftmost.  In the latter case, this diagram maps to 0 in $\K[y]\otimes \Delta_{\beta}^{(n-1)}M $, so we can subtract it from the lift, and assume that the image is a sum of diagrams of the former type.  That is, it is in the submodule generated by 
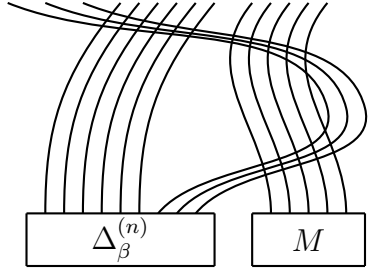
\begin{center}
\begin{tikzpicture}[yscale=0.7]
    \rect{0}{0}{2.5}{1}
    \rect{3}{0}{1.5}{1}
    \node at (1.25,0.5) {$\Delta_\beta^{(n)}$};
    \node at (3.75,0.5) {$M$};
    \draw[thick] (0.25,1) to[out =90, in =240] (1.25,5);
    \draw[thick] (0.5,1) to[out =90, in =240] (1.5,5);
    \draw[thick] (0.75,1) to[out =90, in =240] (1.75,5);
    \draw[thick] (1,1) to[out =90, in =240] (2,5);
    \draw[thick] (1.25,1) to[out =90, in =240] (2.25,5);
    \draw[thick] (1.5,1) to[out =90, in =240] (2.5,5);
    \draw[thick] (1.75,1) to [out =60, in =280] (4,3) to [out=100,in=-30] (-.25,5);
     \draw[thick] (2,1) to [out =60, in =280] (4.25,3) to [out=100,in=-30](.25,5);
      \draw[thick] (2.25,1) to [out =60, in =280] (4.5,3) to [out=100,in=-30] (.75, 5);
     \draw[thick] (3.25,1) to [out =90, in =240] (3,5);
     \draw[thick] (3.5,1) to [out =90, in =240] (3.25,5);
     \draw[thick] (3.75,1) to [out =90, in =240] (3.5,5);
     \draw[thick] (4,1) to [out =90, in =240] (3.75,5);
     \draw[thick] (4.25,1) to [out =90, in =240] (4,5);
\end{tikzpicture}
\end{center}
After simplification, this is the submodule generated by the complete symmetric polynomial $h_{n-k}(y_1,\cdots, y_n)$.  This ideal in the polynomials on $\C[y_1,\dots, y_n]$ is complementary to the subspace $U\otimes \C[y_1,\cdots, y_{n-1}]$, so this shows that the element of the kernel must be 0 and the map is injective.  
\end{proof}

\begin{cor}
For any $\PR$-simple root $\beta$, the property (KM3) holds.
\end{cor}
\begin{proof}
    By \cite[Th. 5.27]{CR}, it suffices to show that the maps $E=[\eE]$ and $F=[\eF]$ on $K_0(\Ccat)$ satisfy the relations of $\mathfrak{sl}_2$, with $H=[E,F]$ satisfying $H|_{\Ccat_{k}}=k\cdot \operatorname{Id}$.  

    If $M\in \Ccat_k$ and $\eE M=0$, then by \cref{lem:sl2-rels}, the operators $[\eE]$ and $[\eF]$ act on the span of the vectors $[\eF^{(n)}M]$ for $n=0,\dots, k$ by the structure coefficients of the usual $\mathfrak{sl}_2$-representation with highest weight $k$.  Thus, on the span $V$ of these vectors, the operators $[\eE]$ and $[\eF]$ on $K_0(\Ccat)$ satisfy the relations of $\mathfrak{sl}_2$.

    The only question is if $V$ is all of $K_0(\Ccat)$.  Assume not, and let $k$ be the largest weight such that there is a vector in $K_0(\Ccat_k)$ that is not in $V$.  Since the classes of simple modules are a basis for the Grothendieck group, we can assume that this vector is $[L]$ for $L\in \Ccat_k$ simple.  By assumption, vectors of the form $[\eF^nM]$ for $\eE M=0$ span $K_0(\Ccat_{k'})$ for all $k'>k$.  In particular, if we consider the category $\Ccat'$ of objects presented by projectives that are summands of $\eF^nP$ for $P\in \Ccat_{k'}$ with $k'>k$, the operators $[\eE]$ and $[\eF]$ satisfy the $\mathfrak{sl}_2$-relations on the Grothendieck group of this subcategory, and thus this subcategory has a strong $\mathfrak{sl}_2$-action. 

    In particular, let $h^+>0$ be the maximal integer such that $\eE^{(h^+)}(L)\neq 0.$  We can assume that we have chosen $L$ to minimize this statistic.  Let $L'$ be a simple subobject in $\eE^{(h^+)}(L)$.  By exactness, $\eE L'=0$.  By \cite[Prop. 5.20]{CR}, the module $\eF^{(h_+)}L'$ has a unique simple quotient, this quotient must be $L$ and  $[\eF^{(h_+)}L']=\binom{k+2h_+}{h_+}[L]+[M']$ where $\eE^{h_+}M'=0$.  By our assumption of minimality, this means that $[M']\in V$, and of course, $[\eF^{(h_+)}L']\in V$ by construction, so $[L]\in V$, contradicting our assumption that it was not.  

    Thus, we have $V=K_0(\Ccat)$ and we have the desired categorical action.
\end{proof}
This completes the proof of \cref{sl2action}.

\subsection{Morita equivalence}
We can now apply the usual results on derived equivalences associated to $\mathfrak{sl}_2$ actions.  It is useful to consider the filtration of $\Ccat$ by the Serre subcategories 
\begin{equation}\label{eq:Serre-subcats}
    \Ccat(m)=\bigoplus_{k\in\Z}\Ccat_k(m)\qquad  \Ccat_k(m)=
\begin{cases}
    \{M\mid \eE^{(\lceil (m-k+1)/2\rceil)}M=0\} &|k|\leq m \\
    0 & |k|>m
\end{cases}
\end{equation}
The subcategory $\Ccat(m)$ is the largest Serre subcategory closed under the action of $\eE$ and $\eF$ whose intersection with $\Ccat_k$ for all $k>m$ is trivial.  This is dual to the filtration constructed in \cite[Th. 5.8]{R} on the category of projectives in $\Ccat$.  We let $\Ccat'$ denote the category defined in \eqref{eq:Ccat-def} for the pressure $s_{\beta}\PR$, and let $\Ccat_k'(m)$ be the corresponding subcategories as in \eqref{eq:Serre-subcats}.

As before, let $k=\langle\Lambda-\alpha,\beta^{\vee}\rangle $ and consider the complex of functors
\begin{equation}\label{Theta-def}
    \Theta'=\begin{cases}
    \eF^{(k)}\to \eE\eF^{(k+1)} \to \eE^{(2)}\eF^{(k+2)} \to \cdots   & k\geq 0\\
      \eE^{(-k)}\to \eF\eE^{(-k+1)} \to \eF^{(2)}\eE^{(-k+2)} \to \cdots   & k\leq 0
\end{cases}
\end{equation}
In both cases, the left-most term is in homological degree 0.  We can think of these complexes of exact functors as functors between the corresponding derived categories.  
By \cite[Th. 6.4]{CR} and \cite[Prop. 8.4]{chuangPerverseEquivalences}:
\begin{lem}
    The complex \eqref{Theta-def} defines a derived equivalence
    \[\CR{\alpha}{s_{\beta} \Ldot{\Lambda}\alpha }\colon D^b(\R{\PR}{\Lambda}{\alpha}\mmod)\to D^b(\R{\PR}{\Lambda}{s_{\beta} \Ldot{\Lambda}\alpha }\mmod )\]
On the filtration $\mathcal{A}_{k,i}=\Ccat_k(|k|+2i)$, the equivalence $\CR{\alpha}{s_{\beta} \Ldot{\Lambda}\alpha }$ is perverse with respect to the perversity function $p(i)=i$.  
\end{lem}
This equivalence is an important step toward the Morita equivalence of \Cref{conj}.

 We can always consider the algebras $\R{\PR}{\Lambda}{\alpha}$ and $\R{s_{\beta} \PR}{\Lambda}{\alpha}$ as quotients of the algebra $\mathcal{R}_0(\Lambda, \alpha)$.  
Let  \[\bR{\PR}{s_{\beta}\PR}{\Lambda}{\alpha}=\R{\PR}{\Lambda}{\alpha}\Lotimes_{\mathcal{R}_0(\Lambda, \alpha)}  \R{s_{\beta} \PR}{\Lambda}{\alpha}.\]  This is a complex of $\R{\PR}{\Lambda}{\alpha}\operatorname{-}\R{s_{\beta} \PR}{\Lambda}{\alpha}$-bimodules.  
By \cref{lem:Tor-flag}, we have that $\bR{\PR}{\PR}{\Lambda}{\alpha}=\R{\PR}{\lambda}{\alpha}$.   Derived tensor product with $\bR{\PR}{s_{\beta}\PR}{\Lambda}{\alpha}$ defines a functor 
    \[\tw{s_{\beta}\PR}{\PR}\colon D^b(\R{s_{\beta}\PR}{\Lambda}{\alpha}\mmod)\to D^b(\R{\PR}{\Lambda}{\alpha}\mmod).\]

\begin{lem}\label{lem:shift}
    If $\eE M=0$, then $\tw{s_{\beta}\PR}{\PR}(\eF^{(m)}M)=\eF^{(m)}M[-m]$.
\end{lem}
\begin{proof}
    This follows immediately from \eqref{eq:divided power sequence}. This shows  $\R{\PR}{\Lambda}{\alpha}\Lotimes_{\mathcal{R}_0(\Lambda, \alpha)} \eF^{(m)}M$ can be calculated by the sequence \[	\R{\PR}{\Lambda}{\alpha}\Lotimes_{\mathcal{R}_0(\Lambda, \alpha)}(M \circ \Delta_{\beta}^{(m)})  \to  \cdots \to \R{\PR}{\Lambda}{\alpha}\Lotimes_{\mathcal{R}_0(\Lambda, \alpha)}(\Delta_{\beta}^{(m)}\circ M ).\]
    All of the terms are $\bar{\Delta}$-filtered, and thus have no higher Tor's with $\R{\PR}{\Lambda}{\alpha}$ by \cref{lem:Tor-flag}.  Furthermore, $\R{\PR}{\Lambda}{\alpha}\Lotimes_{\mathcal{R}_0(\Lambda, \alpha)}(\Delta_{\beta}^{(a)}\circ M \circ \Delta_{\beta}^{(m-a)})=0$ if $a>0$, so only the last term contributes.  The result follows.  
\end{proof}

\begin{thm}\label{th:Xi-equivalence}
    The functor $\tw{s_{\beta}\PR}{\PR}$ is a derived equivalence which is perverse for the filtration by the subcategories $\mathcal{A}_{k,i}=\Ccat_k(|k|+2i)$ and  $\mathcal{A}_{k,i}'=\Ccat_k(|k|+2i)$ with perversity $p(i)=-i$.  
\end{thm}
\begin{proof}
    Let $\tw{s_{\beta}\PR}{\PR}'(M)=\mathbb{R} \Hom(\bR{\PR}{s_{\beta}\PR}{\Lambda}{\alpha},M)$ be the right adjoint functor to $\tw{s_{\beta}\PR}{\PR}$.    We thus have a  unit map $\eF^{(m)}M\to \tw{s_{\beta}\PR}{\PR}'\tw{s_{\beta}\PR}{\PR}(\eF^{(m)}M) $.  This map is given by tensoring with the identity map on $\bR{\PR}{s_{\beta}\PR}{\Lambda}{\alpha}$ inducing a map 
    \[\eF^{(m)}M\to \mathbb{R} \Hom(\bR{\PR}{s_{\beta}\PR}{\Lambda}{\alpha},\R{\PR}{\Lambda}{\alpha}\Lotimes_{\mathcal{R}_0(\Lambda, \alpha)} \eF^{(m)}M).\]
   By a dual calculation to \cref{lem:shift}, we have that if $\eE M=0$ then $\tw{s_{\beta}\PR}{\PR}'(\eF^{(m)}M)\cong \eF^{(m)}M[m]$ and so $\tw{s_{\beta}\PR}{\PR}'\tw{s_{\beta}\PR}{\PR}(\eF^{(m)}M) \cong \eF^{(m)}M$.  Since $\tw{s_{\beta}\PR}{\PR}'\tw{s_{\beta}\PR}{\PR}$ sends the identity on this object to the identity, the unit of the adjunction must be injective; since $\eF^{(m)}M$ is finite-dimensional,  this show that the map is an isomorphism if $\eE M=0$.  
   
    Since every simple is a composition factor of such a module, this shows that $\tw{s_{\beta}\PR}{\PR}'\tw{s_{\beta}\PR}{\PR} M\cong M$ for all $M$.  Thus, the functor $\tw{s_{\beta}\PR}{\PR}$ is a derived equivalence.

    The category $\mathcal{A}_{i}/\mathcal{A}_{i-1}$ is generated by the objects $\eF^{(i)}M$ for $M\in \Ccat_{k+2i}$ which satisfy $\eE M=0$, so \cref{lem:shift} shows that $\tw{s_{\beta}\PR}{\PR}[i]$ is exact on this subquotient.  
\end{proof}

\begin{thm}\label{conj2}
    If $k=\langle \Lambda-\alpha,\beta^{\vee}\rangle \geq 0$, the compositions 
    \[\CR{\alpha}{s_{\beta}\Ldot{\Lambda}\alpha}\circ \tw{s_{\beta}\PR}{\PR}\colon D^b(\R{s_{\beta}\PR}{\Lambda}{\alpha}\mmod)\to D^b(\R{\PR}{\Lambda}{s_{\beta}\Ldot{\Lambda}\alpha}\mmod)\] 
     \[\tw{s_{\beta}\PR}{\PR}\circ \CR{s_{\beta}\Ldot{\Lambda}\alpha}{\alpha} \colon D^b(\R{s_{\beta}\PR}{\Lambda}{s_{\beta}\Ldot{\Lambda}\alpha}\mmod)\to D^b(\R{\PR}{\Lambda}{\alpha}\mmod)\] 
     are exact in the usual $t$-structure and thus induce Morita equivalences.  That is, for any sign of $k$, we have a Morita equivalence $\R{s_{\beta}\PR}{\Lambda}{s_{\beta}\Ldot{\Lambda}\alpha}\sim \R{\PR}{\Lambda}{\alpha}$.
\end{thm}
\begin{proof}
The functor $\CR{\alpha}{s_{\beta}\Ldot{\Lambda}\alpha}\circ \tw{s_{\beta}\PR}{\PR}$ defines an equivalence $D^b(\Ccat_k')\to D^b(\Ccat_{-k})$. The equivalence $\tw{s_{\beta}\PR}{\PR}$ is perverse for the chains of subcategories $\mathcal{A}'_{k,i},\mathcal{A}_{k,i}$ for $p(i)=-i$ and $\CR{\alpha}{s_{\beta}\Ldot{\Lambda}\alpha} $ is perverse for the chains $\mathcal{A}_{k,i},\mathcal{A}_{-k,i}$ for $p'(i)=-i$.  By \cite[Lem. 4.4]{chuangPerverseEquivalences}, this means that $\CR{\alpha}{s_{\beta}\Ldot{\Lambda}\alpha}\circ \tw{s_{\beta}\PR}{\PR}$ is perverse with respect to the chains of subcategories $\mathcal{A}'_{k,*},\mathcal{A}_{-k,*}$ with the perversity function $p+p'(i)=0$.  By \cite[Lem. 4.16]{chuangPerverseEquivalences}, this implies that $\CR{\alpha}{s_{\beta}\Ldot{\Lambda}\alpha}\circ \tw{s_{\beta}\PR}{\PR}$ is $t$-exact.

The argument for $ \tw{s_{\beta}\PR}{\PR}\CR{s_{\beta}\Ldot{\Lambda}\alpha}{\alpha}$ is analogous.   
\end{proof}

\begin{proof}[Proof of \Cref{conj}]
Consider $\R{\PR}{\Lambda}{\alpha}$ and $\R{w\cdot \PR}{\Lambda}{w\Ldot{\Lambda} \alpha}$.  By choosing a gallery that joins the alcoves containing $\PR$ and $w\cdot \PR$, we can obtain a reduced factorization $w=s_{\beta_h}\cdots s_{\beta_1}$ such that $\beta_{g+1}$ is $\PR$-simple for $s_{\beta_g}\cdots s_{\beta_1}\PR$.  
    By \Cref{conj2}, we have a chain of Morita equivalences:
    \begin{multline*}
        \R{\PR}{\Lambda}{\alpha}\sim \R{s_{\beta_1}\PR}{\Lambda}{s_{\beta_1}\Ldot{\Lambda}\alpha}\sim \cdots\\ \sim\R{s_{\beta_{h-1}}\cdots s_{\beta_1}\PR}{\Lambda}{s_{\beta_{h-1}}\cdots s_{\beta_1}\Ldot{\Lambda}\alpha}\sim \R{w\cdot \PR}{\Lambda}{w\Ldot{\Lambda} \alpha}.
    \end{multline*}
\end{proof}
Since it will be important in \cref{sec:RoCK}, let us emphasize that the result holds over an arbitrary base ring, in particular, over $\Z$.  The algebra  $R^{\Lambda}_{\alpha}$ is known to be free over $\K$ for all $\Lambda,\alpha$ by \cite[Rmk. 4.20]{kangCategorificationHighest2012}. Since freeness over a base ring is Morita invariant, it follows that:
\begin{cor}\label{cor:R-free}
    The algebra $\R{\PR}{\Lambda}{\alpha}$ is free of finite rank as a $\K$-module.   
\end{cor}

\subsection{Scopes chambers}
\label{sec:Scopes}

Let us return to understanding how the algebra $\R{\PR}{\Lambda}{\alpha}$ varies as $\PR$ changes.  These results depend on an elementary consequence of \cref{conj}:
\begin{cor}\label{cor:number of simples}  Assume that $\K$ is a field.
    The number of simple modules over $\R{\PR}{\Lambda}{\alpha}$ is independent of $\PR$. In particular:
    \begin{enumerate}
        \item We have $\R{\PR}{\Lambda}{\alpha}\neq 0$ if and only if $\Lambda-\alpha$ is a weight of the irrep $V(\Lambda)$ with highest weight $\Lambda$.  
        \item If $\Lambda-\alpha$ is dominant, then $\R{\PR}{\Lambda}{\alpha}\neq 0$.  
    \end{enumerate}
\end{cor}
Fix a dominant weight $\Lambda-\alpha$ and as in \cite[\S 3.1]{rock}, we let $N_{\lambda-\alpha}$ be the (finite) set of the positive roots $\beta$ such that the $\Lambda-\alpha+\beta$ weight space of $V(\Lambda)$ is not zero.
We define the {\bf Scopes walls} to be the vanishing sets of $\beta \in N_{\lambda-\alpha}$, thought of as a function on the space of pressures, and the {\bf Scopes chambers} to be the chambers cut out by this finite number of hyperplanes.  See \cite[\S 2.2]{rock} for more discussion of these walls and chambers.
\begin{thm}\label{th:Scopes-equal}
    Assume that $\Lambda-\alpha$ is dominant.  The ideals $\mathcal{I}(\chi)$ and $\mathcal{I}(\chi')$ are equal if and only if $\chi$ and $\chi'$   lie in the same Scopes chamber.
\end{thm}
Under \cref{conj}, this corresponds to \cite[Lem. 3.2]{rock}; however, in that result and other discussions of RoCK blocks in terms of alcoves, the appearance of alcove  has been purely a book-keeping device, whereas here, the element of the alcove is used directly in our definition of the algebra as the pressure.  
\begin{proof}
    First, assume that $\chi$ and $\chi'$ are in alcoves adjacent to each other across the vanishing hyperplane of a root $\beta$, which is necessarily $\chi$- and $\chi'$-simple.  By symmetry, we can assume that $\chi(\beta)>0 >\chi'(\beta)$.  Since all other roots have the same sign on $\chi$ and $\chi'$, we have $\mathcal{I}(\chi)\neq \mathcal{I}(\chi')$ if and only if there is a non-zero idempotent in $\R{\PR}{\Lambda}{\alpha}$ such that the leftmost strands have labels summing to $\beta$.  This is equivalent to asking that $\eE_{\beta}M\neq 0$ for some $\R{\PR}{\Lambda}{\alpha}$-module. If $\R{\PR}{\Lambda}{\alpha-\beta}=0$, then obviously $\eE_{\beta}M= 0$ for all modules, and if $\R{\PR}{\Lambda}{\alpha-\beta}\neq 0$, then $M=\eF_{\beta}\R{\PR}{\Lambda}{\alpha-\beta}$ shows that $\mathcal{I}(\chi)\neq \mathcal{I}(\chi')$.  Thus, indeed $\mathcal{I}(\chi)\neq \mathcal{I}(\chi')$ if and only $\beta$ defines a Scopes wall.
    
    Of course, applying this inductively shows that if $\chi,\chi'$ are in the same Scopes chamber, then $\mathcal{I}(\chi)= \mathcal{I}(\chi')$. If not, let $\beta$ be a Scopes wall separating $\chi$ and $\chi'$ which is adjacent to the Scopes chamber for $\chi$.  For simplicity, assume that $\chi(\beta)>0 >\chi'(\beta)$. If the signs are reversed, the argument is the same after we swap the roles of left and right.

    As above, we note that the fact that $\beta$ is a Scopes wall implies the module $M=\eF_{\beta}\R{\PR}{\Lambda}{\alpha-\beta}$ is not 0; by definition, multiplication by $\mathcal{I}(\chi)$ is 0 on this module.  On the other hand, by construction, this module is generated by the image of the idempotent $1_{\beta,\alpha}$, that is, by the image of idempotents where the leftmost strands sum to $\beta$. However, such an idempotent lies in $\mathcal{I}(\chi')$, so $\mathcal{I}(\chi')\cdot M=M$. This shows that in $\Rz{\Lambda}{\alpha}$, we have $\mathcal{I}(\chi')\neq \mathcal{I}(\chi)$.  
    \end{proof}

    There is one more redundancy implicit in this result which is especially important in the level one case: if $w\Ldot{\Lambda}\alpha=w'\Ldot{\Lambda}\alpha$, then o$\R{\PR}{\Lambda}{w\Ldot{\Lambda}\alpha}\cong \R{\PR}{\Lambda}{w'\Ldot{\Lambda}\alpha}$, so we have a Morita equivalence $\R{w^{-1}\PR}{\Lambda}{\alpha}\sim \R{(w')^{-1}\PR}{\Lambda}{\alpha}$, even though $w^{-1}\PR$ and $(w')^{-1}\PR$ may not be in the same Scopes chamber.  One way to fix this redundancy is to consider the action of the stabilizer $W_{\alpha}\subset \affweyl$ of $\alpha$ under this action; if we still assume that $\Lambda-\alpha$ is dominant, this is the parabolic subgroup generated by reflections for the simple coroots such that $\langle \Lambda-\alpha,\alpha_i^{\vee}\rangle=0$.  We must have $w'\in wW_{\alpha}$, and so we can make our choice of $w$ unique by assuming it has minimal length in this coset.  

    The second author encapsulated these equivalences of blocks in \cite{rock} in the notion of {\bf Scopes equivalence}.   Define a pressure by $\PRz(\alpha_i)=-1/e$; since every positive root is of negative pressure for $\PRz$, we have $\R{\PRz}{\Lambda}{\alpha}\cong R^{\Lambda}_{\alpha}$ is a cyclotomic quotient.  Note that $w$ is minimal length in $wW_{\alpha}$ if and only if $\al_i^{\vee}(w^{-1}\PRz)<0$ for all $s_i\in W_{\alpha}$.  
    
    We can then restate \cite[Th. A]{rock} for Ariki--Koike algebras as a corollary:
    \begin{cor}
        Assume that $w,w'$ are shortest left coset representatives for $W_{\alpha}$.  The blocks $R^{\Lambda}_{w\Ldot{\Lambda}\alpha}$ and $R^{\Lambda}_{w'\Ldot{\Lambda}\alpha}$ are Scopes equivalent if and only if $\mathcal{I}(w^{-1}\PRz)=\mathcal{I}((w')^{-1}\PRz)$.  
    \end{cor}
    Note that there are equivalences between blocks $R^{\Lambda}_{w\Ldot{\Lambda}\alpha}$ which are not induced by Scopes equivalences; for example, there is one induced by a diagram automorphism used in \cite{DA24}. 

\section{RoCK blocks of symmetric groups}
\label{sec:RoCK}
One particularly interesting case is when $\Lambda=\Lambda_0$ is a fundamental weight.   In this case, when $\K=\mathbb{F}_p$ and the quantum characteristic $e=p$ is prime, the cyclotomic quotient $R^{\Lambda_0}_{\alpha}\cong \R{\PRz}{\Lambda_0}{\alpha}$ is a block algebra of $\mathbb{F}_p\Sym_m$ for $m$ the total number of black strands; on the other hand, if $\K=\C$, then this is a block algebra of a Hecke algebra at a primitive $e$th root of unity.  

A basic fact about this case is that for a unique integer $d$, we have $\alpha=w\bullet_{\Lambda_0}d\delta$ for some $w\in \affweyl$.  
One can easily calculate that the stabilizer $W_{d\delta}=W$ is just the finite Weyl group, thus we will always 
take $w$ to be a shortest left coset representative for the finite Weyl group $W$, which makes $w$ unique as well.  

By \cref{conj}, we thus have a Morita equivalence of $R^{\Lambda_0}_{\alpha}$ to $\R{w^{-1}\PRz}{\Lambda_0}{d\delta}$. Since $w^{-1}$ is a shortest right coset representative, $w^{-1}\PRz(\al_i)<0$ unless $i=0$; unless $w=1$, we have $w^{-1}\PRz(\al_0)>0$.  
The Scopes walls in this case are given by $\PR(\beta)=d-1,d-2,\dots, -d+1$ for $\beta$ a root of the finite root system; see \cite[\S 4.1]{rock} for a discussion of this fact in this language, though the importance of this condition is clear already in \cite{scopesCartanMatrices1991}.   
There has been a particular focus on the unique Scopes chamber satisfying these conditions that is a translate of the anti-dominant Weyl chamber, defined by the inequalities $\chi(\al_i)<1-d$ for all $i\neq 0$.  The blocks $R^{\Lambda_0}_{w\bullet_{\Lambda_0}d\delta}$ such that $w^{-1}\PRz$ is in this Scopes chamber are called {\bf RoCK} or {\bf Rouquier}; all of these blocks are equivalent by \cite[Th. 4.2]{scopesCartanMatrices1991}, which is a construction of the Chuang-Rouquier equivalences in the case where they are exact {\it avant la lettre}.  For simplicity, we choose \begin{equation}
    \PR_{\mathsf{RoCK}}(\al_i)=\begin{cases}
    1-d-1/e & i\neq 0\\
    (e-1)(d-1)-1/e & i=0
\end{cases}
\end{equation}
as a representative element of this class.  
The algebra $\R{\PR_{\mathsf{RoCK}}}{\Lambda_0}{d\delta}$ thus provides a single, unique representative for the Morita equivalence class of RoCK blocks, the ``local model'' discussed in the introduction.

Since these blocks are already extensively studied, it is natural to consider how our results manifest here.  We'll show directly that $\R{\PR_{\mathsf{RoCK}}}{\Lambda_0}{d\delta}$ is Morita equivalent to the Turner double, originally shown to be Morita equivalent to a RoCK block $R^{\Lambda}_{w\Ldot{\Lambda_0}(d\delta)}$ by Evseev--Kleshchev \cite{EK2}.  

In particular, we draw a direct connection between the results of Chuang--Kessar \cite{CK}, showing that there is a projective $R^{\Lambda_0}_{w\bullet_{\Lambda_0}d\delta}$-module whose endomorphisms are isomorphic to the wreath product of $\Sym_d$ with the algebra of $R^{\Lambda_0}_{\delta}$ (which in the case $\K=\mathbb{F}_p$ and $e=p$ is prime is the principal block algebra of $\Sym_p$).  Furthermore, if $d<e$, then this projective generates the block and thus induces a Morita equivalence. 

In order to understand this connection, we consider the map 
$\iota\colon R_{d\delta}\to \R{\PR_{\mathsf{RoCK}}}{\Lambda_0}{d\delta}$ given by placing a single red strand on the left-hand side of a diagram in $R_{d\delta}$.  Let $e_+$ be the image of the identity under this map; this is exactly the sum of all idempotents where there are no black strands to the left of the red.  
\begin{lem}
    The homomorphism $\iota$ factors through the semi-cuspidal quotient $C_{d\delta}$ and surjects onto $e_+\R{\PR_{\mathsf{RoCK}}}{\Lambda_0}{d\delta}e_+$.
\end{lem}
\begin{proof}    
If an idempotent in $R_{d\delta}$ is of the form $(j_k,\dots, j_1)$ with $\al_{(j_p,\dots, j_{1})}=\beta$ a root such that $\PR_{\mathsf{RoCK}}(\beta)>0$ or equivalently, $\beta\prec \delta$, then its image in $\R{\PR_{\mathsf{RoCK}}}{\Lambda_0}{d\delta}$ is zero. 

Since there are no strands to the left of the red at the top or bottom of a diagram in $e_+\R{\PR_{\mathsf{RoCK}}}{\Lambda_0}{d\delta}e_+$, whenever we rewrite a diagram to be a sum of diagrams with a minimal number of crossings, we will avoid all crossings over the red strand, and thus lie in the image of $\iota$.
\end{proof} 

In \cite{EK2}, it's shown that a RoCK block is Morita equivalent to a quotient $C_{\rho,d}$ of $C_{d\delta}$;  this is not explicitly stated, but by \cite[Rk. 5.21]{EK2}, the algebra $C_{\rho,d}$ is an idempotent truncation of the block algebra for a RoCK block, and the Morita equivalence is established in the proof of \cite[Th. 8.28]{EK2}.  We'll reconfirm this more carefully in the proof of the following theorem, which will take up the next section of the paper:
\begin{thm}\label{th:EK-TD}
    The image $e_+\R{\PR_{\mathsf{RoCK}}}{\Lambda_0}{d\delta}e_+$ of $\iota$  is isomorphic to $C_{\rho,d}$ and this map induces a Morita equivalence $\R{\PR_{\mathsf{RoCK}}}{\Lambda_0}{d\delta}\sim C_{\rho,d}$, and thus a Morita equivalence to the Turner double $D(n,d)$ for $n>0$.
\end{thm}
Turner \cite[Ch. X]{Turner} and Kleshchev-Muth \cite[\S 7.9]{KMschurifying} both define quasi-hereditary covers $\mathcal{Q}(n,d)$ and $T^{Z}_{\mathfrak{z}}(n,d)$ of $D(n,d)$; it is unknown at the moment whether these quasi-hereditary covers are equivalent.    A natural extension of \cref{th:EK-TD} would be to construct a Morita equivalence of either or both of these covers with a steadied quotient incorporating a non-trivial weighting (as in \cite[Def. 2.22]{WebwKLR}); we hope to return to this question in future work.

\subsection{The proof of \cref{th:EK-TD}}
As in \cite{EK2}, we first consider the idempotents that correspond to Chuang--Kessar's projective.  Let $\Ba_i=(i,i+1,\dots, e-1,i-1,\dots, 0)$.  Let $\oned$ be the sum of all idempotents obtained by concatenating $d$ of these sequences.  This idempotent has the property that $C_{d\delta}\oned\cong \Delta_{d\delta}$, as noted in \cite[\S 5.4]{kleshchevAffineZigzag2018}.  

By \cite[Th. 6.16]{kleshchevAffineZigzag2018}, we have an isomorphism $\oned C_{d\delta} \oned\cong \Zdaff(\mathsf{ A}_{e-1}^{(1)})$
given by the local map shown in \Cref{fig:KM-iso}.
\begin{figure}
    \centering
    \begin{tikzpicture}[scale=0.8]
    \idm{0}{0}{i} 
    \draw[thick] [|->] (1,1) -- (3,1);
    \idm{4}{0}{i}
    \idm{5}{0}{i+1} \dcf{6}{0}{\dots}
    \idm{7}{0}{e-1}
    \idm{8}{0}{i}
    \dcf{9}{0}{\dots}
    \idm{10}{0}{0}
  \dbcf{9}{0}{\dots}
\didm{0}{-4}{i}{i+1}
\draw[thick] [->] (0,-2) -- (0,-3);
\draw[thick] [|->] (1,-3) -- (3,-3);
\permd{4}{-4}{1}{i+1}{i+1}
\dbcf{5.5}{-4}{\dots}
\dtcf{6.5}{-4}{\dots}
\permd{7}{-4}{1}{e-1}{e-1}
\permd{8}{-4}{-4}{i}{i}
\didm{10}{-4}{i-1}{i-1}
\dbcf{12}{-4}{\dots}
\dtcf{12}{-4}{\dots}
\didm{13}{-4}{0}{0}

\didm{0}{-8}{i}{i-1}
\draw[thick] [->] (0,-6) -- (0,-7);
\draw[thick] [|->] (1,-7) -- (3,-7);

\dbcf{6.5}{-8}{\dots}
\dtcf{5.5}{-8}{\dots}
\permd{4}{-8}{5}{i-1}{i-1}
\permd{5}{-8}{-1}{i}{i}
\permd{8}{-8}{-1}{e-1}{e-1}
\didm{11}{-8}{i-2}{i-2}
\dbcf{13}{-8}{\dots}
\dtcf{13}{-8}{\dots}
\didm{14}{-8}{0}{0}

\idm{-1}{-12}{i}
\point{(-1,-11)}
\dm{-0.5}{-12}
\idm{0}{-12}{i}
\draw[thick] [|->] (1,-11) -- (3,-11);
\didm{4}{-12}{i}{i}
\dcf{5}{-12}{\dots}
\didm{6}{-12}{0}{0}
\point{(4,-11)}

\idm{0}{-16}{i}
\point{(0,-15)}
\draw[thick] [|->] (1,-15) -- (3,-15);
\didm{4}{-16}{i}{i}
\dcf{5}{-16}{\dots}
\didm{6}{-16}{0}{0}
\point{(6,-15)}

\draw[thick] (-2,-20) -- (0,-18);
\draw[thick] (0,-20) -- (-2,-18);
\node at (-2,-20) [anchor=north] {\tiny{$i$}};
\node at (0,-20) [anchor=north] {\tiny{$j$}};
\draw[thick] [|->] (1,-19) -- (3,-19);

\permd{4}{-20}{5}{i}{i}
\dbcf{5}{-20}{\dots}
\dbcf{6}{-20}{e-1}
\dbcf{7}{-20}{i-1}
\dbcf{8}{-20}{\dots}
\permd{9}{-20}{5}{0}{0}
\permd{10}{-20}{-6}{j}{j}
\dbcf{11}{-20}{\dots}
\dbcf{12}{-20}{e-1}
\dbcf{13}{-20}{j-1}
\dbcf{14}{-20}{\dots}
\permd{14.5}{-20}{-6}{0}{0}

\end{tikzpicture}
\caption{The Kleshchev-Muth isomorphism}
    \label{fig:KM-iso}
\end{figure}
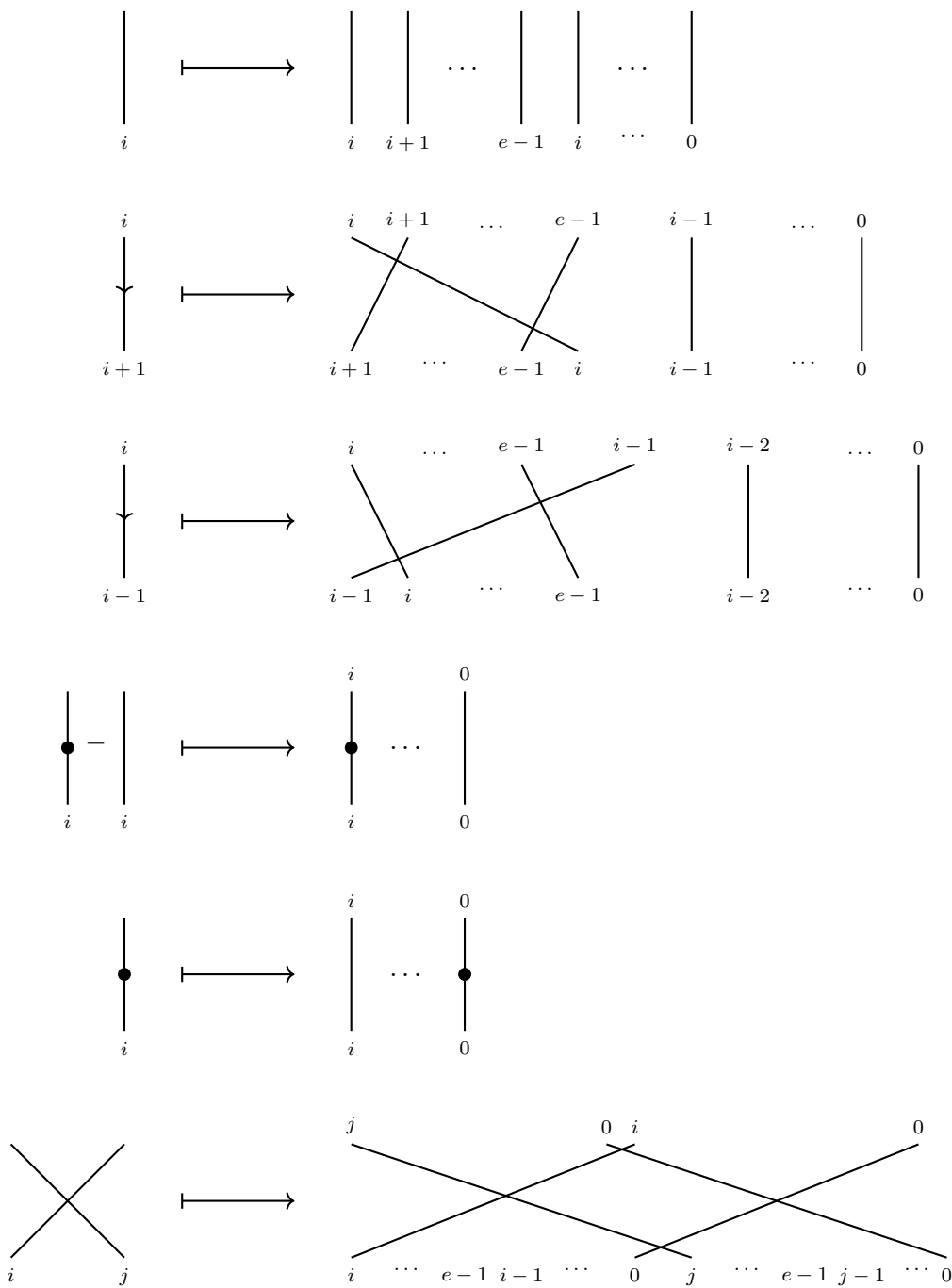

Let $\mathscr{W}_d$ be the wreath product $\mathsf{Z}^{\otimes d}\rtimes \K \Sym_d$ of the zigzag algebra $\mathsf{Z}$ with $\Sym_d$.

\begin{thm}\label{th:KM} The map 
$\Zdaff(\mathsf{ A}_{e-1}^{(1)})\to \oned \R{\PR_{\mathsf{RoCK}}}{\Lambda_0}{d\delta} \oned$  induces an isomorphism $\mathscr{W}_d\cong \oned\R{\PR_{\mathsf{RoCK}}}{\Lambda_0}{d\delta} \oned$.
\end{thm}
\begin{proof}
First, by \cite[Prop. 3.17]{kleshchevAffineZigzag2018}, the algebra $\mathscr{W}_d$ is isomorphic to the quotient of $\Zdaff(\mathsf{ A}_{e-1}^{(1)})$ by the two-sided ideal $I$ generated by $z_1$.

Thus, to complete the proof, we need to show that under the isomorphism \cite[Th. 6.16]{kleshchevAffineZigzag2018}, the kernel of the map $\oned C_{d\delta} \oned\to \oned \R{\PR_{\mathsf{RoCK}}}{\Lambda_0}{d\delta} \oned$ is the same as the ideal $I$.  The element $z_1$ maps to this kernel since it can be written as sweeping the leftmost group of strands of with the labels $\Ba_i$ over the red strand.  That is:
\begin{equation}\label{eq:bigon-sweep}
       \begin{tikzpicture}[baseline=25pt]
    \red{0}{0} \lpull{1}{0}{3}{i} 
\dbcf{2}{0}{\dots}
    \lpull{3}{0}{3}{0}
\dcf{-1}{0}{\dots}
\dtcf{2}{0}{\dots}
\de{4}{0}
\red{5}{0}
\idm{6}{0}{i} 
\dcf{7}{0}{\dots}
\idm{8}{0}{0}
\point{(8,1)}
\end{tikzpicture}
\end{equation}
On the other hand, if we have an element of this kernel, it can be written as a diagram with a group of strands (1) on the far right with labels $r_j,\dots, r_1$ such that $\al_{(r_j,\dots, r_1)}=\beta$ is a root with $\PR_{\mathsf{RoCK}}(\beta)>0$ or (2) on the far left with labels $\ell_1,\dots, \ell_{i}$ such that $\beta=\al_{(\ell_1,\dots, \ell_{i})}$ is a root and $\PR_{\mathsf{RoCK}}(\beta)<0$.

We'll use several times that if we have any diagram and divide the strands into $p$ groups at the bottom (or top), then we can factor our diagram as a sum of diagrams such that:
\begin{enumerate}
\renewcommand{\theenumi}{\roman{enumi}}
    \item the top half has all strands straight and no crossings between strands in the same group at the bottom and 
    \item the bottom half has no crossings between strands in different groups.
\end{enumerate}
We'll call this the {\bf top sorting} factorization, and call its reflection the {\bf bottom sorting} factorization.

For relations of type (1), we just use local relations to move the portion of the diagram left of the red strand to the right.  We can arrange our diagram so that at $y=1/2$ we have  $r_j,\dots, r_1$ at the far right, and divide the strands at $y=1/2$ into those right of the red and those left of the red (considering the red strand as a group of its own).  We can now consider the top sorting factorization of the part of the diagram above $y=1/2$ and bottom sorting factorization of the portion below $y=1/2$.  So our diagram first pulls a bunch of strands to the left of the red  (without reordering them), there is a diagram in the middle with no crossings over the red, and the strands come back to the right as shown below:
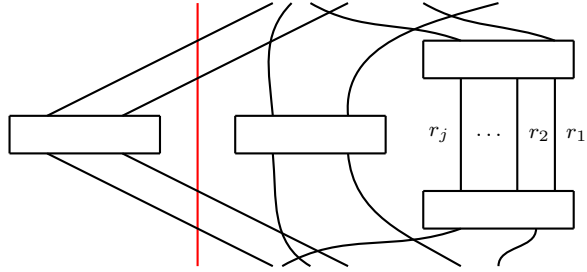
\begin{equation*}
\begin{tikzpicture}[scale=0.5]

\draw[thick,red]  (0,4) -- (0,-3);
\rect{-5}{0}{4}{1}
  \rect{1}{0}{4}{1}

  \rect{6}{2}{4}{1}
  \rect{6}{-2}{4}{1}

\draw[thick] (7,-1) -- (7,2) node [midway, left] {\tiny{$r_j$}};
  \draw[thick] (8.5,-1) -- (8.5,2) node [midway, left] {\tiny{$\dots$}} node [midway, right] {\tiny{$r_2$}};
  \draw[thick] (9.5, -1) -- (9.5, 2) node [midway, right] {\tiny{$r_1$}};
  \draw[thick] ((2,-3) -- (-4, 0);
  \draw[thick] (4,-3) -- (-2,0);
  \draw[thick] (-4,1) -- (2,4);
  \draw[thick] (-2,1) -- (4,4);

  \draw[thick] (2,0) to [out =275, in =150] (3,-3);
   \draw[thick] (4,0) to [out =275, in =150] (7,-3);

   \draw[thick] (4,1) to [out =95, in =195] (8,4);
      \draw[thick] (2,1) to [out =95, in =225] (2.5,4);

      \draw[thick] (7,3) to [out = 160, in = -30] (3,4);

      \draw[thick] (9.5,3) to [out = 160, in = -30] (6,4);

      \draw[thick] (7,-2) to [out = 200, in = 30] (2.25,-3);
      \draw[thick] (9,-2) to [out=270, in=90] (8,-3);
   
\end{tikzpicture}
\end{equation*}

It's easy to check we can use the relations to pull the strands left to the red back to the right, without changing the right half of the middle diagram.  Thus, we obtain a diagram to the right of the red strand, which is in the image of $\oned R_{d\delta} \oned$ which vanishes in $\oned C_{d\delta} \oned$ since $\beta\prec \delta$.

For relations of type (2), we must work a little harder.  Consider the same sorting factorization discussed above, but now applied to the grouping with first group the far left strand with labels $\ell_1,\dots, \ell_{i}$, then the rest of the black strands left of the red strand at $y=1/2$, then the red strand, then the strands right of the red with $y=1/2$.  

Since $\PR_{\mathsf{RoCK}}(\gamma)<0$, at least one of the labels $\ell_1,\dots, \ell_{i}$ must be $0$.  Let $m$ be the first such index, and consider where this element connects to the top of the diagram.    This 0 at the top must be the last letter in one of the words $\Ba^{i}$. Thus, 
 the $e-1$ strands to the left of this terminal at the top have labels different from $0$ and all the strands which meet $y=1$ to the left of this terminal have labels summing to $m\delta$ for some $m$.  

 If we have any crossings between this strand and the $e-1$ strands to the left, then immediately after that crossing, the sum $\beta'$ of the label on the strand crossed over and all others to the right will be $\alpha_j+m\delta$ and so $\PR_{\mathsf{RoCK}}(\beta')>0$ and this diagram is 0 by a relation of type (1), which we have already covered.  Thus, all $e$ of these strands must connect to the group $\ell_1,\dots, \ell_{i}$.  Applying this to all the other strands with label $0$ in this group, we find that we must have the group of $e$ strands that goes with each of them.  If there are any other strands with a label different from $0$, then we would have $\PR_{\mathsf{RoCK}}(\gamma)>0$, which contradicts the assumption that $\PR_{\mathsf{RoCK}}(\gamma)<0$.  Thus, $(\ell_1,\dots, \ell_{i})$ must be a concatenation of words $\Ba_{q_1},\dots, \Ba_{q_{p/e}}$  for some $q_k\in \{1,\dots, e-1\}$.  

 In particular, we have $\al_{(\ell_1,\dots, \ell_e)}=\delta$, and $(\ell_1,\dots, \ell_e)=\Ba_{q_1}$.  Now, we can apply the relations to move all other strands that cross the red line to the right of it, as well as any crossings or dots on these $e$ strands.  Thus, we have a diagram as on the left hand side of \eqref{eq:bigon-sweep} and applying this relation, we see that our diagram is in the 2-sided ideal generated by $z_1$.  This completes the proof.
\end{proof}

To simplify our proof of \cref{th:EK-TD}, let us first state the general result we use.  Let $A$ be a ring and $A^{\Q}=A\otimes_{\Z} \Q$.  Let $j\colon A\to A^{\Q}$ be the obvious map.   
\begin{lem}\label{lem:quotient}
    Let $B=A/J$ be a quotient which is free as a $\Z$-module and $e\in A$ an idempotent such that $A^{\Q} e A^{\Q}=A^{\Q}$.  Then $J=j^{-1}(A^{\Q}ej(J)eA^{\Q}).$  In particular, if $B'=A/J'$ is another quotient which is free over $\Z$ such that $eJe=eJ'e$, then $J=J'$ and $B\cong B'$.  
\end{lem}
\begin{proof}
    Since $A^{\Q} e A^{\Q}=A^{\Q}$, we have that $ e A^{\Q}e$ and $A^{\Q}$ are Morita equivalent.  Applying this to $J^{\Q}=J\otimes_{\Z}\Q$ as a bimodule, it follows that $J^{\Q}=A^{\Q}eJ^{\Q}eA^{\Q}=A^{\Q}ej(J)eA^{\Q}$.  

    Since $B$ is free as a $\Z$-module, we have that $J=j^{-1}(J^{\Q})$, so applying the equality above, we find that $J=j^{-1}(A^{\Q}ej(J)eA^{\Q})$ as desired.  
\end{proof}

\begin{proof}[Proof of \cref{th:EK-TD}]
We'll use \cref{lem:quotient} to show that $C_{\rho,d}\cong e_+\Rk{\PR_{\mathsf{RoCK}}}{\Lambda_0}{d\delta}{\Z}e_+$ and then show that $e_+$ induces a Morita equivalence by comparing the number of simple modules.  We'll often want to change the base ring $\K$, so let $C^{\K}_{d\delta}, C^{\K}_{\rho,d}, \Rk{\PR_{\mathsf{RoCK}}}{\Lambda_0}{d\delta}{\K}$ denote these rings for a fixed choice of $\K$.

{\bf Matching kernels:} Let us apply \cref{lem:quotient} in the case where 
\[A=C^{\Z}_{d\delta}, \;B=C^{\Z}_{\rho,d}, \;B'=e_+\Rk{\PR_{\mathsf{RoCK}}}{\Lambda_0}{d\delta}{\Z}e_+, \;e=\oned.\]   

Let us check the needed hypotheses:
\begin{enumerate}
    \item The algebras $C_{\rho,d}$ and $\R{\PR_{\mathsf{RoCK}}}{\Lambda_0}{d\delta}$ are free as $\Z$-modules by \cite[Lem. 5.18]{EK2} and \cref{cor:R-free}, respectively.
    \item $C_{d\delta}^{\Q}\oned C_{d\delta}^{\Q}=C_{d\delta}^{\Q}$ by \cite[Lem. 6.22]{kleshchevStratifyingKLR2017}.  
    \item Letting $J=\ker(C_{d\delta}^{\Z}\to C_{\rho,d}^{\Z})$ and $J'=\ker(C_{d\delta}\to \R{\PR_{\mathsf{RoCK}}}{\Lambda_0}{d\delta})$, we have $\oned C_{\rho,d}^{\Z}\oned=\mathscr{W}_d^{\Z}$ by \cite[Lem. 8.10]{EK2} and $\oned\Rk{\PR_{\mathsf{RoCK}}}{\Lambda_0}{d\delta}{\Z}\oned=\mathscr{W}_d^{\Z}$ by \cref{th:KM}.  Thus, we have $eJe=eJ'e=\ker(\Zdaff(\mathsf{ A}_{e-1}^{(1)})\to \mathscr{W}_d^{\Z})$.  
\end{enumerate}
Therefore, we can conclude that $J=J'$, and $C_{\rho,d}^{\Z}$ is isomorphic to $e_+\Rk{\PR_{\mathsf{RoCK}}}{\Lambda_0}{d\delta}{\Z}e_+$, the image of $C_{d\delta}$ in $\Rk{\PR_{\mathsf{RoCK}}}{\Lambda_0}{d\delta}{\Z}$.

{\bf Morita equivalence}: Finally, we need to show that $e_+$ generates $\Rk{\PR_{\mathsf{RoCK}}}{\Lambda_0}{d\delta}{\Z}$ as a 2-sided ideal. Assume not.  

Following the notation of \cite{EK2}, let $\ell(A)$ be the number of non-isomorphic simple modules over a finite-dimensional $\K$-algebra $A$ for a field $\K$.   By \cite[Lem. 8.25]{EK2}, for some prime $p$, we have $\ell(\Rk{\PR_{\mathsf{RoCK}}}{\Lambda_0}{d\delta}{\mathbb{\bar F}_p})>\ell(C^{\mathbb{\bar F}_p}_{\rho,d})$.  

Let $\mathscr{P}^{[1,e-1]}(d)$ be the set of $e-1$-multipartitions of $d$.  For any field $\K$, we have an equality $\ell(\Rk{\PR_{\mathsf{RoCK}}}{\Lambda_0}{d\delta}{\K})=\#\mathscr{P}^{[1,e-1]}(d)$;  this is proven for an arbitrary cyclotomic quotient of the form $R^{\Lambda_0}_{w\Ldot{\Lambda_0}(d\delta)}$ in \cite[Th. 5.7]{EK2} and this extends to $\Rk{\PR_{\mathsf{RoCK}}}{\Lambda_0}{d\delta}{\K}$ by \cref{cor:number of simples}.
Thus, we have $\ell(C_{\rho,d}^{\mathbb{\bar F}_p})<\#\mathscr{P}^{[1,e-1]}(d)$. 

 On the other hand, by \cite[Lem. 8.27]{EK2}, the Turner double $D^{\mathbb{\bar F}_p}(n,d)$ appears as the endomorphisms of a projective $C_{\rho,d}^{\mathbb{\bar F}_p}$-module.  As discussed in the proof of \cite[Th. 8.28]{EK2}, we thus have $\ell(D^{\mathbb{\bar F}_p}(n,d))=\#\mathscr{P}^{[1,e-1]}(d)\leq \ell(C_{\rho,d}^{\mathbb{\bar F}_p})$.  
This contradicts the inequality in the previous paragraph, so the Morita equivalence holds.
\end{proof}

\subsection{Adjustment algebras}

If we consider an algebra $A$ which is free as a module over $\Z$ which is not semi-simple over $\Q$ or $\mathbb{F}_p$, then there is a ``usual package'' of results connecting simples over $\mathbb{F}_p$ and over $\Q$ in terms of a decomposition matrix.  To avoid confusion with other decomposition matrices, this can also be called the {\bf adjustment matrix}.  

We can handle these more generally in the context of a PID $\mathbb{K}$ which surjects to a field $\K$ and has fraction field $K$.
In this context, the adjustment matrix is the Cartan matrix of the  {\bf adjustment algebra} $\aA$ of $A$ given by the quotient $\aA= A_{\K}/(A\cap J(A_{K}))$. 

\begin{defn}
	We call a finite-dimensional graded algebra $A$ over a field $\K$ {\bf mixed} if it has non-negative grading and degree zero part ${}^+A_0$ semi-simple.   
\end{defn}
If an algebra is mixed, then its Jacobson radical is precisely the elements of positive degree.  

Thus, we have that:
\begin{lem}\label{adjustment}
	If an algebra $A$ over $\mathbb{K}$ is Morita equivalent to a non-negatively graded algebra ${}^+A$ with ${}^+A_K$ mixed, then its adjustment algebra ${}^+\aA$ is its degree 0 part ${}^+A_0$.
\end{lem}
This is relevant for us because: 
\begin{lem}
For the case of $I=\Z/e\Z$, for any $\PR, \Lambda, \alpha$, the algebra $A=\R{\PR}{\Lambda}{\alpha}$ satisfies the condition that $A_{\Q}$ is Morita equivalent to a mixed algebra for all $\Lambda,\alpha$.  
\end{lem}
It thus follows that $\R{\PR}{\Lambda}{\alpha}$ satisfies the hypotheses of \cref{adjustment} if and only if this Morita equivalence is defined over $\Z$.   This hypothesis is satisfied by RoCK blocks of $\mathbb{F}_p\Sym_n$, which the Morita equivalence with Turner doubles demonstrates.  More generally, it shows that in Turner's question \cite[Question 179]{Turner} ``Can all blocks of $q$-Schur algebras of weight $w$ be $\Z_{\geq 0}$-graded so that the degree 0 part is Morita equivalent to the James adjustment algebra,'' the second clause of the question follows from the first: if they are given a $\Z_{\geq 0}$-grading, the degree 0 part will necessarily be the adjustment algebra.  It is not clear in which cases one can implement a Morita equivalence to a non-negatively graded algebra over the base rings $\K=\Z$ or $\K=\mathbb{F}_p$.  The authors are not optimistic about the existence of such an equivalence in full generality, but have yet to find a counterexample.  
\begin{proof}
By \cref{conj}, we only need to show this for $A=R^{\Lambda}_{\alpha}$.  By \cite[Th. 6.17]{bowmanManyIntegral2022}, we can write $A$ as the idempotent truncation $A=eT^{\vartheta}e$ where $T^{\vartheta}$ is the steadied quotient considered in \cite[\S 4.2]{WebRou}.   Thus, it's enough to check mixedness for $T^{\vartheta}$, which follows from the Koszulity of \cite[Th. 6.2]{WebRou}, since Koszul algebras are necessarily mixed (see also \cite[Cor. 4.7]{WebwKLR}).
\end{proof}

\printbibliography

\end{document}